\documentclass[
fontsize=11pt,
a4paper,
twoside,
english,
headsepline=true,
footsepline=false,
automark,
headings=normal,
parskip=half-,
appendixprefix,
open=any,
cleardoublepage=plain,
abstract=true,
index=totoc,
listof=totoc,
bibliography=totoc
]{scrartcl}
\usepackage[utf8]{inputenc}
\usepackage[english]{babel}
\usepackage[margin=2cm]{geometry}
\usepackage{acronym} 
\usepackage{calc}
\usepackage[table]{xcolor}

\usepackage{graphicx}
\usepackage{epstopdf}
\usepackage{subfig}

\usepackage{amsmath}
\usepackage{amssymb}
\usepackage{pdfpages}
\usepackage{todonotes}
\usepackage{listings}

\setkomafont{disposition}{\normalcolor\bfseries}
\usepackage[plainpages=false, pdfpagelabels, hidelinks, colorlinks=false, pdfpagelayout=TwoPageRight]{hyperref}

\usepackage{scrhack}

\usepackage{transparent} 

\usepackage{amsfonts, amsthm}

\usepackage{newtxtext}
\usepackage{newtxmath}
\usepackage{pifont}
\usepackage[
        activate={true,nocompatibility},
        final,
        tracking=true,
        kerning=true,
        factor=1100,
        stretch=10,
        shrink=10
]{microtype}

\usepackage{datenumber}
\usepackage{array}
\usepackage{xcolor}
\colorlet{shadecolor}{yellow!20}

\usepackage{mathtools}

\DeclarePairedDelimiter\abs{\lvert}{\rvert}
\DeclarePairedDelimiter\norm{\lVert}{\rVert}
\makeatletter
\let\oldabs\abs
\def\abs{\@ifstar{\oldabs}{\oldabs*}}
\let\oldnorm\norm
\def\norm{\@ifstar{\oldnorm}{\oldnorm*}}
\makeatother

\usepackage{graphicx,wrapfig,lipsum}
\newtheorem{theorem}{Theorem}[section]
\newtheorem{lemma}[theorem]{Lemma}
\newtheorem{corollary}[theorem]{Corollary}
\newtheorem{definition}[theorem]{Definition}

\newtheorem{proposition}[theorem]{Proposition}
\newtheorem{remark}[theorem]{Remark}

\newcommand{\sgn}{\operatorname{sgn}}

\newtheorem{example}[theorem]{Example}

\usepackage{accents}

\usepackage{enumitem}
\usepackage{enumitem, hyperref}
\makeatletter
\def\namedlabel#1#2{\begingroup
    #2
    \def\@currentlabel{#2}
    \phantomsection\label{#1}\endgroup
}
\makeatother
\DeclareMathOperator{\supp}{supp}

\DeclareMathOperator*{\argmin}{arg\,min}

\usepackage{bbold}

\usepackage{mathrsfs}

\usepackage{mathtools}

\newcommand{\babla}{\overline{\nabla}}
\newcommand{\beq}{\begin{equation}}
\newcommand{\eeq}{\end{equation}}
\newcommand{\baq}{\begin{equation}\begin{aligned}}
\newcommand{\eaq}{\end{aligned}\end{equation}}
\newcommand{\beqs}{\begin{equation*}}
\newcommand{\eeqs}{\end{equation*}}
\newcommand{\baqs}{\begin{equation*}\begin{aligned}}
\newcommand{\eaqs}{\end{aligned}\end{equation*}}
\newcommand{\Rd}{{\mathbb{R}^d}}
\newcommand{\R}{{\mathbb{R}}}

\newcommand{\RdRd}{{\mathbb{R}^d\times\mathbb{R}^d}}

\newcommand{\dd}{\mathrm{d}}

\newcommand{\Mloc}{\mathcal{M}}

\usepackage[many]{tcolorbox}
\tcbuselibrary{breakable, skins}

\graphicspath{{img/},{img-results/}}

\setkomafont{pageheadfoot}{\small\scshape} 
\newcommand{\ignore}[1]{} 
\setdatetoday

\newcommand{\dckeywords}{gradient flow, nonlocal, Fokker-Planck equation, optimal transport}

\hypersetup{
 pdftitle = {Nonlocal cross-interaction systems on graphs},
 pdfkeywords = {\dckeywords}, 
 pdfcreator = {pdfTeX with Hyperref and Thumbpdf}, 
 pdfproducer = {LaTeX, hyperref, thumbpdf}, 
}

\titlehead{
  \vspace*{0.0cm}
}

\title{\huge Nonlocal cross-interaction systems on graphs:
  \\
  \vspace*{0.0cm} 
  Energy landscape and dynamics
}
\author{Georg Heinze\thanks{corresponding author}$\;\;^,$\thanks{Fakult\"at f\"ur Mathematik, Technische Universit\"at Chemnitz Reichenhainer Stra{\ss}e 41, Chemnitz, Germany. (\{georg.heinze,jfpietschmann\}@math.tu-chemnitz.de)}\,\,,\, Jan-Frederik Pietschmann$^{\dag}$, Markus Schmidtchen\thanks{Institute of Scientific Computing, Technische Universit\"at Dresden, Zellescher Weg 12-14,
01069 Dresden, Germany. (markus.schmidtchen@tu-dresden.de).}}

\begin{document}
\maketitle

\begin{abstract}
We explore the dynamical behavior and energetic properties of a model of two species that interact nonlocally on finite graphs. The authors recently introduced the model in the context of nonquadratic Finslerian gradient flows on generalized graphs featuring nonlinear mobilities. In a continuous and local setting, this class of systems exhibits a wide variety of patterns, including mixing of the two species, partial engulfment, or phase separation. This work showcases how this rich behavior carries over to the graph structure. We present analytical and numerical evidence thereof.
\end{abstract}

\section{Introduction}
In recent years there has been a growing interest in dynamics on graphs in the applied mathematics community. A particular emphasis lies on systems exhibiting a discrete gradient-flow structure, cf. \cite{MassEntropicRicciDiscreteSpaces2017, ErbarMaasGradientFlowPorousMedium2014, MaasGradFlowEntropyFiniteMarkov2011, ChowFokkerPlanckMarkovGraph2012}. 
In \cite{HPJ21}, the authors proposed a system of two nonlocally interacting species on a graph, extending the recent work \cite{EPSS2021}, which can be cast into a Finslerian gradient-flow framework. The result pushes the frontiers of current developments in the direction of coupled multi-species systems with nonlinear mobilities. The dynamics are posed on (possibly infinite) graphs where either species may inhabit each node. Each species is represented by a probability measure on the set of vertices. The set of probability measures on the graph is then equipped with a, not necessarily quadratic, extended quasi-metric. Subsequently, the interaction system is shown to exhibit a gradient-flow structure.
The results in \cite{HPJ21} can be seen as extensions of previous results from several angles:
\begin{itemize}
    \item The on-lattice dynamics of \cite{EPSS2021} are extended to multiple species, nonlinear mobilities and non-quadratic Finslerian structures.
    \item The Finslerian structure of the continuous $p$-Wasserstein spaces unveiled in \cite{Agueh2012_finsler} are carried over to a (discrete) graph setting.
    \item The research on distances induced by concave mobilities on the set of probability measures and the gradient flows in this space \cite{Lisini2010, Dolbeault_2008, Carrillo2010} is translated into a (discrete) graph setting for two species.
\end{itemize}
The driving motivation of the extension to several species is the well-known property  of these interaction systems to exhibit a rich variety of patterns depending on the intraspecific and the interspecific interactions. In \cite{DiFrancesco2013}, a link is drawn between systems of interacting particles  with two species of the form
\begin{align}\label{eq:particle}
    \frac{\dd X_l^{(i)}}{\dd t} = -\frac1N \sum_{l=1}^N \nabla K^{(i1)}(X_l^{(i)}-X_l^{(1)}) - \frac1N \sum_{l=1} \nabla K^{(i2)}(X_l^{(i)}-X_l^{(2)}),
\end{align}
where $l=1,\ldots, N$ numbers consecutively the individuals belonging to species $i=1,2$ and the potentials $K^{(ik)}$ encode the interactions between individuals of the same species (self-interactions) and members of opposing species (cross-interactions), respectively. Exploiting the $2$-Wasserstein gradient flow structure solutions  to the continuous counterpart, i.e., 
\begin{align*}
    \partial_t \rho_t^{(i)} = \nabla\cdot(\rho^{(i)}\nabla (K^{(i1)}\ast \rho^{(1)} + K^{(i2)}\ast\rho^{(2)})),
\end{align*}
are constructed for the nonlocal interaction energy
\begin{align*}
    \mathcal{E}(\rhoup) &=\frac{1}{2}\sum_{i,k=1}^2\int_{\Rd} K^{(ik)}\ast \rho^{(i)}\dd\rho^{(k)}, \qquad \text{ on } \Rd,
\end{align*}
cf. \cite{DiFrancesco2013}. Here, $\rho^{(i)}$, $i=1,2$, denote the probability densities of the two species.  The formation of a rich variety of patterns in this class of systems is widely known, for instance in the context of predator-prey dynamics \cite{di2016nonlocal},  pigment-cell interactions in zebrafish \cite{volkening2020modeling, volkening2018iridophores}, or cell-adhesion and tissue-growth \cite{armstrong2006continuum, carrillo2019population, painter2015nonlocal, murakawa2015continuous}. Additionally, at the continuous level, behavior such as engulfment, mixing, and phase-separation can be found in 
\cite{evers2017equilibria, carrillo2018zoology, berendsen2017cross, 1556-1801_2020_3_307, CARRILLO201875, burger2018sorting, burger2020segregation, carrillo2020measure, fagioli2021multiple}. 
As pointed out earlier, the model proposed in \cite{HPJ21} features nonlinear mobilities which typically arise in contexts of saturation or crowding phenomena, i.e., as the configuration becomes increasingly crowded the mobility decreases \cite{Burger2010cross, painter2015nonlocal,berendsen2017cross, bruna2017cross, bailo2021bound, mason2022macroscopic, simpson2011corrected, burger2016flow}. For convenience, let us recall the system introduced in \cite{HPJ21} which reads
\baq
	\label{eq:NL2CIE_intro}
	\partial_t\rho_t^{(i)}(x) 
	&=  - (\babla\cdot j_t^{(i)})(x),\\
	j_t^{(i)}(x,y) 
	&= m(\rho_t^{(i)}(x),\rho_t^{(i)}(y))\left((v_t^{(i)})_+(x,y)\right)^{\frac{1}{p-1}}\\ &-m(\rho_t^{(i)}(y),\rho_t^{(i)}(x))\left((v_t^{(i)})_-(x,y)\right)^{\frac{1}{p-1}},\\
	v_t^{(i)}(x,y) 
	&= - \babla \left(K^{(i1)}\ast\rho_t^{(1)}+ K^{(i2)}\ast\rho_t^{(2)}\right)(x,y),
\eaq
where $p\in(1,\infty)$ and $x,\,y  \in \R^d$, $t\geq0$. As before, $\rho_t^{(i)}$, with $i=1,2$, represent the population densities corresponding to the two species where the subscript $t$ indicated the time variable. Furthermore, the quantities $\mu$ and $\eta$ encode the structure of the graph, $m$ is a concave mobility of two variables, and the operators $\babla$ and $\babla\cdot$ are discrete analogues of gradient and divergence.

The goal of this paper is to complement the structural results of \cite{HPJ21} with an extensive discussion of the behavior of the interaction systems and to provide an insight into their dynamical and energetic properties. 

The rest of this paper is organized as follows. In Section \ref{sec:framework} we introduce the rigorous setting and the notation used throughout the article. Section \ref{sec:dynamics_energy} is dedicated to a discussion of the dynamical properties of the system and stationary states, as well as their relation to minimizers of the interaction energy. Finally, Section \ref{sec:BehaviorAndIllustrations} is devoted to illustrations of the dynamics using finite graphs and the influence of the graph parameters $\mu$, $\eta$, as well as the exponent $p$, and the scaling parameter $\beta$. Moreover, we observe aggregation and phase separation on the graph depending on the choice of potentials.

\section{Framework and notation}\label{sec:framework}

This section provides the necessary details on the setting and the model. For more details and the rigorous interpretation as a gradient flow, we refer the reader to \cite{EPSS2021,HPJ21}.

We start by introducing the graph setting. The vertices are given by a positive Radon measure $\mu\in \Mloc^+(\Rd)$, called the base measure. In this work, we will consider only finite graphs, i.e. base measures of the form
\baqs 
    \mu = \sum_{\iota=1}^N\mu_\iota\delta_{x_\iota},
\eaqs
for some $N\in\mathbb{N}$, $\mu_1,\ldots,\mu_N\in (0,\infty)$ and $x_1,\ldots,x_N\in\Rd$. We only consider the case $N\ge2$ as otherwise there are no dynamics.

The edges of the graph are defined as the set
\baqs
    G\coloneqq\{(x,y)\in\Rd\times\Rd:x\neq y, \eta(x,y)>0\},
\eaqs
where $\eta:\Rd\times\Rd\to[0,\infty)$ is a nonnegative weight function. We consider undirected graphs, i.e. we assume $\eta$ to be symmetric in the sense that $\eta(x,y) = \eta(y,x)$ for $x,y\in\Rd$. Further, we assume the restriction $\eta|_G$ to be continuous.

For $x_\iota,x_\kappa\in\supp\mu$, we shall use the notation $\eta_{\iota\kappa} = \eta(x_\iota,x_\kappa)$ and similarly for other functions.

We denote by $\mathcal{P}_\mu(\Rd)$ the space of probability measures $\rho$ on $\Rd$, which satisfy $\rho\ll\mu$, and by $\Mloc_\mu(G)$ the space of Radon measures $j$ on $G$ which satisfy $j \ll\mu\otimes\mu$.

Next we introduce nonlocal analogous of gradient and divergence. Given $\varphi:\Rd\to\R$, its nonlocal gradient is $\babla\varphi(x,y)\coloneqq \varphi(y)-\varphi(x)$ while for  $j \in \Mloc(\Rd)$ we define the nonlocal divergence $\babla\cdot j$ by
\baq\label{eq:def_divergence}
	\int_\Rd \varphi(x) \dd \babla\cdot j(x) \coloneqq -\frac{1}{2}\iint_G\babla\varphi(x,y)\eta(x,y)\dd j(x,y),\quad \forall \varphi:\Rd\to\R.
\eaq
We see that if $j$ is symmetric, i.e. $\dd j(x,y) = \dd j(y,x)$ for all $x,y\in\Rd$, then $\babla\cdot j=0$. In the antisymmetric case $\dd j(x,y) = -\dd j(y,x)$ for all $x,y\in\Rd$, $\babla\cdot j$ simplifies to
\baqs
	\int_\Rd \varphi(x) \dd \babla\cdot j(x) = \iint_G\varphi(x)\eta(x,y)\dd j(x,y),\quad \forall \varphi:\Rd\to\R.
\eaqs

Our dynamics is driven by the nonlocal cross-interaction energy given as 
\baq\label{eq:Energy_general}
    \mathcal{E}(\rhoup) &=\frac{1}{2}\sum_{i,k=1}^2\iint_{\Rd\times\Rd} K^{(ik)}(x,y)\dd\rho^{(i)}(x)\dd\rho^{(k)}(y),
\eaq
where $K^{(ik)}:\Rd\times\Rd\to\R$, $i,k=1,2$ are continuous and satisfy satisfy
\begin{align}
	\label{K2}\tag{K1} &\forall x,y\in \Rd\text{ it holds } K^{(ik)}(x,y)=K^{(ik)}(y,x)
\end{align}
To admit a gradient flow w.r.t. the nonlocal operators defined above, even on generalized graphs, the continuity assumption needs to be strengthened (see \cite{HPJ21}) to
\begin{align}
     \label{K3}\tag{K2} &\exists L_K\in(0,\infty) \text{ such that }\; \forall(x,y),(x',y')\in \Rd\times\Rd, \text{ and $i,k=1,2$, we have }\\
 	\notag &\abs{K^{(ik)}(x,y)-K^{(ik)}(x',y')}\leq L_K\big(\abs{(x,y)-(x',y')}\lor\abs{(x,y)-(x',y')}^p\big).
\end{align}
However, for the following considerations, this assumption is not required and will therefore be dropped. In our setting, \eqref{eq:Energy_general} thus becomes
\baq\label{eq:Energy graph}
\mathcal{E}(\rhoup) &=\frac{1}{2}\sum_{i,k=1}^2\sum_{\iota,\kappa=1}^N K^{(ik)}_{\iota\kappa}\rho_\iota^{(i)}\rho_\kappa^{(k)}\mu_\iota\mu_\kappa, \quad \rho\in(\mathcal{P}_\mu(\Rd))^2,
\eaq
where we used $K^{(ik)}_{\iota\kappa}\coloneqq K^{(ik)}(x_\iota,x_\kappa)$ for $i,k=1,2$ and $\iota,\kappa =1,\ldots,N$.

The dynamical system we are interested in (which can rigorously be understood as a gradient flow in the EDE sense \cite{HPJ21}) driven by the nonlocal cross-interaction energy in the nonlocal framework is then given by 
\begin{equation}
    	\label{eq:NL2CIE_general}
	\tag{\eqref{eq:NL2CIE_intro} revisited}
\begin{aligned}
	\partial_t\rho_t^{(i)}
	&=  - (\babla\cdot j_t^{(i)}),\\
	j_t^{(i)}(x,y) &= m\big(\rho_t^{(i)}(x),\rho_t^{(i)}(y)\big) \big[(v_t^{(i)})_+(x,y)\big]^{q-1} \\
	&- m\big(\rho_t^{(i)}(x),\rho_t^{(i)}(y)\big) \big[(v_t^{(i)})_-(y,x)\big]^{q-1},\\
	v_t^{(i)}(x,y) 
	&= - \beta^{(i)}\babla \left(K^{(i1)}\ast\rho_t^{(1)}+ K^{(i2)}\ast\rho_t^{(2)}\right)(x,y),
	\end{aligned}
\end{equation}
where $i=1,2$, $1<q=p/(p-1)<\infty$, $m$ is a suitable mobility function, and where we abused notation by writing $\dd j^{(i)} = j^{(i)}\dd \mu\otimes\mu$.

Note the the equations contain a possibly non-linear mobility function which we define as follows
\begin{definition}
Given two thresholds $R,S\in(0,\infty]$, we call $m\in C([0,R)\times[0,S))$ a \emph{mobility function} if it is concave and strictly positive in $(0,R)\times(0,S)$. For $(r,s)\in [0,R)\times[0,S)$ we denote $m(r,S) = \lim_{s\to S}m(r,s)$ and $m(R,s) = \lim_{r\to R}m(r,s)$. We call such a mobility $m$ \emph{upwind-admissible} if for every $s\ge 0$ we have $m(0,s)=0$.
\end{definition}
In the sequel we require $m$ to satisfy the technical assumption
\begin{equation}\label{A}\tag{A}
    R\land S<\infty \;\text{or}\; \forall r,s\ge 0: \lim_{\lambda\to\infty}\frac{1}{\lambda}m(\lambda r, \lambda s) = 0  \quad\text{or}\quad \forall s\ge 0: m(r,s)=0.
\end{equation}
Given such a mobility $m$ and $\rho^{(i)}\in\mathcal{P}_\mu(\Rd)$, $i=1,2$, we define 
\baqs
    \mathfrak{m}^{(i)}\coloneqq m\left(\frac{\dd\rho^{(i)}}{\dd\mu},\frac{\dd\rho^{(i)}}{\dd\mu}\right).
\eaqs
As with $\eta$, we will often use the notation $\mathfrak{m}^{(i)}_{\iota\kappa}\coloneqq \mathfrak{m}^{(i)}(x_\iota,x_\kappa)$.

\begin{remark}
    The assumption $m(0,s)=0$ is crucial as it ensures the non-negativity of $\rho$. However, note that this assumption excludes, among others, the choice $m\equiv 1$.
    
    The assumption \eqref{A} ensures that any family $(\rhoup_t)_{t\in[0,T]}$ solving $\eqref{eq:NL2CIE_general}$ and satisfying $\supp\rhoup_0\subset\supp\mu$ stays supported on $\supp\mu$ for all times \cite[Proposition 2.35]{HPJ21}. Since we only consider counting measures $\mu$, $\supp\rhoup_t\subset \supp\mu$ is equivalent to $\rhoup_t\ll\mu$. Our main examples $m(r,s)=r^{\theta_1}(1-s)^{\theta_2}$, $\theta_1,\theta_2 \in [0,1]$ all satisfy \eqref{A}.
\end{remark}

Since $\mu$ is a counting measure each curve  $(\rho_t)_{t\in[0,T]}\subset \mathcal{P}_\mu(\Rd)$ can be expressed as 
\baqs
    \rho_t^{(i)} = \sum_{\iota=1}^N\rho^{(i)}_\iota\mu_\iota\delta_{x_\iota},
\eaqs
for suitable functions $\rho^{(i)}_\iota: [0,T]\to[0,(\mu_\iota)^{-1}]$ satisfying $\sum_{\iota=1}^N\rho^{(i)}_\iota\mu_\iota=1$. By slight abuse of notation, we will sometimes write $\rho^{(i)}(x)\coloneqq \rho^{(i)}(\{x\})$. Using these definitions and choosing test functions $\varphi$ such that $\varphi(x_\kappa) = \delta_{\{\kappa=\iota\}}$, in \eqref{eq:def_divergence}, system \eqref{eq:NL2CIE_general} becomes
\baq\label{eq:NL2CIE}
    \frac{\dd}{\dd t}\rho_\iota^{(i)} &= \sum_{\kappa=1}^{N}\Big(\mathfrak{m}^{(i)}_{\kappa\iota}\big[(v^{(i)}_{\iota\kappa})_-\big]^{q-1}-\mathfrak{m}^{(i)}_{\iota\kappa}\big[(v^{(i)}_{\iota\kappa})_+\big]^{q-1}\Big)\eta_{\iota\kappa}\mu_\kappa,
\eaq
for $i=1,2$ and $\iota=1,\ldots,N$, where
\baq\label{eq_def:v^(i)}
v^{(i)}_{\iota\kappa} = -\beta^{(i)}\sum_{k=1}^2\babla\left(K^{(ik)}\ast\rho^{(k)}\right)_{\iota\kappa}= \beta^{(i)}\sum_{k=1}^2\sum_{\lambda=1}^{N}\left(K^{(ik)}_{\iota\lambda}-K^{(ik)}_{\kappa\lambda}\right)\rho_\lambda^{(k)}\mu_\lambda.
\eaq

\section{Dynamics, energy minimizers, and equilibria on finite graphs}\label{sec:dynamics_energy}
In the first part of this section, we study properties of the interaction kernels yielding aggregation and segregation of energy minimizers, respectively. Subsequently, we link the equilibria of the dynamics to critical points of the energy functional, \eqref{eq:Energy_general}. 

\subsection{Energy minimizers}

Our first result provides conditions on the interaction kernels that lead to aggregation or segregation, respectively, of minimizers of the energy functional \eqref{eq:Energy graph}. It is useful to introduce the notation 
\baq\label{eq_def:D^(ik)}
D^{(ik)}(x,y)\coloneqq K^{(ik)}(x,x)-K^{(ik)}(x,y).
\eaq
Note that when $K^{(ik)}$ is an attractive kernel, then it satisfies $D^{(ik)}< 0$ on $\RdRd$, while repulsive kernels $K^{(ik)}$ are such that $D^{(ik)}> 0$ on $\RdRd$.

\begin{theorem}[Conditions for energy optimality of aggregation]\label{thm:aggregation}
Let $\mu = \sum_{\iota=1}^N\mu_\iota\delta_{x_\iota}$ for $x_1,\ldots,x_N \in \Rd$ and let $\rhoup = (\rho^{(1)},\rho^{(2)})$ with $\rho^{(i)} = \sum_{\iota=1}^N \rho_\iota^{(i)}\mu_\iota\delta_{x_\iota}$. Then, the following holds:
\begin{enumerate}[label=\alph*)]
\item If the kernels 
satisfy for every $x,y\in\Rd$ with $x\ne y$
\baq\label{eq_def:attractive kernel thm}
D^{(11)}(x,x)<0, \; D^{(22)}(x,x)<0,
\eaq
and
\baq\label{eq_def:cross weaker than self thm}
D^{(11)}(x,y)D^{(22)}(x,y)> \left(D^{(12)}(x,y)\right)^2,
\eaq
then any energy-minimizing distribution is of the form $\rhoup_* = (\delta_{x_{\iota_1}},\delta_{x_{\iota_2}})$ for two indices $\iota_1, \iota_2 \in \{1,\ldots,N\}$. 
\item For either $i=1$ or $i=2$ and every $x,y\in\Rd$ with $x\ne y$, there holds
\baq\label{eq_def:attractive kernel thm2}
D^{(ii)}(x,y)<0
\eaq
and
\baq\label{eq:K^ii constant on diagonal thm}
K^{(ii)}(x,x) = K^{(ii)}(y,y),
\eaq
then any energy minimizer $\rhoup_*$ is such that species $\rho_*^{(i)}$ is aggregated, i.e., $\rho_*^{(i)}(x_{\iota_i})=\delta_{x_{\iota_i}}$ for an index $\iota_i \in \{1,\ldots,N\}$. 
\item If for both $i=1,2$ we have \eqref{eq_def:attractive kernel thm2} and \eqref{eq:K^ii constant on diagonal thm}, then again $\rhoup_* = (\delta_{x_{\iota_1}},\delta_{x_{\iota_2}})$ for $\iota_1, \iota_2 \in \{1,\ldots,N\}$. If in addition $K^{(12)}$ satisfies for every $x,y\in\Rd$ with $x\ne y$
\baq\label{eq_def:attractive cross-kernel 2}
D^{(12)}(x,y)<0,
\eaq
then $\iota_1=\iota_2$. In this case, any choice $\iota_1=\iota_2\in\{1,\ldots,N\}$ provides the minimal energy. Conversely, if for every $x,y\in\Rd$ with $x\ne y$ we have
\baqs
D^{(12)}(x,y)>0,
\eaqs
then $\iota_1\ne \iota_2$. In this case, there exist at least two energy minimizing distributions, the second being obtained by interchanging $\iota_1$ and $\iota_2$.
\end{enumerate}
\end{theorem}

\begin{remark}\label{rem:K(x,y)=K(y-x)}
Note that \eqref{eq:K^ii constant on diagonal thm} holds in the important special case, where, for $i=1,2$, there exist suitable kernels $\tilde{K}^{(ii)}:\Rd\to\R$ such that
\baqs
K^{(ii)}(x,y)= \tilde{K}^{(ii)}(y-x),
\eaqs
 for any $x,y\in\Rd$.
\end{remark}
\begin{proof}[Proof of Theorem \ref{thm:aggregation}]
We prove each case separately.

\underline{Ad. 1.} We argue by constructing an appropriate competitor. To this end, we fix two indices $\iota_0, \kappa_0\in \{1,\ldots,N\}$ and define a shift $s =(s^{(1)},s^{(2)})^\top$, with $s^{(i)}$ satisfying $-\rho_{\kappa_0}^{(i)}\mu_{\kappa_0}\leq s^{(i)}\leq\rho_{\iota_0}^{(i)}\mu_{\iota_0}$. Then, we obtain a modified probability measure, $\tilde\rho$, by setting
\baqs
\tilde\rho^{(i)}(x_\iota)=\tilde\rho_\iota^{(i)}\mu_\iota =\begin{cases}
\rho_{\kappa_0}^{(i)}\mu_{\kappa_0}+s^{(i)}, &\iota = \kappa_0, \\
\rho_{\iota_0}^{(i)}\mu_{\iota_0}-s^{(i)}, &\iota = \iota_0, \\
\rho_\iota^{(i)}\mu_\iota, &\iota \notin \{\iota_0,\kappa_0\}.
\end{cases}
\eaqs
Then, using the notation $D^{(ik)}_{\iota\kappa}=D^{(ik)}(x_\iota,x_\kappa)$ and defining for $i,k=1,2$
\baq\label{eq_def:a_ik b_i}
a_{ik} &\coloneqq \frac{1}{2}\left(D^{(ik)}_{\kappa_0\iota_0}+D^{(ik)}_{\iota_0\kappa_0}\right),\quad b_{i} \coloneqq \frac{1}{2}\sum_{k=1}^2\sum_{\iota=1}^N\mu_\iota\rho^{(k)}_\iota\left(D^{(ik)}_{\iota\iota_0}-D^{(ik)}_{\iota\kappa_0}\right),
\eaq
we see that the difference of the energies  of $\tilde\rho$ and $\rho$ is given by
\baq\label{eq:shifted energy}
\mathcal{E}(\tilde\rhoup)-\mathcal{E}(\rhoup)&=\sum_{i=1}^2\left(b_is^{(i)}+\frac{1}{2}\sum_{k=1}^2a_{ik}s^{(i)}s^{(k)}\right).
\eaq
Thus, the optimal competitor is the solution to the  constrained quadratic optimization problem
\baq\label{eq:2-point-minimization}
\begin{cases}
\hfill\text{minimize}& f(s)=\frac{1}{2}s^\top A s + b^\top s\\
\hfill\text{s.th.}& -\rho_{\kappa_0}^{(1)}\mu_{\kappa_0}\leq s^{(1)}\leq\rho_{\iota_0}^{(1)}\mu_{\iota_0},\\
\hfill\text{and}& -\rho_{\kappa_0}^{(2)}\mu_{\kappa_0}\leq s^{(2)}\leq\rho_{\iota_0}^{(2)}\mu_{\iota_0},
\end{cases}
\eaq
where $A = (a_{ij})_{i,j=1,2}$ and $b=(b_1,b_2)^\top$ as in \eqref{eq_def:a_ik b_i}.  Since $s=(0,0)^\top$ is admissible, any minimizer of \eqref{eq:2-point-minimization} does not increase the energy.

Now, aggregation occurs whenever the minimizer lies on a corner point of the admissible rectangle. This happens when the problem is strictly concave, i.e.,  whenever $A$ is negative definite. This is the case if
\baq\label{eq:A neg def}
a_{11}+a_{22} < -\sqrt{(a_{11}-a_{22})^2+4a_{12}^2}. 
\eaq
Hence, $a_{ii}<0$ is required for both $i=1,2$ and additionally $a_{11}a_{22}> a_{12}^2$, which holds if and only if \eqref{eq_def:attractive kernel thm}, \eqref{eq_def:cross weaker than self thm} are satisfied.

\underline{Ad. 2.} Requiring \eqref{eq_def:attractive kernel thm2} in conjunction with \eqref{eq:K^ii constant on diagonal thm} for either $i=1$ or $i=2$, 
we obtain $\rho_*^{(i)} = \delta_{x_{\iota_i}}$. Indeed, due to the continuity of $K^{(ik)}$, equations  \eqref{eq:K^ii constant on diagonal thm} and \eqref{eq_def:attractive cross-kernel 2} are equivalent to
\baq\label{eq:self glob min bei (x,x)}
K^{(ii)}(x,x) < K^{(ii)}(y,z), \quad \text{ for all }x,y,z\in\Rd\text{ with }y\ne z.
\eaq
To see this, we fix $\rho$ as before and, without loss of generality, we assume $i=1$. Next, we choose 
\baq\label{eq:cross-minimizer}
    \iota_\ast\in \argmin_{\kappa\in\{1,\ldots,N\}}\left(\sum_{\iota=1}^N K^{(12)}_{\kappa\iota}\rho^{(2)}_\iota\right),
\eaq
which exists, but need not be unique. 
We define a competitor $\tilde \rho$ via $\tilde \rho^{(1)} = \delta_{\tilde \iota}$ and $\tilde \rho^{(2)} =  \rho^{(2)}$. Then, we have
\baq\label{eq:Energy graph 2}
\mathcal{E}(\rhoup) &=\frac{1}{2}\sum_{i,k=1}^2\sum_{\iota,\kappa=1}^N K^{(ik)}_{\iota\kappa}\rho_\iota^{(i)}\rho_\kappa^{(k)}\mu_\iota\mu_\kappa\\
&=\frac{1}{2}\sum_{\iota,\kappa=1}^N \left(K^{(11)}_{\iota\kappa}\rho_\iota^{(1)}\rho_\kappa^{(1)}+K^{(22)}_{\iota\kappa}\rho_\iota^{(2)}\rho_\kappa^{(2)}+2K^{(12)}_{\iota\kappa}\rho_\iota^{(1)}\rho_\kappa^{(2)}\right)\mu_\iota\mu_\kappa\\
&\geq \frac{1}{2}K^{(11)}_{{\tilde \iota}{ \tilde \iota}}+\sum_{\iota=1}^N K^{(12)}_{{\tilde\iota}\iota}\rho_\iota^{(2)}\mu_\iota+\frac{1}{2}\sum_{\iota,\kappa=1}^N K^{(22)}_{\iota\kappa}\rho_\iota^{(2)}\rho_\kappa^{(2)}\mu_\iota\mu_\kappa=\mathcal{E}(\tilde \rho).
\eaq
We achieve equality if and only if $\rho^{(1)}=\delta_\iota$ for some $\iota\in\{1,\ldots,N\}$, which also satisfies \eqref{eq:cross-minimizer}. Hence, any energy minimizer is such that species $1$ is aggregated. The same argument can be made for $i=2$, if $K^{(22)}$ satisfies \eqref{eq:self glob min bei (x,x)}. \\
\underline{Ad. 3.} If \eqref{eq:self glob min bei (x,x)} holds for both $i=1$ and $i=2$, then a calculation similar to \eqref{eq:Energy graph 2} shows that $ \rhoup_* = (\delta_{\iota_1},\delta_{\iota_2})$ 
\baqs
(\iota_1,\iota_2)\in \argmin_{(\iota,\kappa)\in\{1,\ldots,N\}\times\{1,\ldots,N\}} K^{(12)}_{\kappa\iota},
\eaqs
defines an energy-minimzing distribution. 
Finally, to distinguish between the cases $\iota_1 = \iota_2$, and $\iota_1 \ne \iota_2$, we substitute $\rhoup_* = (\delta_{x_{\iota_1}},\delta_{x_{\iota_2}})$ into the definition of $\mathcal{E}(\rhoup)$ to find
\baqs
\mathcal{E}(\rhoup_*) = \frac{1}{4}\left(K^{(11)}_{\iota_1\iota_1}+K^{(22)}_{\iota_2\iota_2}+2K^{(12)}_{\iota_1\iota_2}\right).
\eaqs
Due to \eqref{eq:K^ii constant on diagonal thm}, the first two terms are independent of the choice of $\iota_1$ and $\iota_2$. 
Thus, if $D^{(12)} < 0$, then \eqref{eq:self glob min bei (x,x)} holds and the last term is optimal for $\iota_1 = \iota_2$.
Conversely, when $D^{(12)} < 0$, then, nescessarily it must hold $\iota_1\neq \iota_2$ to achieve optimality in the last term.
\end{proof}

\begin{remark}[Energy optimality of segregation]\label{rem:seggregation_without_aggregation}
A sufficient condition yielding energy optimality of a fully segregated distribution without assuming aggregation of the two species individually, is
\baq\label{eq:condition segregated without aggregation}
\min_{\substack{1\leq \iota,\kappa \leq N\\ \kappa\ne \iota}} \left[K^{(12)}_{\iota\iota} - K^{(12)}_{\iota\kappa}\right] > \frac{1}{2}\sum_{i=1}^2\left(\max_{1\leq \iota,\kappa\leq N}K^{(ii)}_{\kappa\iota}-\min_{1\leq \iota,\kappa\leq N}K^{(ii)}_{\kappa\iota}\right).
\eaq
However, for this to hold for arbitrary graphs requires, using the continuity of the kernels, that the self-interaction kernels $K^{(ii)}$, $i=1,2$ have to be constants.
\end{remark}
\subsection{Dynamics and stationary states}
Let us proceed to studying properties of solutions to the system \eqref{eq:NL2CIE_general}. We are particularly interested in the link between energy minimizers and stationary states of the gradient system. To this end, recall that the velocity $v=(v^{(1)},v^{(2)})$ is defined as
\baqs 
v^{(i)}_{\iota\kappa} = -\beta^{(i)}\sum_{k=1}^2\babla\left(K^{(ik)}\ast\rho^{(k)}\right)_{\iota\kappa}= \beta^{(i)}\sum_{k=1}^2\sum_{\lambda=1}^{N}\left(K^{(ik)}_{\iota\lambda}-K^{(ik)}_{\kappa\lambda}\right)\rho_\lambda^{(k)}\mu_\lambda.
\eaqs
Therefore, by the symmetry of $\eta$, for arbitrary $f:\supp\mu\to \R^n$, $n\in\mathbb{N}$, and for $i=1,2$ we obtain
\baqs
\sum_{\iota=1}^N f_\iota\mu_\iota\frac{\dd}{\dd t}\rho_\iota^{(i)} &= \sum_{\iota,\kappa=1}^{N}f_\iota\Big(\mathfrak{m}^{(i)}_{\kappa\iota}\big[(v^{(i)}_{\iota\kappa})_-\big]^{q-1}-\mathfrak{m}^{(i)}_{\iota\kappa}\big[(v^{(i)}_{\iota\kappa})_+\big]^{q-1}\Big)\eta_{\iota\kappa}\mu_\kappa\mu_\iota\\
&= \sum_{\iota,\kappa=1}^{N}f_\iota\Big(\mathfrak{m}^{(i)}_{\kappa\iota}\big[(v^{(i)}_{\kappa\iota})_+\big]^{q-1}-\mathfrak{m}^{(i)}_{\iota\kappa}\big[(v^{(i)}_{\iota\kappa})_+\big]^{q-1}\Big)\eta_{\iota\kappa}\mu_\kappa\mu_\iota\\
&= \sum_{\iota,\kappa=1}^{N}(f_\kappa-f_\iota)\mathfrak{m}^{(i)}_{\iota\kappa}\big[(v^{(i)}_{\iota\kappa})_+\big]^{q-1}\eta_{\iota\kappa}\mu_\kappa\mu_\iota.
\eaqs
Choosing $f \equiv 1$, we have that $f_\kappa-f_\iota=0$ for all $\kappa,\iota\in\{1,\ldots,N\}$ and thus conservation of mass is guaranteed.

The second question is, whether the center of mass is preserved by the dynamics, which corresponds to the choice $f \equiv x$. The following result provides a counterexample.
\begin{example}[Nonconservation of the center of mass]
Consider the case of a single species with kernel $K$, linear mobility $m(r,s)=r$ and $p=q=2$. We choose $d=1$, $\mu = \delta_{-1}+\delta_{2}$, $\eta\equiv 1$, $\rho = \frac{2}{3}\delta_{-1}+\frac{1}{3}\delta_{2}$, such that the center of mass $x_c = \int_\R x\dd\rho(x)$ is initially located at $0$. Then, calculating its time derivative yields
\baqs
\frac{\dd}{\dd t} x_c&= 2\left(\frac{2}{3}\left(K(2,-1)-K(-1,-1)\right)+\frac{1}{3}\left(K(2,2)-K(-1,2)\right)\right)_-\\
&-\left(\frac{2}{3}\left(K(-1,-1)-K(2,-1)\right)+\frac{1}{3}\left(K(-1,2)-K(2,2)\right)\right)_-,
\eaqs
which vanishes if and only if
\baqs
\frac{2}{3}\left(K(2,-1)-K(-1,-1)\right)+\frac{1}{3}\left(K(2,2)-K(-1,2)\right) = 0.
\eaqs
Since we can choose these terms independently without violating \eqref{K2} or even \eqref{K3}, we find that, in general, $x_c$ is not conserved in time. Indeed, this statement is true even for a single species. The reason the preservation of the center of mass breaks down is the lack of symmetry of the upwind scheme in conjunction with the choice of the underlying graph as well as the structure of the kernel $K$.

However, supposing $K(x,x)=K(y,y)$, for all $x,y\in\Rd$, and upon replacing the upwind relation $j(x,y)= \rho(x)v_+(x,y)-\rho(y)v_-(x,y)$ by a symmetric relation of the form $j(x,y)= \theta(\rho(x),\rho(y)v(x,y)$, with $\theta(x,y)=\theta(y,x)$, the center of mass is conserved in time. The same is true for volume schemes without upwind \cite[Lemma 4.3]{CJLV2016}.

On the other hand, center of mass conservation results are also valid in the particular case of finite volume upwind schemes, i.e. when symmetry is provided by the graph (see, e.g., \cite[Lemma 3.2]{DLV2019}).
\end{example}

Another important choice is $f_\iota =\sum_{k=1}^2\sum_{\lambda=1}^{N}K^{(ik)}_{\iota\lambda}\rho_\lambda^{(k)}\mu_\lambda$, which, if we sum over $i=1,2$, yields the dissipation of the energy $\mathcal{E}(\rhoup)$. We have
\baq\label{eq:energy time derivative}
\frac{\dd}{\dd t}\mathcal{E}(\rhoup)&=\sum_{i,k=1}^2\sum_{\iota,\lambda=1}^N\left(K^{(ik)}_{\iota\lambda}\rho_\lambda^{(k)}\mu_\lambda\right)\mu_\iota\frac{\dd}{\dd t}\rho_\iota^{(i)}\\
&= \sum_{i=1}^2\sum_{\iota,\kappa=1}^{N}\left(\sum_{k=1}^2\sum_{\lambda=1}^{N}\left(K^{(ik)}_{\kappa\lambda}-K^{(ik)}_{\iota\lambda}\right)\rho_\lambda^{(k)}\mu_\lambda\right)\mathfrak{m}^{(i)}_{\iota\kappa}\big[(v^{(i)}_{\iota\kappa})_+\big]^{q-1}\eta_{\iota\kappa}\mu_\kappa\mu_\iota\\
&= -\sum_{i=1}^2\sum_{\iota,\kappa=1}^{N}\frac{1}{\beta^{(i)}}\mathfrak{m}^{(i)}_{\iota\kappa}\big[(v^{(i)}_{\iota\kappa})_+\big]^{q}\eta_{\iota\kappa}\mu_\kappa\mu_\iota \le 0.
\eaq
We expect the (nonpositive) dissipation to vanish at solutions to the stationary equation
\baq\label{eq:rho_nu^(i) stationary}
    0=\sum_{\kappa=1}^{N}\Big(\mathfrak{m}^{(i)}_{\kappa\iota}\big[(v^{(i)}_{\iota\kappa})_-\big]^{q-1}-\mathfrak{m}^{(i)}_{\iota\kappa}\big[(v^{(i)}_{\iota\kappa})_+\big]^{q-1}\Big)\eta_{\iota\kappa}\mu_\kappa,\quad \iota=1,\ldots,N,
\eaq
which, indeed, is equivalent to
\baq\label{eq:rho_nu^(i) time derivative vanishing condition}
    \sum_{\kappa=1}^{N}\mathfrak{m}^{(i)}_{\iota\kappa}\big[(v^{(i)}_{\iota\kappa})_+\big]^{q-1}\eta_{\iota\kappa}\mu_\kappa=\sum_{\kappa=1}^{N}\mathfrak{m}^{(i)}_{\kappa\iota}\big[(v^{(i)}_{\iota\kappa})_-\big]^{q-1}\eta_{\iota\kappa}\mu_\kappa,\quad \iota=1,\ldots,N.
\eaq
From its definition \eqref{eq_def:v^(i)}, we see that the velocity $v^{(i)}(x,y)$ is antisymmetric with respect to its two arguments $x,y\in\Rd$ and also additive, i.e., for $x,y,z\in\Rd$ we have
\baqs
v^{(i)}(x,z) = v^{(i)}(x,y)+v^{(i)}(y,z).
\eaqs
This observation allows for the following characterization:
\begin{lemma}\label{lem:stationary state rho_nu^(i) disintegration}
Let $\rhoup$ and $m$ be as before. Then, \eqref{eq:rho_nu^(i) time derivative vanishing condition} holds for all $\iota=1,\ldots,N$ if and only if for all $\iota,\kappa=1,\ldots,N$ with $\eta_{\iota\kappa}>0$ and both $i=1,2$, we have
\baq\label{eq:stat_state_disintegration}
    \mathfrak{m}^{(i)}_{\iota\kappa}(v^{(i)}_{\iota\kappa})_+=0.
\eaq
\end{lemma}
\begin{proof}
By the antisymmetry of $v^{(i)}$, the implication \eqref{eq:stat_state_disintegration} $\Rightarrow$ \eqref{eq:rho_nu^(i) time derivative vanishing condition} is immediate. 

For the other implication, first observe that $q>1$, implies that the map $s\mapsto s^{q-1}$ is bijective on $[0,\infty)$. Also, all summands in both sums of \eqref{eq:rho_nu^(i) time derivative vanishing condition} are nonnegative, i.e., each sum  vanishes if and only if all of its summands vanish. We now assume \eqref{eq:rho_nu^(i) time derivative vanishing condition} and argue by contradiction. To this end, we additionally assume that there exist two indices $\lambda_0,\lambda_1\in\{1,\ldots,N\}$ s.th. $\mathfrak{m}^{(i)}_{\lambda_1\lambda_0}(v^{(i)}_{\lambda_1\lambda_0})_+>0$. As all summands in the sum are nonnegative, this implies $\sum_{\kappa=1}^{N}\mathfrak{m}^{(i)}_{\lambda_1\kappa}\big[(v^{(i)}_{\lambda_1\kappa})_+\big]^{q-1}\eta_{\lambda_1\kappa}\mu_\kappa>0$. Hence, by \eqref{eq:rho_nu^(i) time derivative vanishing condition}, there exists at least one index $\lambda_2\in\{1,\ldots,N\}$, such that $\mathfrak{m}^{(i)}_{\lambda_2\lambda_1}(v^{(i)}_{\lambda_1\lambda_2})_->0$. Therefore, by the antisymmetry of $v^{(i)}$, we have $\mathfrak{m}^{(i)}_{\lambda_2\lambda_1}(v^{(i)}_{\lambda_2\lambda_1})_+>0$. We can repeat these arguments to obtain a sequence of indices $(\lambda_n)_{n\in\mathbb{N}\cup\{0\}}$. However, since $v^{(i)}$ is additive, we have that for $l,n\in \mathbb{N}_0$, $l>n$ that $v^{(i)}_{\lambda_l\lambda_n}=v^{(i)}_{\lambda_l\lambda_{l-1}}+\ldots+v^{(i)}_{\lambda_{n+1}\lambda_n}>0$. This means that all the indices $\lambda_n$ are different from each other, which contradicts the finiteness of the graph.
\end{proof}
\begin{remark}
Recalling \eqref{eq:energy time derivative} and that $q>1$, we see that Lemma \ref{lem:stationary state rho_nu^(i) disintegration} implies 
\baqs
\frac{\dd}{\dd t}\mathcal{E}(\rhoup)=0 \iff \frac{\dd}{\dd t}\rhoup=0,
\eaqs
i.e., not only does the energy remain constant in time when the dynamics reach a stationary state, but the converse is true, too.
\end{remark}

\section{Mixing, segregation and pattern formation on graphs of variable sizes \& numerical illustrations}\label{sec:BehaviorAndIllustrations}

This section is dedicated to presenting analytical and numerical results illustrating the various phenomena observed for different choices of graphs and interaction kernels. Moreover, we explore the impact of the model's parameters on the dynamical and energetic properties.

We begin by investigating mixing and segregation on small graphs (i.e.,  containing only two, three and four vertices). We proceed with a numerical study of the formation of patterns for attractive-repulsive kernels and explore the influence of the choice of mobility and exponent $p\neq 2$. 

\subsection{Dynamics on a two-point space}\label{Stationary states on a two-point space for kernels with constant diagonals}

To gain a better understanding of the nature of the gradient flow, we analyze the dynamics on a two-point space, $\supp\mu = \{x_1,x_2\}$, which shall act as a starting point for simulations, below. In addition, we assume that for $i,k=1,2$ and all $x,y\in\Rd$ we have 
\baq\label{eq:K^ik const diag}
    K^{(ik)}(x,x)=K^{(ik)}(y,y),
\eaq
and that for $\iota,\kappa=1,2$ the mobility satisfies
\baq\label{eq:m=0_iff_r=0}
    \mathfrak{m}^{(i)}_{\iota\kappa} = 0 \iff \rho^{(i)}_\iota = 0.
\eaq

The starting point of our considerations is an explicit formulation of all possible stationary states assuming $\eta_{\iota\kappa}>0$ (as otherwise, there is no interaction between the vertices). We recall the definition \eqref{eq_def:D^(ik)} of the differences $D^{(ik)}(x,y)$, which reads on the graph as
\baqs
    D^{(ik)}_{\iota\kappa} \coloneqq K^{(ik)}_{\iota\iota}-K^{(ik)}_{\iota\kappa}.
\eaqs
By \eqref{eq:m=0_iff_r=0} and the symmetry \eqref{K2}, we have
\baqs
D^{(ik)}_{12}= K^{(ik)}_{11}-K^{(ik)}_{12}=K^{(ik)}_{22}-K^{(ik)}_{21}=D^{(ik)}_{21}.
\eaqs 
This motivates the notation $D^{(ik)}\coloneqq D^{(ik)}_{12} = D^{(ik)}_{21}$ that is used for the remainder of this subsection and in  Appendix \ref{appendix}.

\noindent
\textbf{Decoupled case.}
Let us first consider the case $D^{(12)}=0$, in which case the equations for the two species are decoupled. Then, it suffices to consider only one species $i\in\{1,2\}$. By Lemma \ref{lem:stationary state rho_nu^(i) disintegration} and \eqref{eq:m=0_iff_r=0} there are two possible ways of obtaining a stationary state:
\begin{enumerate}[label=\alph*)]
    \item\label{case decoupled steady not agg} $v^{(i)}_{\iota\kappa} = 0$,
    \item\label{case decoupled steady agg} $v^{(i)}_{\iota\kappa} > 0$,  $\rho^{(i)}_\iota=0$,
\end{enumerate}
where $\iota,\kappa\in\{1,2\}$ and $\iota\ne \kappa$. The corresponding stationary states are characterized explicitly in Proposition \ref{prop:2-point-energy_decoupled}.
\begin{proposition}[Two-point stationary states - decoupled case]\label{prop:2-point-energy_decoupled}
Let $N=2$, let \eqref{eq:K^ik const diag} hold, and let $D^{(12)}= 0$. Then, for $\iota,\kappa\in\{1,2\}$ with $\iota\ne\kappa$, we have
\begin{enumerate}[label=\alph*)]
\item If $D^{(ii)}=0$, then $v^{(i)}_{\iota\kappa} = 0$ holds for every $\rho^{(i)}(x_1)=1-\rho^{(i)}(x_2)\in[0,1]$.\\
Conversely, if $D^{(ii)}\ne 0$, then $v^{(i)}_{\iota\kappa} = 0$ is achieved only at $\rho^{(i)}(x_1)=\rho^{(i)}(x_2)=\frac{1}{2}$.
\item $D^{(ii)}<0$ is equivalent to being able to satisfy  $v^{(i)}_{\iota\kappa} > 0$ and $\rho^{(i)}(x_\iota)=0$ simultaneously for either $\iota=1$ or $\iota=2$.
\end{enumerate}
\end{proposition}
The proof of this proposition and the remaining results in this subsection is straightforward, yet technical, and can be found in the appendix \ref{appendix}.

Having derived stationary states and conditions for their existence, next we describe their stability.
\begin{proposition}[Dynamics - Decoupled case]\label{prop:2-point-dynamics_decoupled}
Let $N=2$, let \eqref{eq:K^ik const diag} and \eqref{eq:m=0_iff_r=0} hold, and let $D^{(12)}= 0$. Let $\rho^{(i)}_0(x_\iota) = 1-\rho^{(i)}_0(x_\kappa) = \frac{1}{2}+\varepsilon$ with $\varepsilon\in[0,1/2]$ and $\iota, \kappa \in \{1,2\}$, $\iota\ne\kappa$. Then, it holds
\begin{enumerate}[label=\arabic*.]
    \item If $D^{(ii)}>0$, then $\rho^{(i)}_t(x_\iota)\to 1/2$ as $t\to\infty$ for every $\varepsilon\in[0,1/2]$.
    \item If $D^{(ii)}<0$, then
    \begin{itemize}
        \item If $\varepsilon=0$, then $\rho^{(i)}_t(x_\iota) = \rho^{(i)}_t(x_\kappa) = 1/2$, for all $t\ge 0$.
        \item If $\varepsilon\in(0,1/2)$, then $\rho^{(i)}_t(x_\iota) \to 1$ and $\rho^{(i)}_t(x_\kappa) \to 0$, as $t\to\infty$.
    \end{itemize}
\end{enumerate}
\end{proposition}
\begin{remark}\label{rem:two-point_decoupled}
Proposition \ref{prop:2-point-energy_decoupled} shows that, in the decoupled case, species $i$  aggregates on a single vertex if and only if $D^{(ii)} \le 0$. By Proposition \ref{prop:2-point-dynamics_decoupled} we see that these aggregated states are asymptotically stable, whenever $D^{(ii)} > 0$. In particular, this holds for attractive potentials $K^{(ii)}$. Conversely, states in which the mass of species $i$ is distributed equally on the two vertices are always a stationary. However, this type of stationary state is only stable, when $D^{(ii)} < 0$, which is in particular true for repulsive potentials $K^{(ii)}$.
\end{remark}

\noindent
\textbf{Coupled case.}
Next, let us assume $D^{(12)}\ne0$, i.e., there is a coupling. Then, by Lemma \ref{lem:stationary state rho_nu^(i) disintegration}, there are four different conditions leading to stationary states:
\begin{enumerate}[label=\alph*)]
\item\label{case steady 0 agg} $v^{(1)}_{\iota\kappa} = v^{(2)}_{\iota\kappa}=0$,
\item\label{case steady 1 agg} $v^{(i)}_{\iota\kappa} > 0$ and $\rho^{(i)}_\iota=0$ as well as $v^{(k)}_{\iota\kappa}= 0$,
\item\label{case steady 2 agg 1 node} $v^{(1)}_{\iota\kappa}>0$ and $v^{(2)}_{\iota\kappa}> 0$  as well as $\rho^{(1)}_\iota=\rho^{(2)}_\iota=0$,
\item\label{case steady 2 agg 2 nodes} $v^{(1)}_{\iota\kappa}> 0$ and $v^{(2)}_{\iota\kappa}< 0$  as well as $\rho^{(1)}_\iota=0$ and $\rho^{(2)}_\iota=1$,
\end{enumerate}
for $i,k,\iota,\kappa\in\{1,2\}$ with $i\ne k$ and $\iota\ne \kappa$. In the simplified setting, these stationary states can also be characterized explicitly as shown in the following proposition.
\begin{proposition}[Stationary states - Coupled case]\label{prop:2-point-energy_coupled}
Let $N=2$, let \eqref{eq:K^ik const diag} and \eqref{eq:m=0_iff_r=0} hold, and let $D^{(12)}\ne 0$. Then, for $\iota,\kappa\in\{1,2\}$ with $\iota\ne\kappa$, we have
\begin{enumerate}[label=\alph*)]
\item A probability measure satisfying \ref{case steady 0 agg} exists. \\ If $D^{(11)}D^{(22)}\ne\left(D^{(12)}\right)^2$, then it is unique and given by
\baqs
\rho_\mathrm{a}^{(1)}(x_1) &= \rho_\mathrm{a}^{(1)}(x_2) = \frac{1}{2},\;\text{ and } \;\rho_\mathrm{a}^{(2)}(x_1) = \rho_\mathrm{a}^{(2)}(x_2) = \frac{1}{2}.
\eaqs
If $D^{(11)}D^{(22)}=\left(D^{(12)}\right)^2$, then the probability measures
\baq\label{eq:rho_a,r}
\rho_{\mathrm{a}_r}^{(1)}(x_1) &= \frac{1}{2}(1+r), &&\rho_{\mathrm{a}_r}^{(1)}(x_2) = \frac{1}{2}(1-r),\\
\rho_{\mathrm{a}_r}^{(2)}(x_1) &= \frac{1}{2}\left(1-\frac{D^{(12)}}{D^{(22)}}r\right), &&\rho_{\mathrm{a}_r}^{(2)}(x_2) = \frac{1}{2}\left(1+\frac{D^{(12)}}{D^{(22)}}r\right),
\eaq
where $r\in\left[-1\lor -\abs{\frac{D^{(22)}}{D^{(12)}}},1\land \abs{\frac{D^{(22)}}{D^{(12)}}}\right]$ all satisfy \ref{case steady 0 agg}.
\item A probability measure satisfying \ref{case steady 1 agg} exists if and only if for one $i,k \in\{1,2\}$, $k\neq i$ we have
\baq\label{eq:1 agg existenz 1}
\abs{D^{(12)}}\le \abs{D^{(kk)}},
\eaq
and
\baq\label{eq:1 agg existenz 2}
D^{(ii)}< \frac{\left(D^{(12)}\right)^2}{D^{(kk)}}.
\eaq
Then, it is given by
\baqs
&\rho_{\mathrm{b}_i}^{(i)}(x_\iota) = 1, &&\rho_{\mathrm{b}_i}^{(i)}(x_\kappa) = 0,\\
&\rho_{\mathrm{b}_i}^{(k)}(x_\iota) = \frac{1}{2}\left(1-\frac{D^{(12)}}{D^{(kk)}}\right), &&\rho_{\mathrm{b}_i}^{(k)}(x_\kappa) = \left(1+\frac{D^{(12)}}{D^{(kk)}}\right).
\eaqs
In this case, there exists a second stationary state of the same form since $\iota$ and $\kappa$ are interchangeable.
\item A probability measure satisfying \ref{case steady 2 agg 1 node} exists if and only if
\baq\label{eq:2 agg 1 node existenz}
D^{(11)}&< -D^{(12)},\quad
D^{(22)}< -D^{(12)}.
\eaq
Then, it is given by
\baqs
&\rho_\mathrm{c}^{(1)}(x_\iota) = 1, \rho_\mathrm{c}^{(1)}(x_\kappa) = 0\text{ and } \rho_\mathrm{c}^{(2)}(x_\iota) = 1, \rho_\mathrm{c}^{(2)}(x_\kappa) = 0.
\eaqs
In this case, since $\iota$ and $\kappa$ can be swapped, a second stationary state of the same type exists.
\item A probability measure satisfying \ref{case steady 2 agg 2 nodes} exists if and only if
\baq\label{eq:2 agg 2 nodes existenz}
D^{(11)}&< D^{(12)},\quad
D^{(22)}< D^{(12)}.
\eaq
Then, it is given by
\baqs
\rho_\mathrm{d}^{(1)}(x_\iota) = 1, \rho_\mathrm{d}^{(1)}(x_\kappa) = 0\text{ and } \rho_\mathrm{d}^{(2)}(x_\iota) = 0, \rho_\mathrm{d}^{(2)}(x_\kappa) = 1.
\eaqs
Again, there is a second stationary state of the same form as $\iota$ and $\kappa$ can be swapped.
\end{enumerate}
\end{proposition}
Having established these existence conditions, we make the following observations:
\begin{corollary}\label{cor:only_a<=>pos_semidef}
Let $N=2$, let \eqref{eq:K^ik const diag} and \eqref{eq:m=0_iff_r=0} hold, and let $D^{(12)}\ne 0$. Then, $\rho_\mathrm{a}$ is the only stationary state  if and only if
\baq\label{eq:a_only}
D^{(11)}>0, D^{(22)} > 0,\quad D^{(11)}D^{(22)}> \left(D^{(12)}\right)^2.
\eaq
All existing stationary states are of the form $\rho_{\mathrm{a}_r}$ as in \eqref{eq:rho_a,r} if and only if
\baq\label{eq:a,r_only}
D^{(11)}>0, D^{(22)} > 0,\quad D^{(11)}D^{(22)}\geq \left(D^{(12)}\right)^2.
\eaq
\end{corollary}
\begin{remark}
Note that \eqref{eq:a_only} turns the corresponding minimizing problem \eqref{eq:2-point-minimization} into a strictly convex problem with unique minimizer $\rho_\mathrm{a}$, while \eqref{eq:a,r_only} turns \eqref{eq:2-point-minimization} into a convex problem with minimizers $\rho_{\mathrm{a}_r}$.
\end{remark}
\begin{corollary}\label{cor:c+d<=>b^2}
Let $N=2$, let \eqref{eq:K^ik const diag} and \eqref{eq:m=0_iff_r=0} hold, and let $D^{(12)}\ne 0$. If $\dot\rhoup_\mathrm{c}=\dot\rhoup_\mathrm{d}=0$, then we also have $\dot\rhoup_{\mathrm{b}_1}=\dot\rhoup_{\mathrm{b}_2}=0$. 

Conversely, $\dot\rhoup_{\mathrm{b}_1}=\dot\rhoup_{\mathrm{b}_2}=0$ implies that either $\dot\rho_\mathrm{c}=\dot\rhoup_\mathrm{d}=0$ holds or that we have $\dot\rhoup_\alpha=0$ and $D^{(ii)}=-\abs{D^{(12)}}$ for one $i\in\{1,2\}$ and one $\alpha\in\{\mathrm{c},\mathrm{d}\}$.
\end{corollary}
Having derived stationary states and conditions for their existence, we now want to learn more about their stability. Let us study the stability of the different stationary states in the case $D^{(12)}\ne0$ and $D^{(11)}D^{(22)}\ne \left(D^{(12)}\right)^2$. In this case, up to symmetry there are only four distinct stationary states. As seen in Proposition \ref{prop:2-point-energy_coupled}, the stationary distributions are given by
\begin{enumerate}[label=\alph*)]
\item\label{case stationary state 0 agg} $\rho_\mathrm{a}^{(1)}(x_\iota) = \frac{1}{2}$, $\rho_\mathrm{a}^{(2)}(x_\iota) = \frac{1}{2}$,
\item\label{case stationary state 1 agg} $\rho_{\mathrm{b}_i}^{(i)}(x_\iota) = 1$, $\rho_{\mathrm{b}_i}^{(k)}(x_\iota) = \frac{1}{2}\left(1-\frac{D^{(12)}}{D^{(kk)}}\right)$,
\item\label{case stationary state 2 agg 1 node} $\rho_\mathrm{c}^{(1)}(x_\iota)  = 1$, $\rho_\mathrm{c}^{(2)}(x_\iota) = 1$,
\item\label{case stationary state 2 agg 2 nodes} $\rho_\mathrm{d}^{(1)}(x_\iota)  = 1$, $\rho_\mathrm{d}^{(2)}(x_\iota) = 0$.
\end{enumerate}
Here it is important to keep in mind that since $\iota=1$ and $\iota=2$ are both possible, $\rhoup_{\mathrm{b}_1}$, $\rhoup_{\mathrm{b}_2}$, $\rhoup_{\mathrm{c}}$ and $\rhoup_{\mathrm{d}}$ each characterize a pair of two distinct stationary states. 
\begin{proposition}[Stability I]\label{prop:2-point-dynamics_coupled}
Let $N=2$, let \eqref{eq:K^ik const diag} and \eqref{eq:m=0_iff_r=0} hold, and assume $D^{(12)}\ne 0$ as well as $D^{(11)}D^{(22)}\ne\left(D^{(12)}\right)^2$. Then, it holds
\begin{enumerate}[label=\alph*)]
    \item $\rhoup_\mathrm{a}$ is asymptotically stable if and only if it is the only stationary state.
    \item If $\rhoup_{\mathrm{b}_i}$ with $i\in\{1,2\}$ are stationary, then they are both asymptotically stable if and only if $\rhoup_{\mathrm{b}_{k}}$ with $k=3-i$, $\rhoup_{\mathrm{c}}$ and $\rhoup_{\mathrm{d}}$ are not stationary.
    \item If $\rhoup_\mathrm{c}$ are stationary, then they are asymptotically stable.
    \item If $\rhoup_\mathrm{d}$ are stationary, then they are asymptotically stable.
\end{enumerate}
\end{proposition}
\begin{proof}
This is a consequence of the Remarks \ref{rem:stability_a}, \ref{rem:stability_c,d} and \ref{rem:stability_b}, which give a more nuanced insight into the stability of the respective stationary states.
\end{proof}
It remains to consider the case $D^{(11)}D^{(22)}= (D^{(12)})^2$ where the following holds true.
\begin{proposition}[Stability II]\label{prop:2-point-dynamics_coupled_II}
Let $N=2$, let \eqref{eq:K^ik const diag} and \eqref{eq:m=0_iff_r=0} hold, and assume $D^{(12)}\ne 0$ as well as $D^{(11)}D^{(22)}=\left(D^{(12)}\right)^2$. Then, the stationary states $\rhoup_{\mathrm{a}_r}$ are stable, but not asymptotically stable.
\begin{enumerate}[label=\arabic*.]
    \item If $D^{(ii)}>0$ for $i=1,2$, then no other stationary state exists.
    \item If $D^{(ii)}<0$ for $i=1,2$ and $D^{(12)}<0$, then the stationary states $\rhoup_{\mathrm{c}}$ are asymptotically stable.
    \item If $D^{(ii)}<0$ for $i=1,2$ and $D^{(12)}>0$, then the stationary states $\rhoup_{\mathrm{d}}$ are asymptotically stable.
\end{enumerate}
\end{proposition}
\begin{proof}
This is a consequence of the Remarks \ref{rem:stability_a}, \ref{rem:stability_c,d} and \ref{rem:stability_b}, which give a more nuanced insight into the stability of the respective stationary states.
\end{proof}

\begin{remark}\label{rem:two-point_decoupled}
Summarizing the results for the coupled case, we see the following:
\begin{itemize}
    \item If both self-interactions are attractive (which implies $D^{(ii)}<0$ for $i=1,2$), then Proposition \ref{prop:2-point-energy_coupled} shows that the stationary states with the lowest energy are such that the species are both aggregated. If $D^{(12)}<0$, the lowest energy is achieved for aggregation of both species at the same vertex. If, however, $D^{(12)}>0$, the lowest energy is achieved for aggregation at opposite vertices. This is true both in the regime of dominant self-interactions, i.e., when $D^{(11)}D^{(22)}\ge(D^{(12)})^2$, and in the regime of dominant cross-interactions, i.e. when $D^{(11)}D^{(22)}<(D^{(12)})^2$.
    \item In the regime of dominant self-interactions $D^{(11)}D^{(22)}\ge(D^{(12)})^2$ Proposition \ref{prop:2-point-energy_coupled} implies that the sign of the self-interactions, which is the same for the two species, dictates whether the stationary states with the lowest energy aggregate on a vertex ($D^{(ii)}<0$ for $i=1,2$) or are distributed uniformly, ($D^{(ii)}>0$ for $i=1,2$).
    \item In the regime of dominant cross-interaction ( $D^{(11)}D^{(22)}<(D^{(12)})^2$) Proposition \ref{prop:2-point-energy_coupled} yields that the cross-interactions dictate that the optimal stationary states are such that at least one species is aggregated -- even for two repulsive self-interactions.
    \item Regarding the stability, Proposition \ref{prop:2-point-dynamics_coupled} shows that when $D^{(11)}D^{(22)}\ne\left(D^{(12)}\right)^2$, then optimal stationary states are asymptotically stable. Further, it shows that when aggregation of both species at the same vertex as well as aggregation of the species at opposite vertices are stationary, then all these states are asymptotically stable. Finally, we see that when $D^{(11)}D^{(22)}\ne\left(D^{(12)}\right)^2$, then stationary states that are not energetically optimal or such that both species are aggregated, are unstable.
    \item Proposition \ref{prop:2-point-energy_coupled} shows that when the self- and cross-interactions are balanced ($D^{(11)}D^{(22)}\ne\left(D^{(12)}\right)^2$), then there exists a continuum of stationary states. By Proposition and \ref{prop:2-point-dynamics_coupled_II}, these are stable but not asymptotically stable when the self-interactions are repulsive. Conversely, when the self-interactions are attractive, then they are unstable, and there exist two asymptotically stable stationary states, where both species are aggregated either at the same vertex ($D^{(12)}<0$) or opposite vertices ($D^{(12)}>0$).
\end{itemize}
\end{remark}

\begin{example}[Volume filling mobility]
    The volume filling mobility $\mathfrak{m}^{(i)}_{\iota\kappa}=\rho^{(i)}_\iota(1-\rho^{(i)}_\kappa)$, is covered by the considerations of this section, if $\mu_1,\mu_2\le 1$. On the other hand, if $\mu_1,\mu_2>2$, then the set of admissible distributions for the dynamics is reduced. For example, in the case $\mu_1=\mu_2=2$, the only admissible distribution is $\rho^{(1)}(x_1)=\rho^{(1)}(x_2)=\rho^{(2)}(x_1)=\rho^{(2)}(x_2)=1/2$. If $1<\mu_1,\mu_2<2$, then there are more admissible distributions, but the species still cannot aggregate in a single vertex, even if the interaction kernels are attractive, i.e., if $D^{(ik)}<0$ for $i,k=1,2$. If $(\mu_1)^{-1}+(\mu_2)^{-1}<1$, then no admissible distribution exists.
    
    If $\mu_1,\mu_2 = 1$, then 
    \baqs
        \mathfrak{m}^{(i)}_{\iota\kappa}=\begin{cases}
        \rho^{(i)}_\iota(1-\rho^{(i)}_\iota), &\text{if}\quad \rho^{(i)}_\iota\in[0,1],\\
        -\infty, &\text{if}\quad \rho^{(i)}_\iota\notin[0,1],
        \end{cases}
    \eaqs
    is also an admissible, though rather unnatural mobility map. With this mobility $\rho_{\mathrm{c}}$ as well as $\rho_{\mathrm{d}}$ are always stationary states, since the mobility vanishes at these points, even if the self interaction kernels are repulsive.
\end{example}

\begin{figure}[ht!]
    \captionsetup[subfloat]{margin=10pt,format=hang,singlelinecheck=false}
    \centering
    \subfloat[$0>D^{(12)}>D^{(ii)}$]{
    \includegraphics[height=3.85cm]{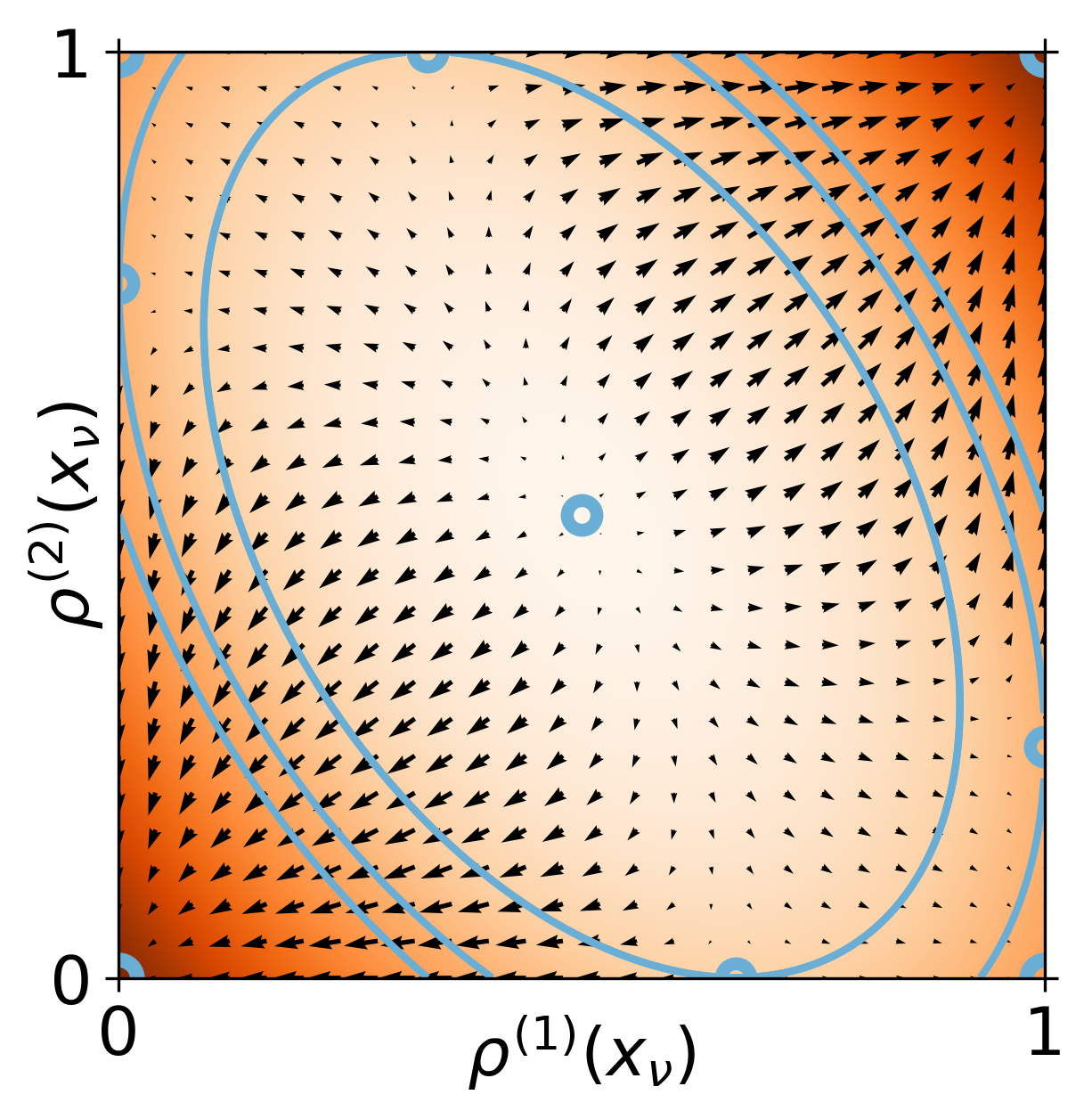}
    }
    \subfloat[$0=D^{(12)}>D^{(ii)}$]{
    \includegraphics[height=3.85cm]{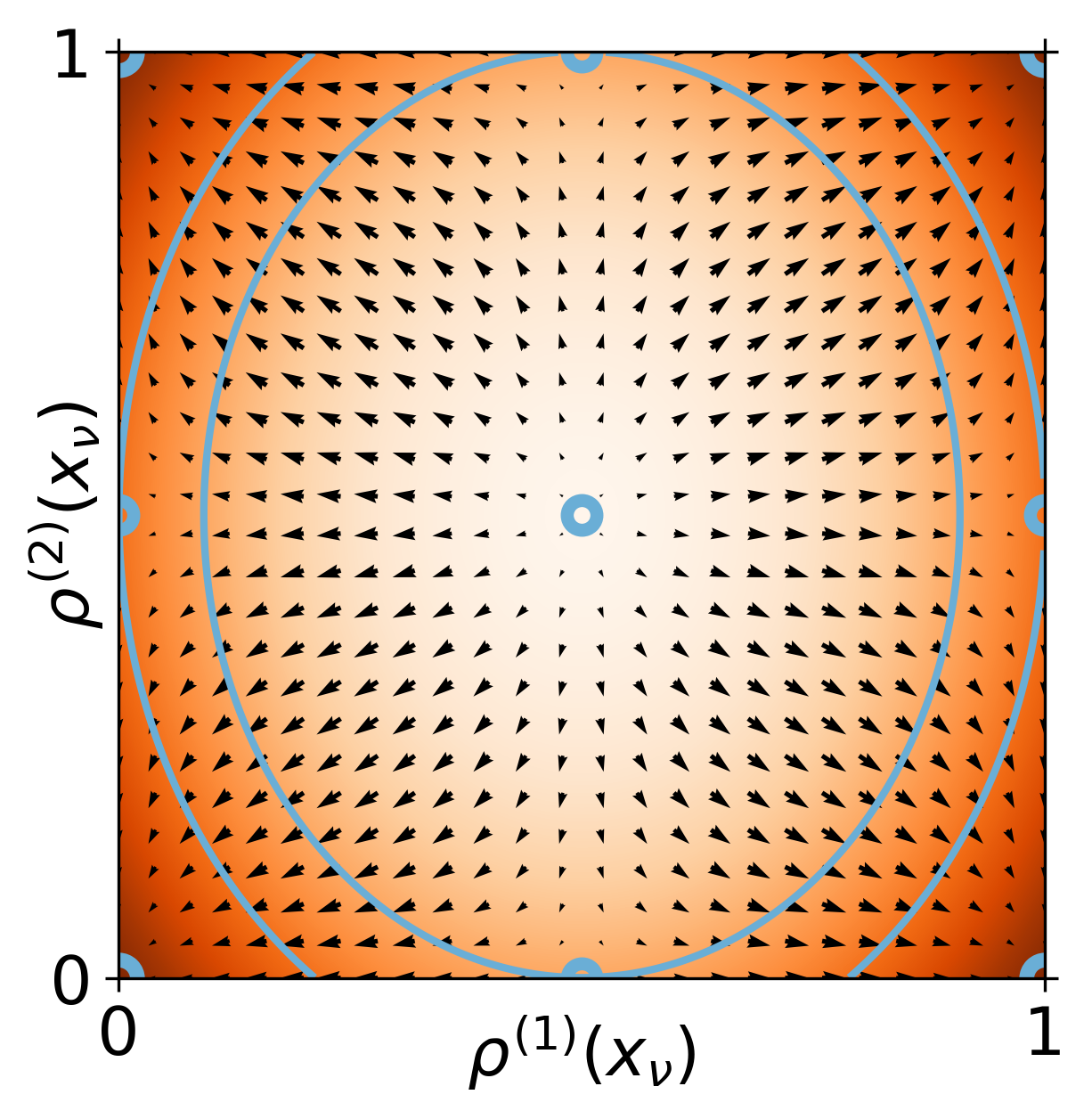}
    }
    \subfloat[$0<D^{(12)}<-D^{(ii)}\qquad\qquad$]{
    \includegraphics[height=3.85cm]{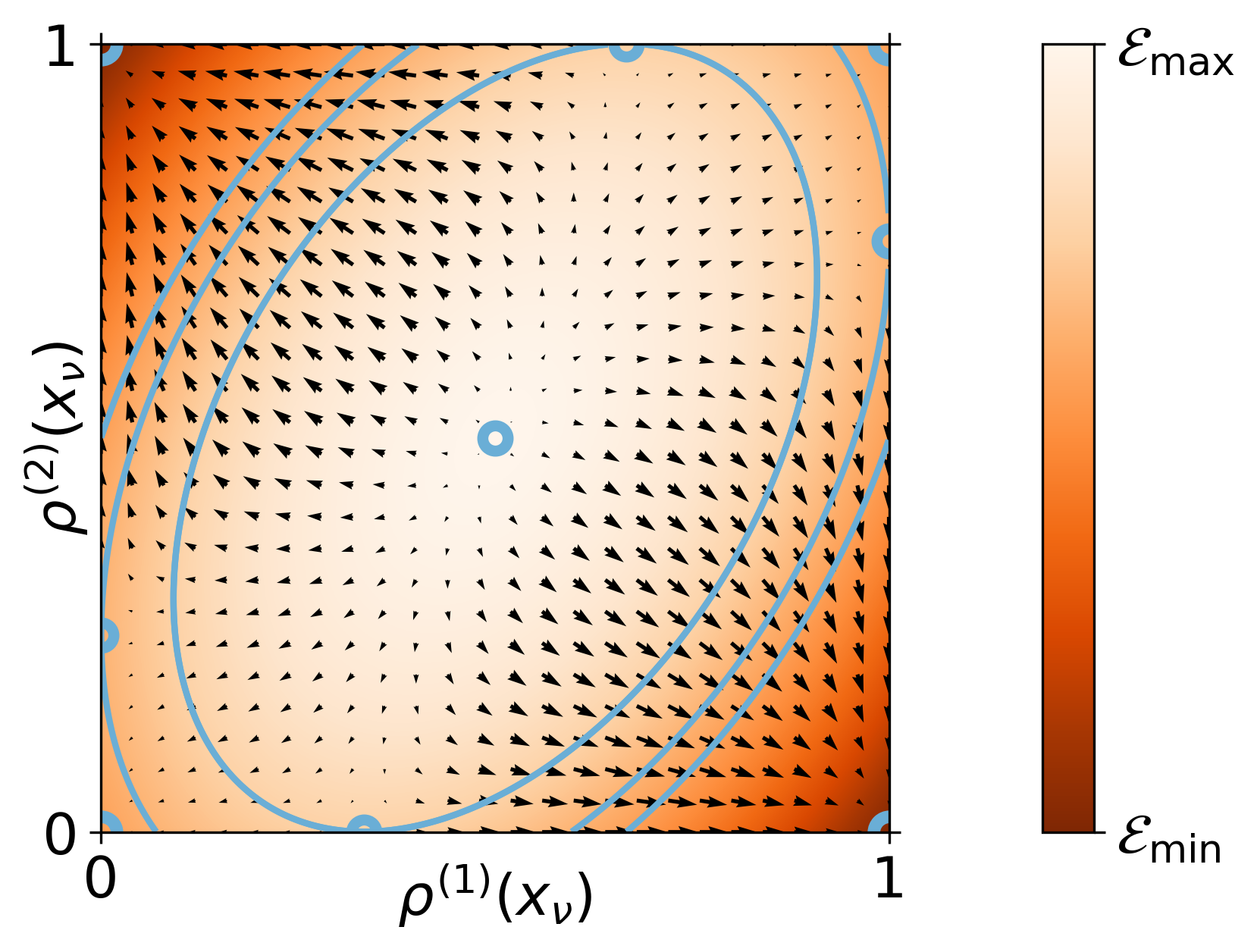}
    }
    \caption{Fixing $\beta^{(1)}=\beta^{(2)}$, $p=2$ and $\mathfrak{m}^{(i)}_{\iota\kappa}=\rho^{(i)}_\iota$, we illustrate the energy landscapes (heat map), stationary states (light blue circles), their corresponding energy level sets (light blue lines), and dynamics (quivers) for different cross-interaction, while staying in the regime $D^{(11)},D^{(22)}<0$, $D^{(11)}D^{(22)}>\left(D^{(12)}\right)^2$. }
    \label{fig:two-point_self:att_vary:cross}
\end{figure}

\begin{figure}[ht!]
    \captionsetup[subfloat]{margin=10pt,format=hang,singlelinecheck=false}
    \centering
    \subfloat[$D^{(12)}>D^{(ii)}>0$]{
    \includegraphics[height=3.85cm]{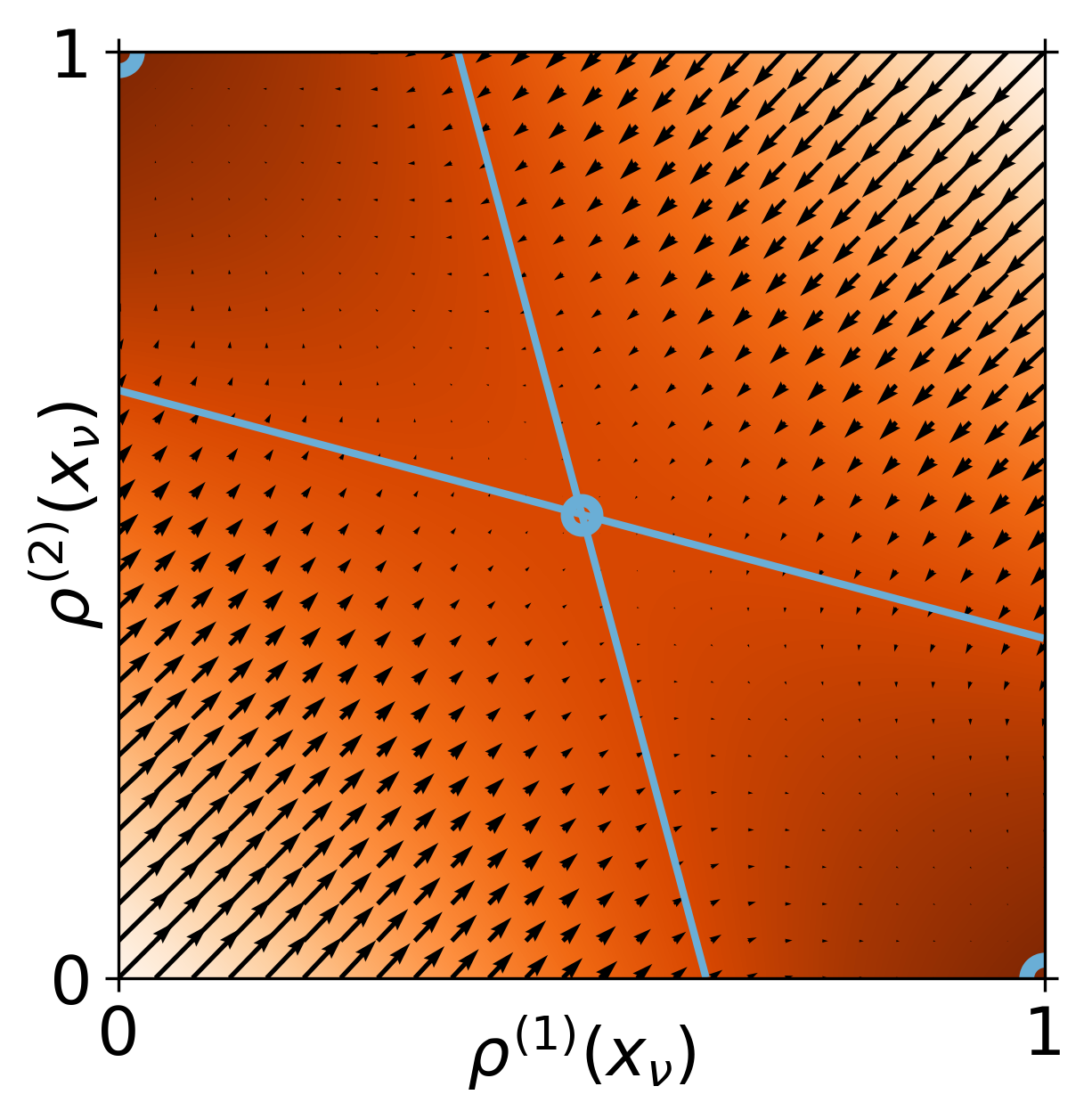}
    }
    \subfloat[$D^{(12)}=D^{(ii)}>0$]{
    \includegraphics[height=3.85cm]{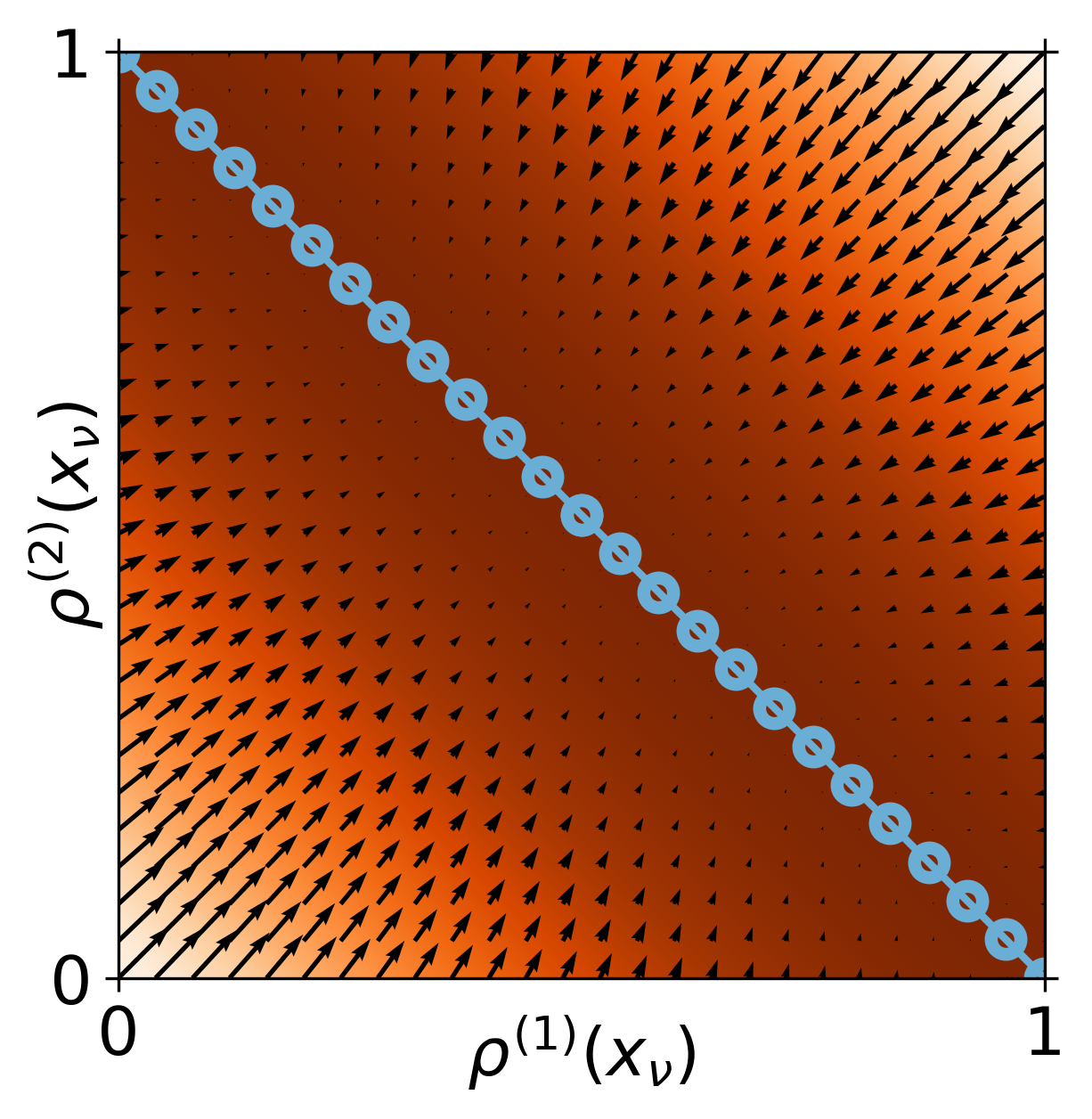}
    }
    \subfloat[$D^{(ii)}>D^{(12)}>0\qquad\qquad$]{
    \includegraphics[height=3.85cm]{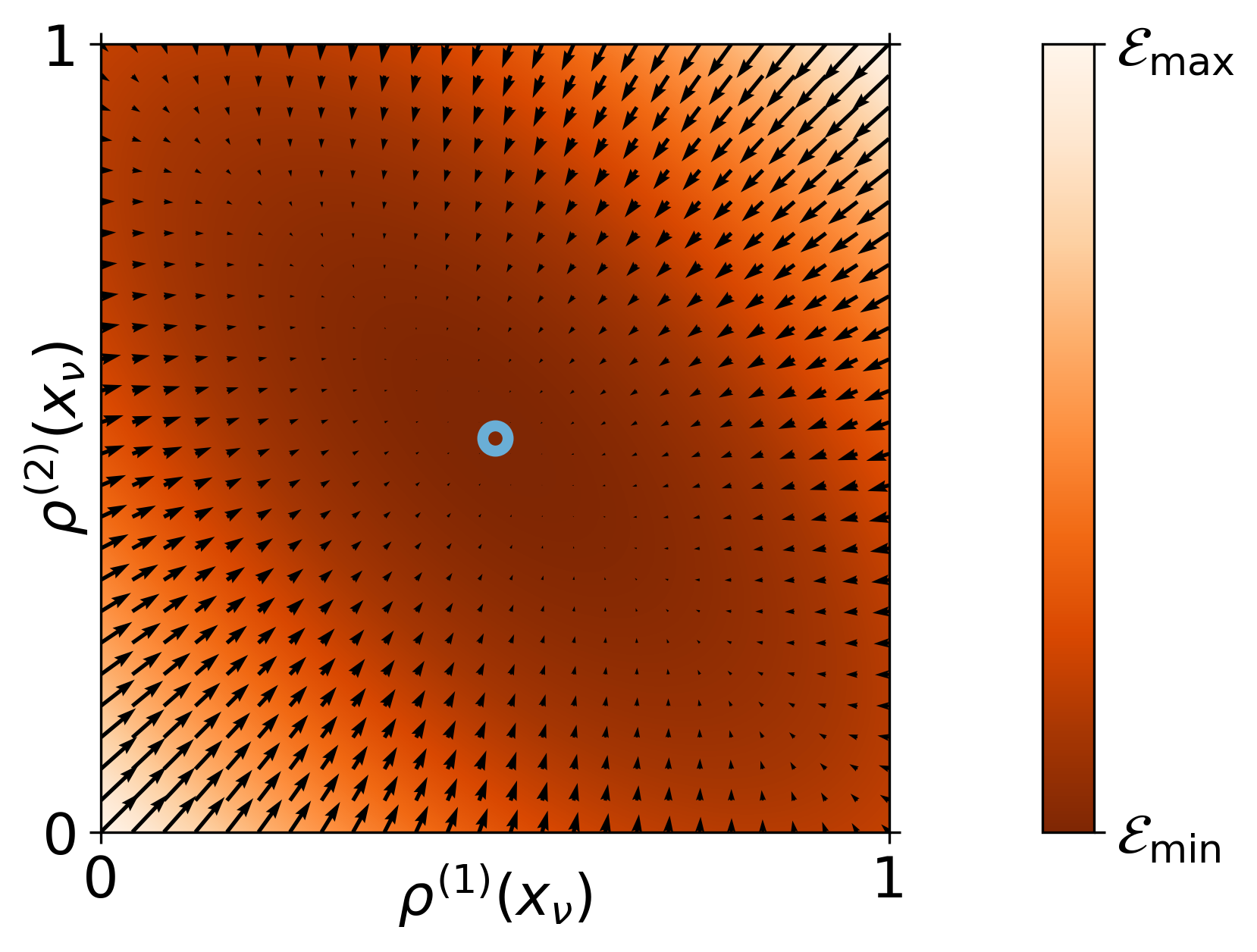}
    }\\
    \subfloat[$D^{(12)}>-D^{(ii)}>0$]{
    \includegraphics[height=3.85cm]{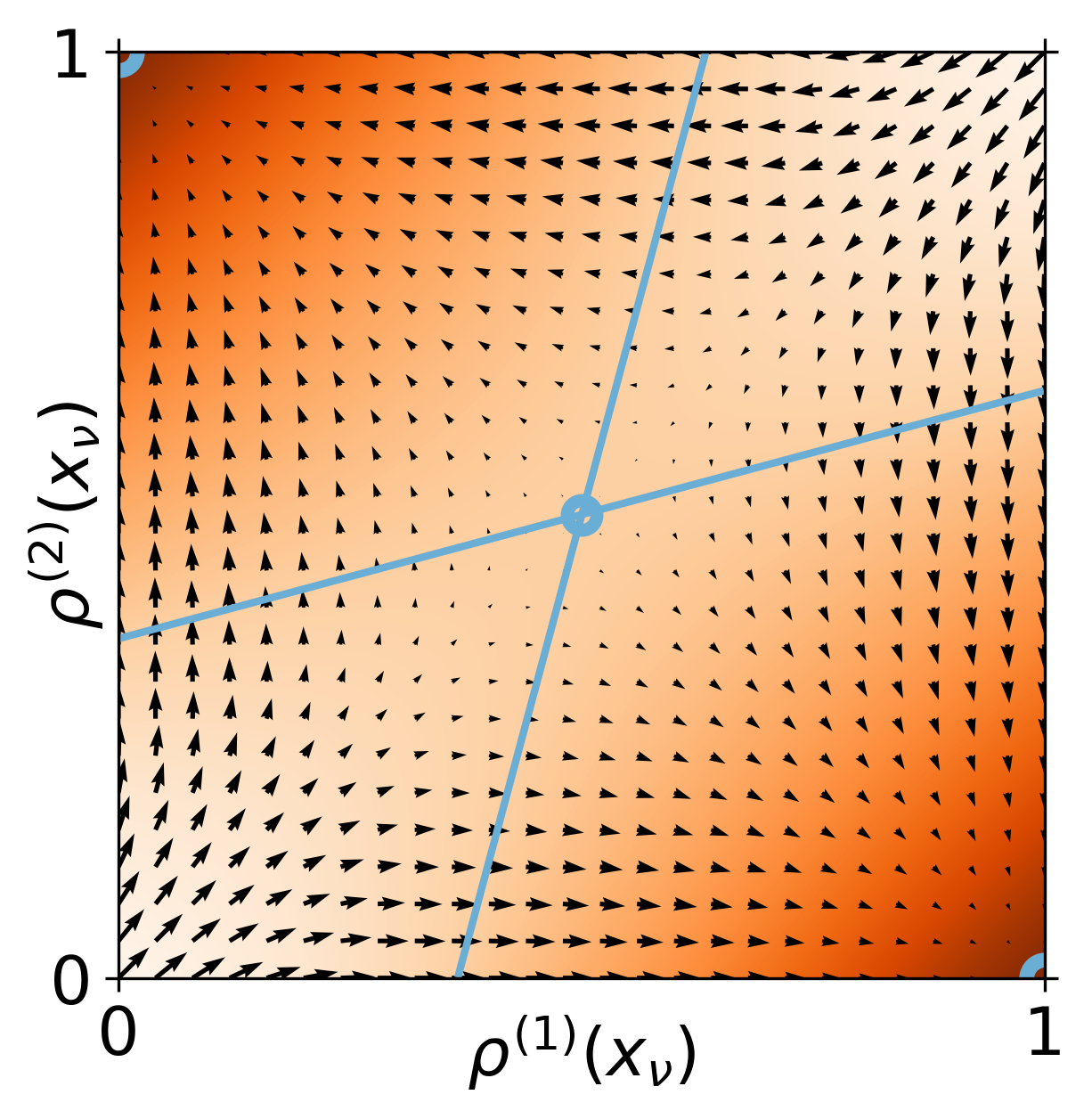}
    }
    \subfloat[$D^{(12)}=-D^{(ii)}>0$]{
    \includegraphics[height=3.85cm]{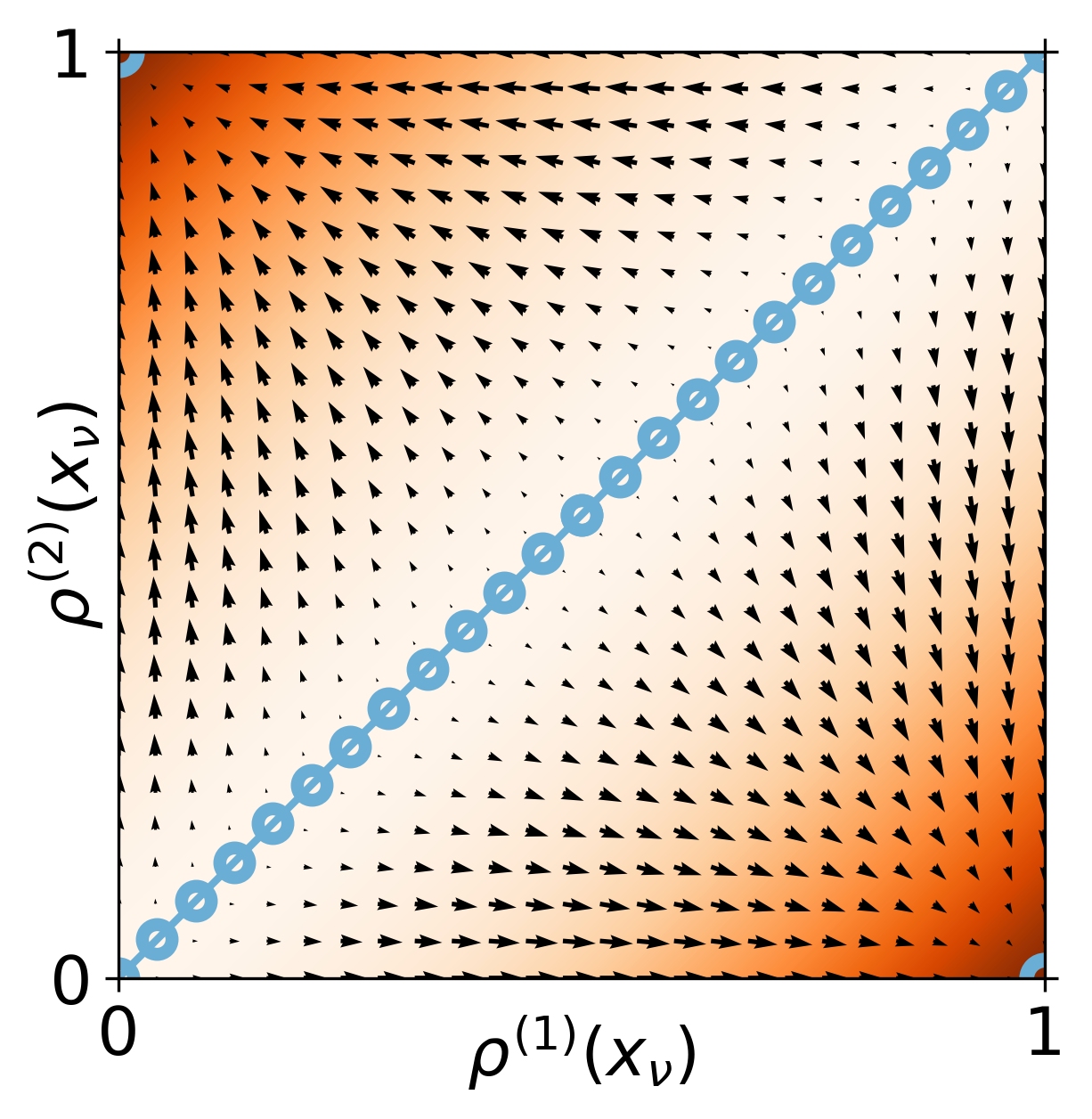}
    }
    \subfloat[$D^{(12)}>-D^{(ii)}>0\qquad\qquad$]{
    \includegraphics[height=3.85cm]{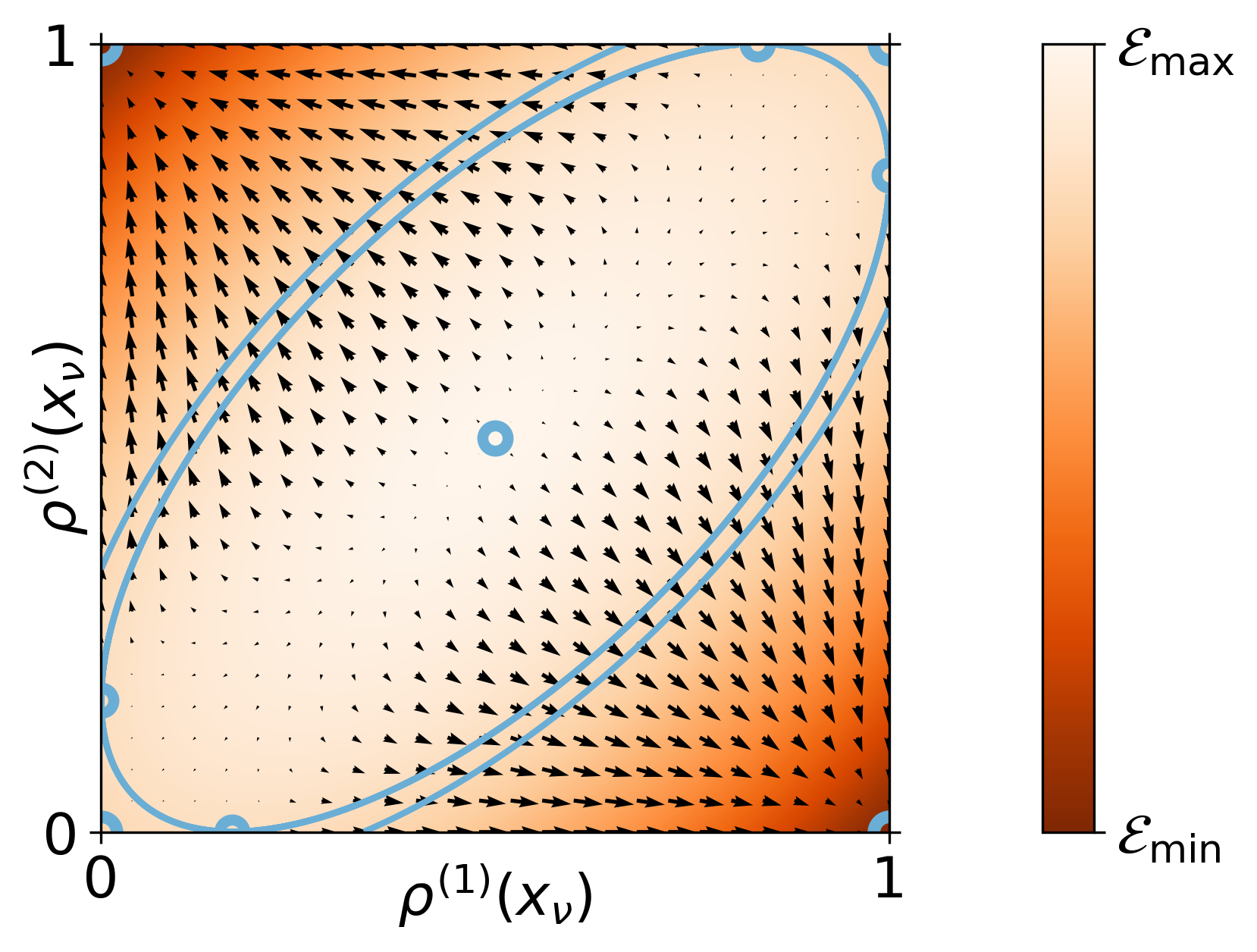}
    }
    \caption{Fixing $\beta^{(1)}=\beta^{(2)}$, $p=2$ and $\mathfrak{m}^{(i)}_{\iota\kappa}=\rho^{(i)}_\iota$, we display the energy landscapes (heat map), stationary states (light blue circles), their corresponding energy level sets (light blue lines), and dynamics (quivers) in the different regimes $D^{(11)}D^{(22)} \lesseqqgtr  \left(D^{(12)}\right)^2$ for repulsive (top row) and attractive (bottom row) self-interactions.}
    \label{fig:two-point_vary:det}
\end{figure}

\begin{figure}[ht!]
    \captionsetup[subfloat]{margin=10pt,format=hang,singlelinecheck=false}
    \centering
    \subfloat[$D^{(11)}=-D^{(22)}=-2$\newline
    $D^{(12)}=1$]{
    \includegraphics[height=3.85cm]{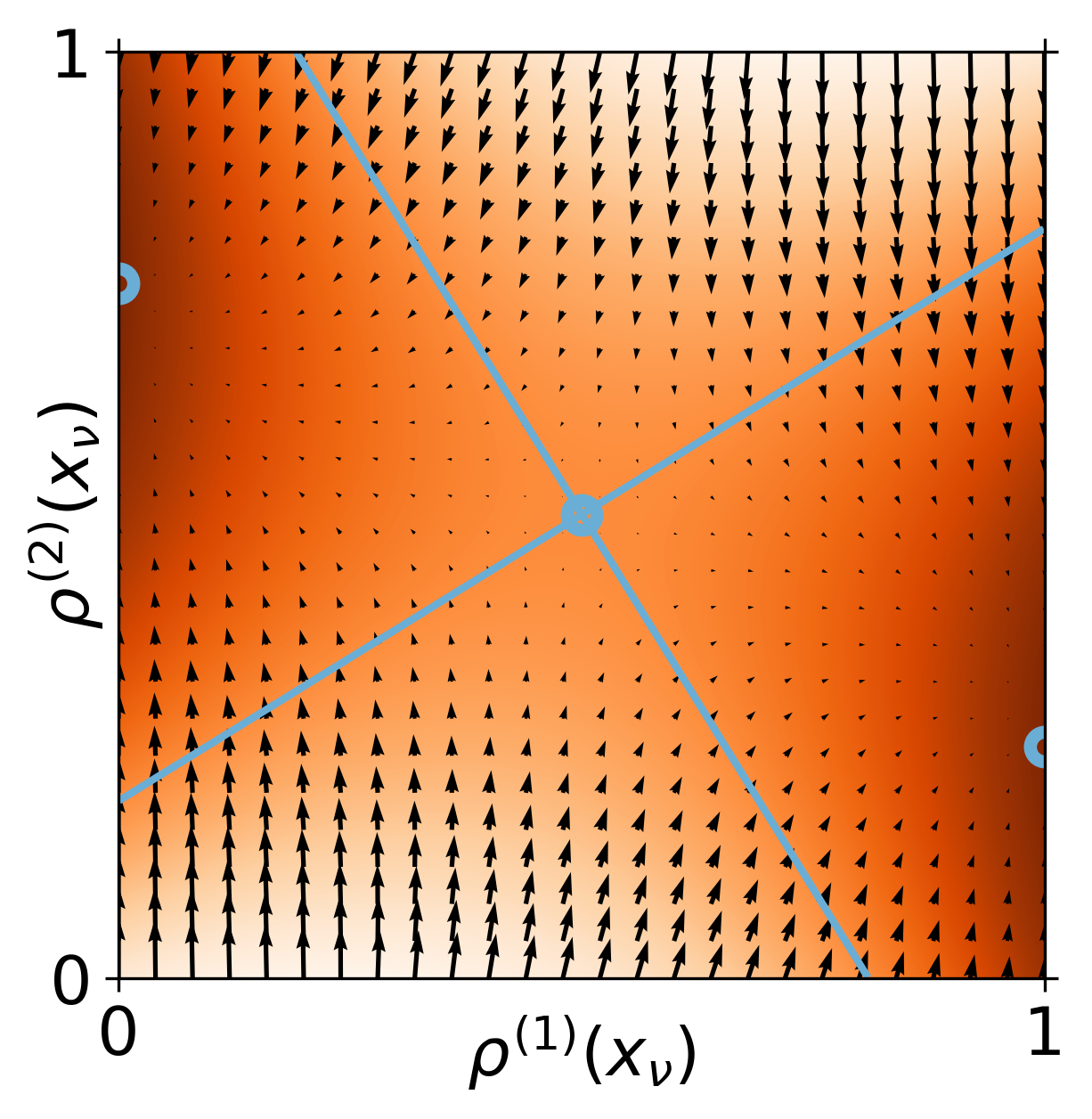}
    }
    \subfloat[$D^{(11)}=-D^{(22)}=0$\newline
    $D^{(12)}=1$]{
    \includegraphics[height=3.85cm]{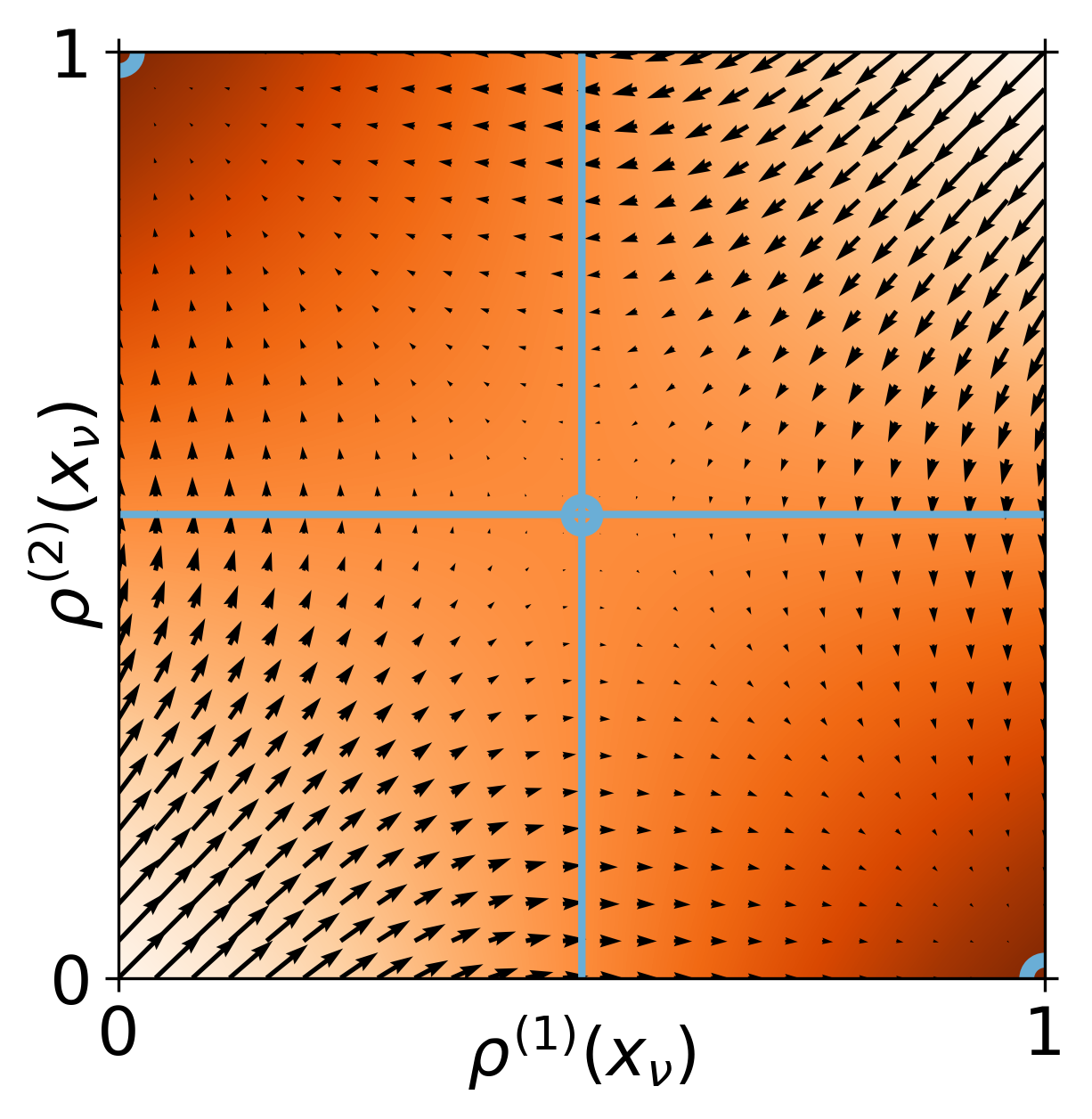}
    }
    \subfloat[$D^{(11)}=-D^{(22)}=2$\newline
    $D^{(12)}=1$]{
    \includegraphics[height=3.85cm]{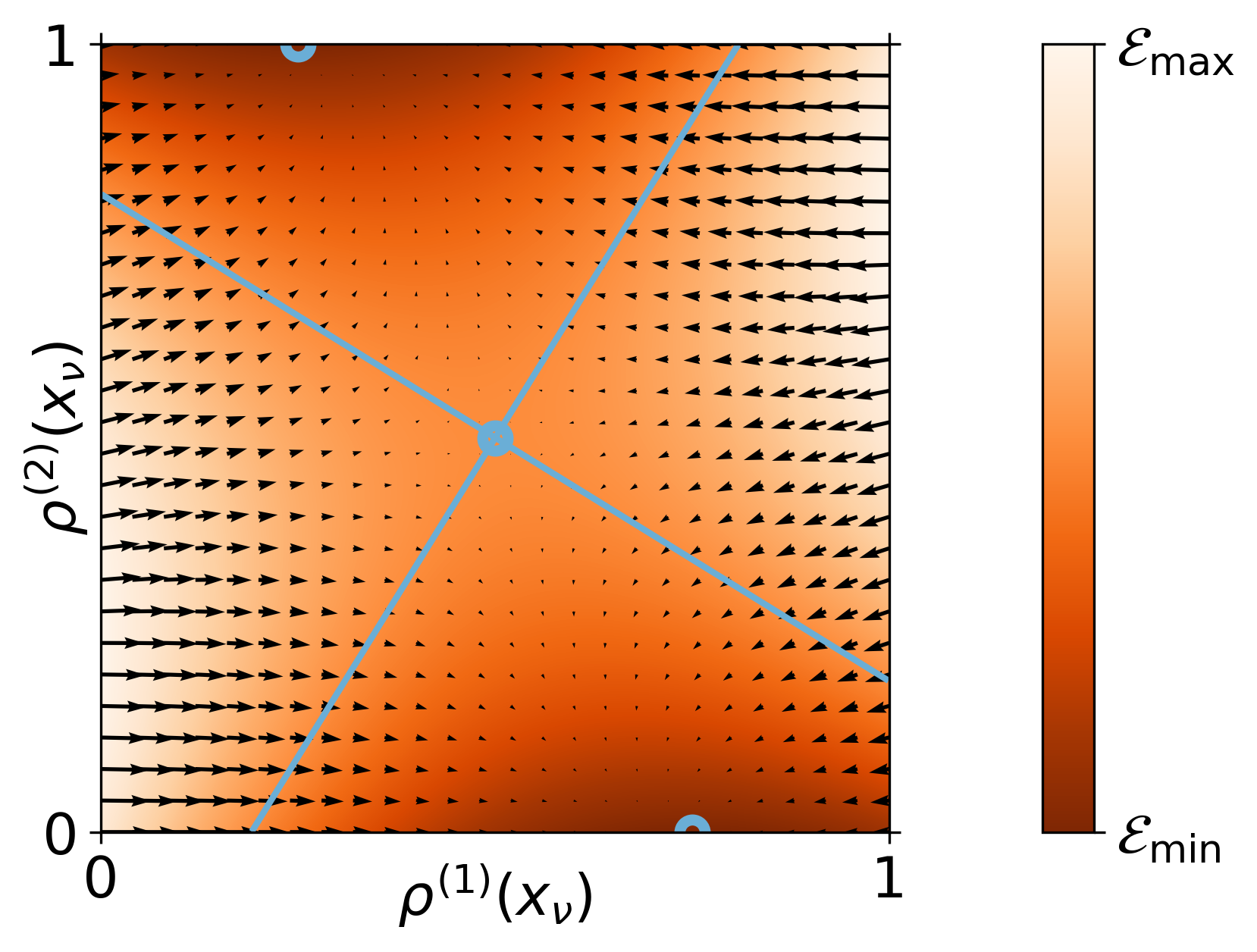}
    }
    \caption{Fixing $\beta^{(1)}=\beta^{(2)}$, $p=2$ and $\mathfrak{m}^{(i)}_{\iota\kappa}=\rho^{(i)}_\iota$, we illustrate the energy landscapes (heat map), stationary states (light blue circles), their corresponding energy level sets (light blue lines), and dynamics (quivers), for self-interactions of opposite signs and varying strength, while the cross-interactions are fixed and repulsive} \label{fig:two-point_rotate}
\end{figure}

\begin{figure}[ht!]
    \captionsetup[subfloat]{margin=10pt,format=hang,singlelinecheck=false}
    \centering
    \subfloat[$D^{(ii)}=1$, $D^{(12)}=0$, $\beta^{(1)}>\beta^{(2)}$]{
    \includegraphics[height=3.85cm]{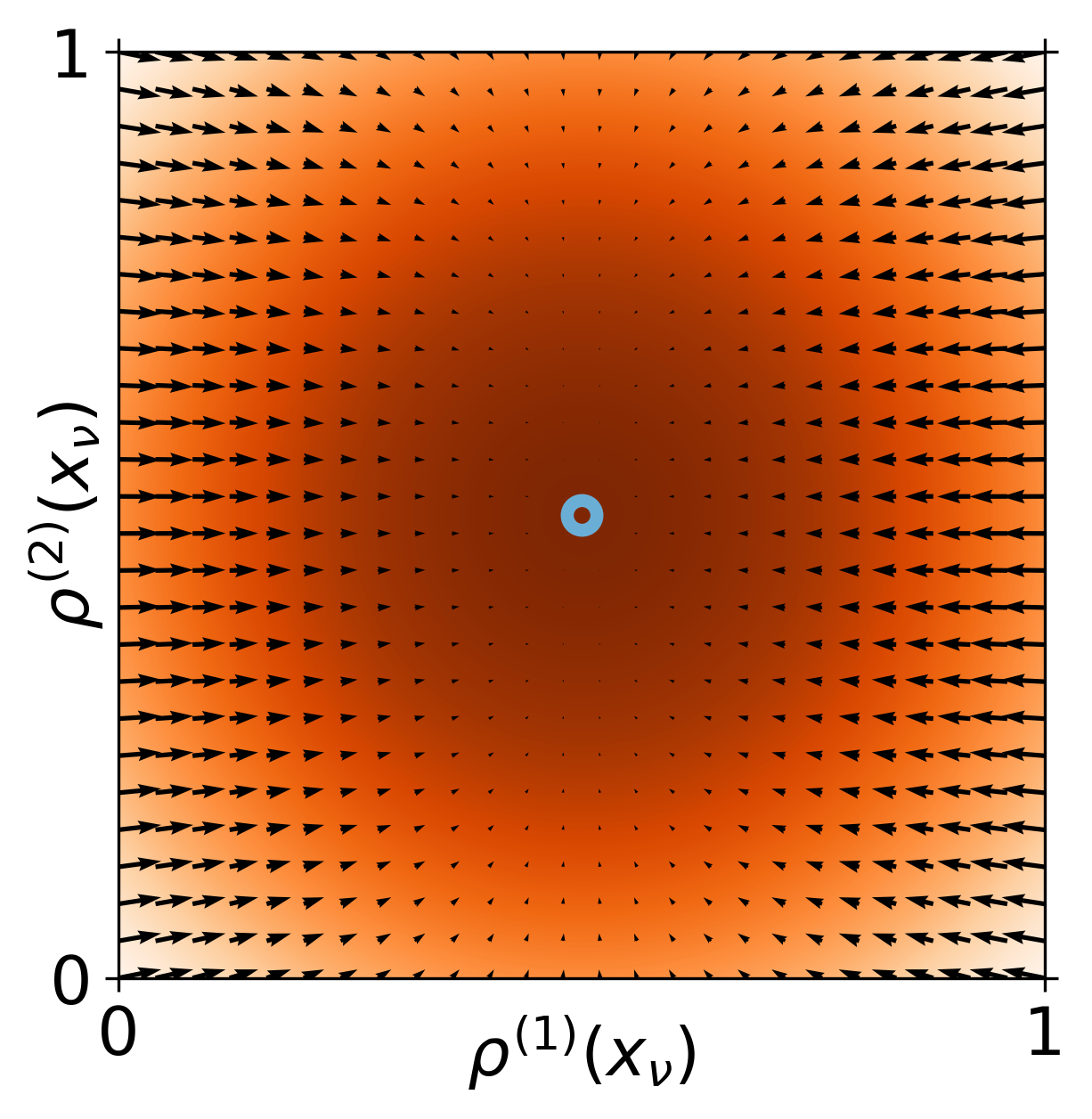}
    }
    \subfloat[$D^{(ii)}=1$, $D^{(12)}=0$, $\beta^{(1)}=\beta^{(2)}$]{
    \includegraphics[height=3.85cm]{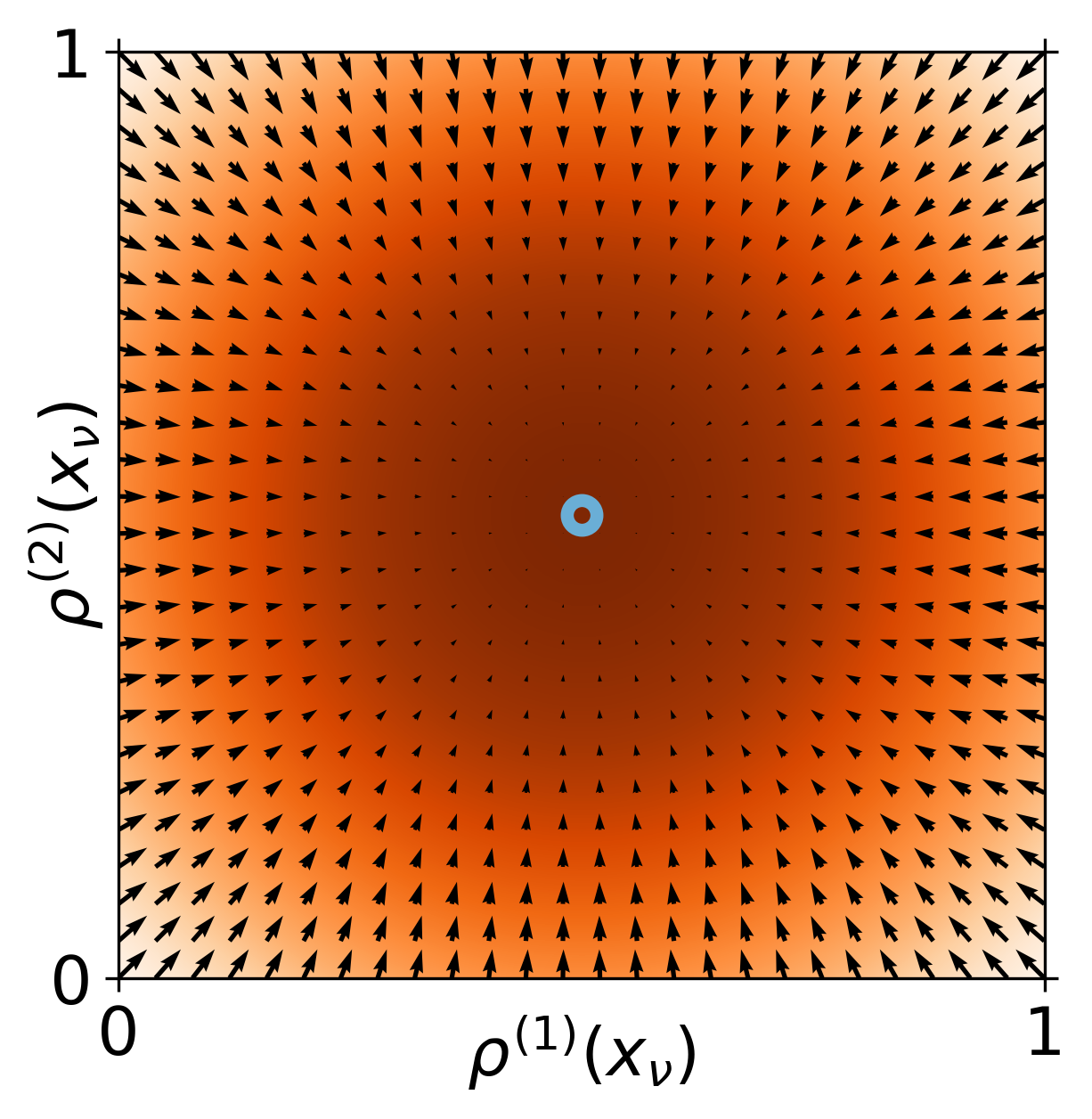}
    }
    \subfloat[$D^{(ii)}=1$, $D^{(12)}=0$,\newline $\beta^{(1)}<\beta^{(2)}$]{
    \includegraphics[height=3.85cm]{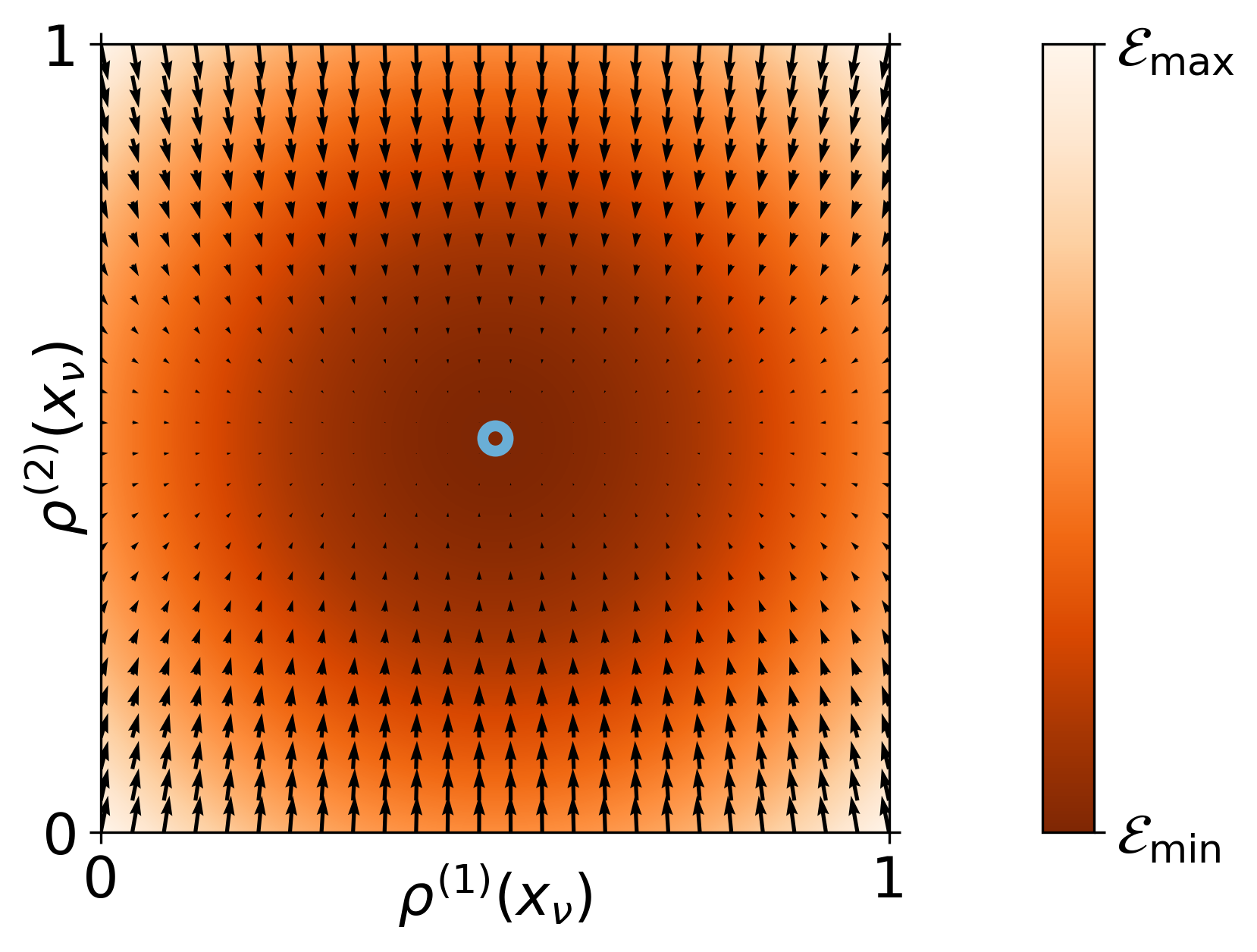}
    }
    \caption{Fixing $p=2$ and $\mathfrak{m}^{(i)}_{\iota\kappa}=\rho^{(i)}_\iota$, $D^{(11)}=D^{(22)}=1$ and $D^{(12)}=0$, we display the energy landscapes (heat map), stationary states (light blue circles), and dynamics (quivers) for varying $\beta^{(1)}$ and $\beta^{(2)}$.}
    \label{fig:two-point_vary:beta}
\end{figure}
\begin{figure}[ht!]
    \captionsetup[subfloat]{margin=10pt,format=hang,singlelinecheck=false}
    \centering
    \subfloat[$D^{(ii)}=1$, $D^{(12)}=0$,\newline $p=1.25$]{
    \includegraphics[height=3.85cm]{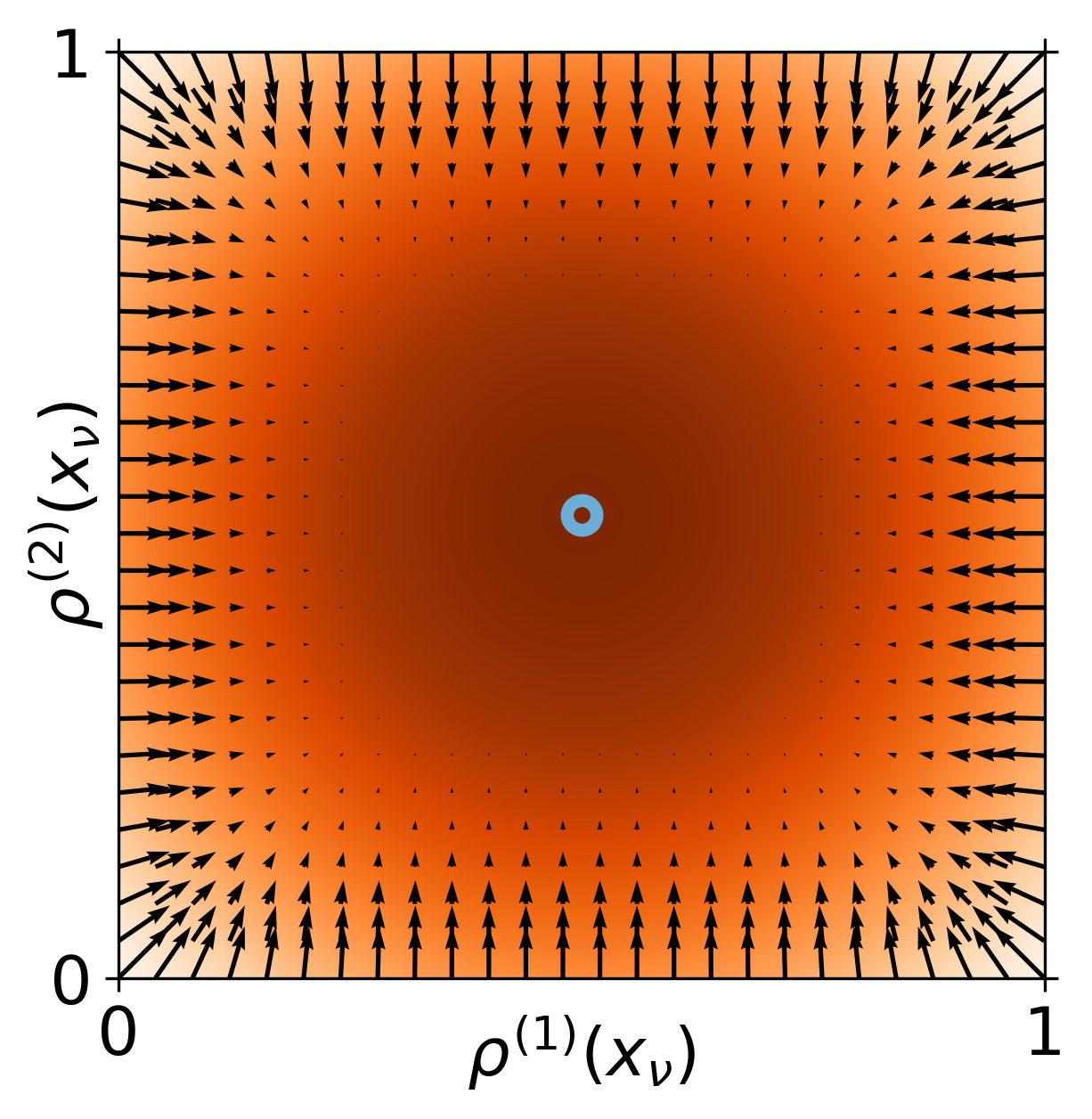}
    }
    \subfloat[$D^{(ii)}=1$, $D^{(12)}=0$, $p=2$]{
    \includegraphics[height=3.85cm]{2_point_D11=1.0_D12=0.0_D22=1.0_beta=1.0_q=2.0_colorbar=False_size=4.5cm.png}
    }
    \subfloat[$D^{(ii)}=1$, $D^{(12)}=0$,\\$p=5$]{
    \includegraphics[height=3.85cm]{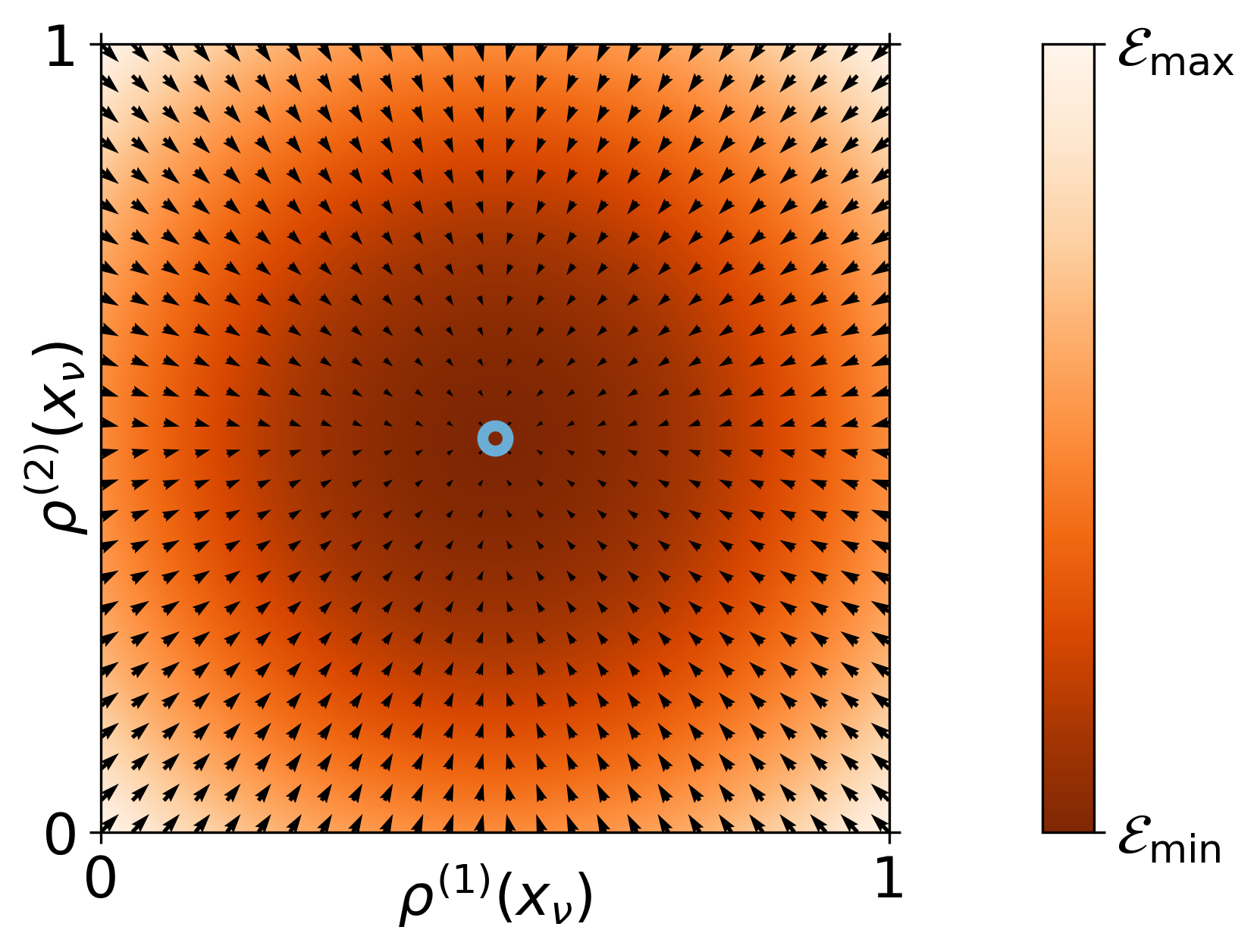}
    }
    \caption{Fixing $\mathfrak{m}^{(i)}_{\iota\kappa}=\rho^{(i)}_\iota$, $D^{(11)}=D^{(22)}=1$, $D^{(12)}=0$ and $\beta^{(1)}=\beta^{(2)}$, we display the energy landscapes (heat map), stationary states (light blue circles), and dynamics (quivers) for varying $p$.}
    \label{fig:two-point_vary:p}
\end{figure}

\begin{remark}[Visualization]\label{rem:two-point_figs}
In Figures \ref{fig:two-point_self:att_vary:cross}  - \ref{fig:two-point_vary:p}, we illustrate the effects of varying different parameters of the two-point model on the energy landscapes (heat maps), stationary states (light blue circles) and corresponding energy level sets (light blue lines), as well as dynamics (quivers) of the two-point space. Here, the mass $\rho^{(1)}(x_\iota) = 1 - \rho^{(1)}(x_\kappa)$ is displayed on the x-axis, while the mass $\rho^{(2)}(x_\iota) = 1 - \rho^{(2)}(x_\kappa)$ is displayed on the y-axis. We focus only on the case of linear mobility $\mathfrak{m}^{(i)}_{\iota\kappa}=\rho^{(i)}_\iota$ for $i=1,2$.

\begin{itemize}
\item In Figure \ref{fig:two-point_self:att_vary:cross} we vary the cross-interactions while staying in the regime of strong attractive self-interactions ($D^{(11)},D^{(22)}<0$, $D^{(11)}D^{(22)}>\left(D^{(12)}\right)^2$). We observe that it the cross-interaction is attractive (A), the stationary states $\rhoup_{\mathrm{c}}$ are optimal. If the cross-interaction is zero (B), both $\rhoup_{\mathrm{c}}$ and $\rhoup_{\mathrm{d}}$ are optimal. If the cross-interaction is repulsive (C),  $\rhoup_{\mathrm{d}}$ is optimal. 

\item In Figure \ref{fig:two-point_vary:det} we compare the different regimes $D^{(11)}D^{(22)} \lesseqqgtr  \left(D^{(12)}\right)^2$ for repulsive (top row) and attractive (bottom row) self-interactions. In (A) and (D) the cross-interaction is stronger than the self-interactions, which leads to $\rhoup_{\mathrm{d}}$ being optimal, both for repulsive and attractive self-interactions. In (B) the cross- and self-repulsions balance each other out, leading to a family of stationary states of the form $\rhoup_{\mathrm{a}_r}$ on the line $\rho^{(1)}(x_\iota)=1-\rho^{(2)}(x_\iota)$. Conversely, in (E) the cross-repulsions and self-attractions balance each other out on the line $\rho^{(1)}(x_\iota)=\rho^{(2)}(x_\iota)$, thus also leading to a family of stationary states of the form $\rhoup_{\mathrm{a}_r}$. Outside this line, the self-attractions and cross-repulsions draw the dynamics towards $\rhoup_{\mathrm{d}}$. In (C) the self-repulsions dominate the energy landscape, thus pushing the dynamics towards the only stationary state, $\rhoup_{\mathrm{a}}$. Conversely, in (F) the self-attractions dominate, leading to $\rhoup_{\mathrm{c}}$ and $\rhoup_{\mathrm{d}}$ being asymptotically stable, while $\rhoup_{\mathrm{a}}$, $\rhoup_{\mathrm{b}_1}$ and $\rhoup_{\mathrm{b}_2}$ are unstable.

\item In Figure \ref{fig:two-point_rotate} we vary the strength of the self-interactions, which have opposite signs, while fixing repulsive cross-interactions. In (A) the self-interaction of species $1$ is repulsive, leading to $\rhoup_{\mathrm{b}_1}$ being asymptotically stable. Conversely, in (C) the self-interaction of species $2$ is repulsive, leading to $\rhoup_{\mathrm{b}_2}$ being asymptotically stable. In (B) both self-interactions vanish, leading to $\rhoup_{\mathrm{d}}$ being asymptotically stable, similar to Figure \ref{fig:two-point_vary:det}(A,D).

\item In Figure \ref{fig:two-point_vary:beta} we display the effects of varying $\beta^{(1)}$ and $\beta^{(2)}$. We observe that the energy landscape and stationary states remain unchanged, while the ratio between $\partial_t\rho^{(1)}$ and $\partial_t\rho^{(2)}$ changes with the ratio $\beta^{(1)}$ and $\beta^{(2)}$, achieving symmetry between the species, when $\beta^{(1)}=\beta^{(2)}$.

\item In Figure \ref{fig:two-point_vary:p} we illustrate the effects of varying $p$. Similar to Figure \ref{fig:two-point_vary:beta}, we observe that the energy landscape and stationary states remain unchanged, while $\partial_t\rho^{(1)}$ and $\partial_t\rho^{(2)}$ change with $p$.
\end{itemize}
\end{remark}

\subsection{The effects of the cut-off function $\eta$ on a 3-point space}

Consider three points on a straight line, i.e., $\mu=\delta_{x_1}+\delta_{x_2}+\delta_{x_3}$, where $x_1=-r_1,x_2=0,x_3=r_2$ for some $r_1,r_2 > 0$. Let $K^{(ii)}(x,y)=\abs{x-y}$ for $i=1,2$ and  $K^{(12)}(x,y)=-\alpha\abs{x-y}$, for some $\alpha\ge0$. Besides, assume  $\mathfrak{m}^{(i)}_{\iota\kappa}=\rho^{(i)}_\iota$ for $i=1,2$ and $\iota,\kappa =1,2,3$. It is easy to check that the energy minimum is achieved, when one species is aggregated at $x_1$, while the other is aggregated at $x_3$.

However, when we impose the initial distribution $\rho^{(1)}_1=\rho^{(1)}_3=0$, $\rho^{(1)}_2=1$ and $\rho^{(1)}_1=1-\rho^{(1)}_3=\frac{1}{2}(1+\delta)$, $\rho^{(2)}_2=0$, for some $\delta \in (0,1)$, then the velocities given by
\baqs
v^{(i)}_{\iota\kappa}=\sum_{k=1}^2\sum_{\lambda=1}^3\left(K^{(ik)}_{\iota\lambda}-K^{(ik)}_{\kappa\lambda}\right)\rho^{(k)}_\lambda,
\eaqs
can be calculated to
\baqs
&v^{(1)}_{12}=r_1(\alpha\delta+1), &&v^{(1)}_{13} =r_1-r_2+\alpha\delta(r_1+r_2), &&&v^{(1)}_{23} =r_2(\alpha\delta-1),\\
&v^{(2)}_{12} =r_1(-\alpha-\delta), &&v^{(2)}_{13} =\alpha r_2-\alpha r_1-\delta(r_1+r_2), &&&v^{(2)}_{23} =r_2(\alpha-\delta).
\eaqs
Based on this, the time derivatives become
\baqs
\partial_t\rho^{(i)}_\iota=\frac{1}{\beta^{(i)}}\sum_{\kappa=1}^3\eta_{\iota\kappa}\left[(v^{(i)}_{\iota\kappa})_-\rho^{(i)}_\kappa-(v^{(i)}_{\iota\kappa})_+\rho^{(i)}_\iota\right],
\eaqs
whence
\baq\label{eq:3-point-dtrho}
\partial_t\rho^{(1)}_1&=\left[r_1(1+\alpha\delta)\right]_-\eta_{12},\\
\partial_t\rho^{(1)}_2&=-\left[r_1(1+\alpha\delta)\right]_-\eta_{12}-\left[r_2(1-\alpha\delta)\right]_-\eta_{23},\\
\partial_t\rho^{(1)}_3&=\left[r_2(1-\alpha\delta)\right]_-\eta_{23},\\
\partial_t\rho^{(2)}_1&=\frac{1}{2\beta}\left[\alpha (r_1-r_2)+\delta(r_1+r_2)\right]_-(1-\delta)\eta_{13}\\
&-\frac{1}{2\beta}\left[\alpha (r_1-r_2)+\delta(r_1+r_2)\right]_+(1+\delta)\eta_{13}-\frac{1}{2\beta}\left[r_1(\alpha+\delta)\right]_-(1+\delta)\eta_{12},\\
\partial_t\rho^{(2)}_2&=\frac{1}{2\beta}\left[r_1(\alpha+\delta)\right]_-(1+\delta)\eta_{12}+\frac{1}{2\beta}\left[r_2(\alpha-\delta)\right]_-(1-\delta)\eta_{23},\\
\partial_t\rho^{(2)}_3&=\frac{1}{2\beta}\left[\alpha (r_1-r_2)+\delta(r_1+r_2)\right]_+(1+\delta)\eta_{13}\\
&-\frac{1}{2\beta}\left[\alpha (r_1-r_2)+\delta(r_1+r_2)\right]_-(1-\delta)\eta_{13}-\frac{1}{2\beta}\left[r_2(\alpha-\delta)\right]_-(1-\delta)\eta_{23}.
\eaq
Let us consider some interesting particular cases of this situation and look at the effect of different supports of $\eta$.

\noindent
\textbf{Cutting off the interaction between $x_1$ and $x_3$.}
First, consider the situation, where $\eta(x,y)=\mathbb{1}_{\left\{\abs{x-y}<r\right\}}$ for some $r< r_1+r_2$. By \eqref{eq:3-point-dtrho} we see that then $\partial_t\rho^{(1)}=\partial_t\rho^{(2)} = 0$ if $\abs{\delta}\leq \alpha\land \alpha^{-1}$, i.e., we have a family of energetically non-optimal stationary states.

On the other hand, when $\eta_{13}>0$ we require $\delta=\alpha\frac{r_2-r_1}{r_2+r_1}$ together with $\abs{\delta}\leq \alpha\land \alpha^{-1}$ to obtain a stationary state, two conditions which are mutually exclusive for some choices of $r_1$, $r_2$ and $\alpha$.

In particular, we see that even for very weak cross-repulsion, i.e., for $\alpha\ll 1$, we may not obtain aggregation of both species in $x_2$ as one would observe it for a particle system with strong self-attraction and weak cross-repulsion.

\noindent
\textbf{Cutting off $x_3$.}
Assume $r_2>r_1$ and let $\eta(x,y)=\mathbb{1}_{\left\{\abs{x-y}<r\right\}}$ for some $r_1<r<r_2$. 
Then, the interaction of the node $x_3$ with $x_1$ and $x_2$ is completely cut off. Thus, the only remaining condition for our distribution to be stationary is $\delta\ge-\alpha\land-\alpha^{-1}$, leading to an even larger family of stationary states.

\subsection{Geometric effects on the dynamic in a 4-Point space}
In order to elucidate the geometric effects let us consider the following, rather minimal, example of a four point space. It is given by the base measure $\mu = \sum_{\iota=1}^4\delta_{x_\iota}$, which is embedded in $\R^2$. Here the nodes are  given by
\begin{align*}
    x_1 = (0,0), \quad x_2 = (1,0), \quad x_3=(1,1), \quad \text{and} \quad x_4 = (\varepsilon, 1),
\end{align*}
with $\varepsilon \in [0,1]$. For the sake of  the argument let us assume the two species are concentrated on the two neighboring nodes such that, initially, 
\begin{align*}
    \rho_0^{(1)} = \delta_{x_4}, \qquad \text{and} \qquad \rho_0^{(2)} = \delta_{x_1},
\end{align*}
respectively. For simplicity let $\mathfrak{m}^{(i)}_{\iota\kappa}=\rho^{(i)}_\iota$ for $i=1,2$, $\iota,\kappa =1,2,3$, and let $p=2$. In addition, we assume attractive intra-specific interactions of the form
\begin{align*}
    K^{(ii)} (x,y) = |x-y|_1,\quad  i=1,2,
\end{align*}
and mutually repulsive inter-specific interactions of the form 
\begin{align*}
    K^{(12)} (x,y) = - \alpha |x-y|_1,
\end{align*}
for some $\alpha \geq 0$. Here, $|x|_1 = |x^{1}|  + |x^2|$. We also assume that the cross-interaction is symmetric, i.e., $\betaup=(1,1)$. Intuitively, for this type of cross-interactions, we expect that the second species pushes the first one onto the opposing node, $x_3$. However, we now show that this is not the case. On the contrary, let us compute
\begin{align}
    \label{eq:change_of_mass_in_node_3}
    \partial_t \rho^{(1)}_3
    &= \sum_{\lambda=1}^4 (v^{(1)}_{3\lambda})_- \rho^{(1)}_\lambda - (v^{(1)}_{3\lambda})_+ \rho^{(1)}_3.
\end{align}
Recalling that, initially,
\begin{align*}
    v^{(1)}_{\iota\kappa} &= \sum_{\lambda=1}^4 \left[K^{(11)}_{\iota\lambda} - K^{(11)}_{\kappa\lambda}\right] \rho^{(1)}_\lambda + \sum_{\lambda=1}^4 \left[K^{(12)}_{\iota\lambda} - K^{(12)}_{\kappa\lambda}\right]\rho^{(2)}_\lambda\\
    &=\left[K^{(11)}_{\iota4} - K^{(11)}_{\kappa4}\right] \rho^{(1)}_4 +  \left[K^{(12)}_{\iota1} - K^{(12)}_{\kappa1}\right]\rho^{(2)}_1\\
    &=\left[K^{(11)}_{\iota4} - K^{(11)}_{\kappa4}\right] +  \left[K^{(12)}_{\iota1} - K^{(12)}_{\kappa1}\right],
\end{align*}
as all mass is concentrated on the two nodes, $x_1, x_4$. With the above choice of interaction potentials, we obtain
\begin{align}
    \label{eq:velo_4point_space}
   v^{(1)}_{\iota\kappa} = |x_\iota-x_4| - |x_\kappa-x_4| - \alpha |x_\iota-x_1| + \alpha |x_\kappa-x_1|.
\end{align}
Substituting \eqref{eq:velo_4point_space} into  \eqref{eq:change_of_mass_in_node_3} this yields
\begin{align*}
    \partial_t \rho^{(1)}_3 &= \sum_{\iota=1}^{4} (v^{(1)}_{3\iota})_- \rho^{(1)}(x_\iota) - (v^{(1)}_{3\iota})_+ \rho^{(1)}_3=(v^{(1)}_{34})_-\\ &=[1 - \varepsilon - 2\alpha + \alpha (1+\varepsilon)]_-=[(1-\alpha)(1-\varepsilon)]_-.
\end{align*}
\begin{figure}[ht!]
    \captionsetup[subfloat]{labelformat=empty}
    \centering
    \subfloat[$t=0$]{
    \includegraphics[width=0.18\textwidth]{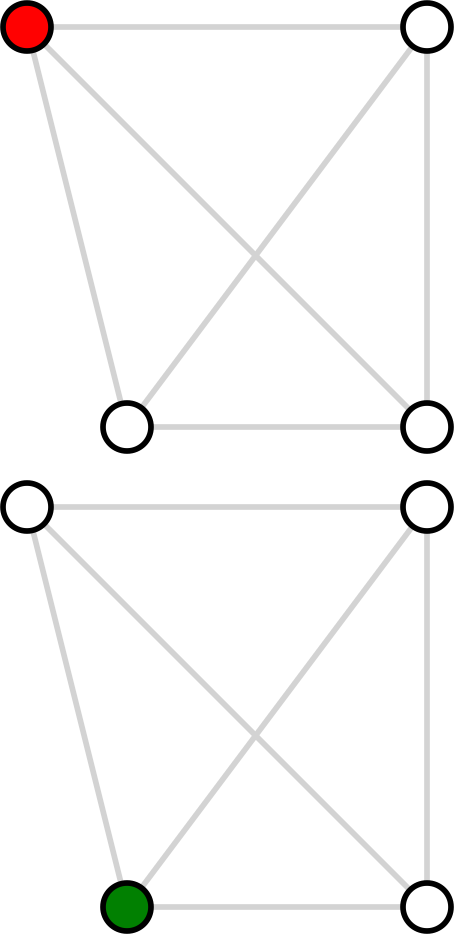}
    }
    \subfloat[$t=4$]{
    \includegraphics[width=0.18\textwidth]{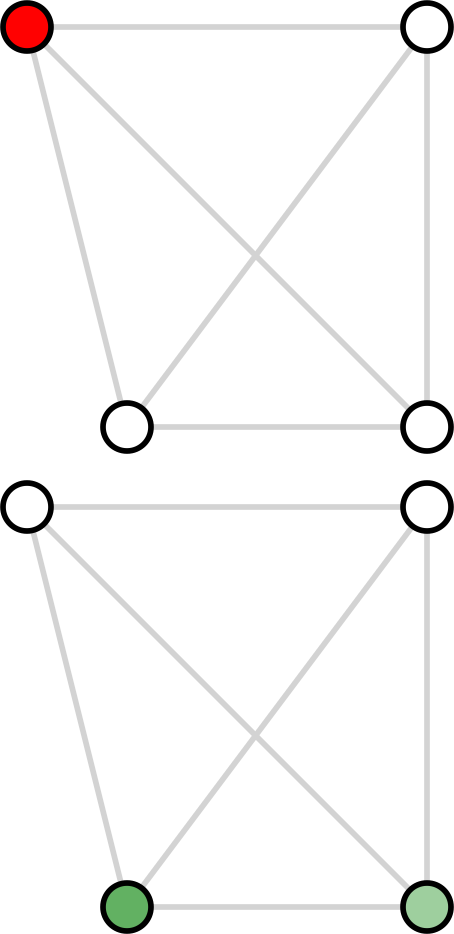}
    }
    \subfloat[$t=7$]{
    \includegraphics[width=0.18\textwidth]{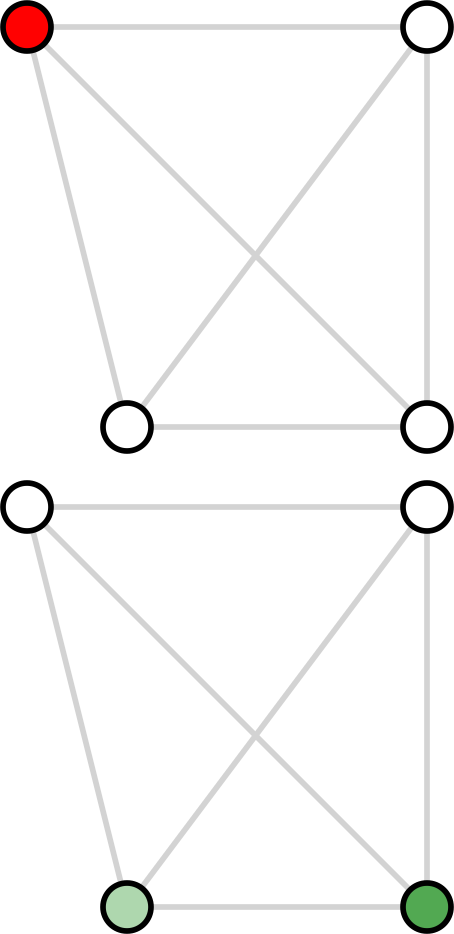}
    }
    \subfloat[$t=9$]{
    \includegraphics[width=0.18\textwidth]{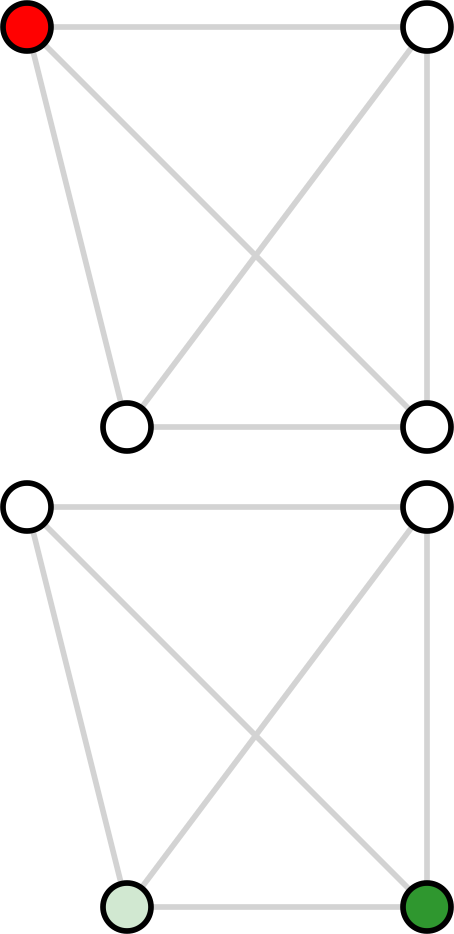}
    }
    \subfloat[$t=14$]{
    \includegraphics[width=0.18\textwidth]{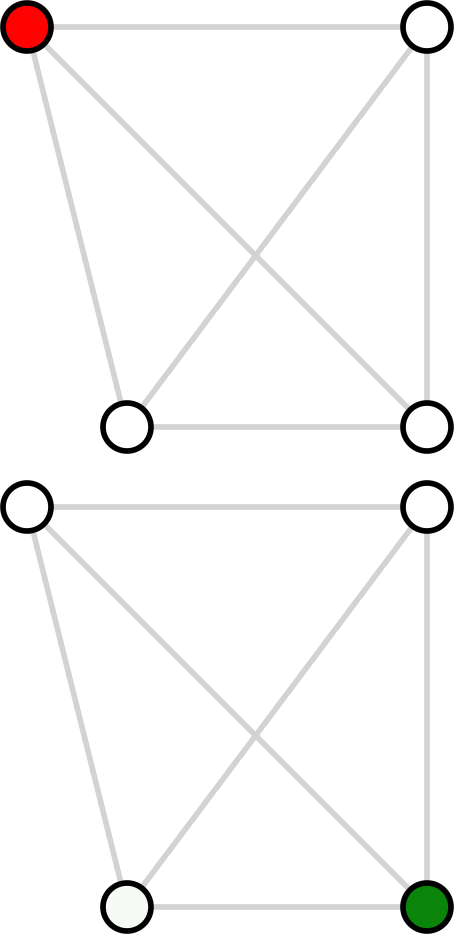}
    }
    \caption{We display the dynamics of the non-symmetric ($\varepsilon \in (0,1)$) four-point space with the initial mass of $\rho^{(2)}$ concentrated on $x_4$ (green, bottom row) and the initial mass of $\rho^{(1)}$ concentrated on $x_1$ (red, top row). Only if the cross-repulsion strength is strong enough, $\alpha>1$, can the second species drive the first one out of the node it is inhabiting. Against any intuition based on the associated particle system, this is not the case if the repelling cross-interaction is weak.}
    \label{fig:fourpoint_epsilon}
\end{figure}
In particular, for $0 \leq \alpha < 1$ no mass flows, even if the node $\rho^{(1)}$ is concentrated on, $x_4$, is moved closer to $x_3$. In fact, only if the cross repulsion is strong enough, i.e., $\alpha > 1$, mass  flows onto $x_3$. What is more, this remains true even if $\varepsilon = 0$, a scenario in which the dynamics of the associated particle system would lead to two particles, one for each species, that move into opposing directions. 

A similar computation reveals that, initially,
\begin{align*}
    \partial_t \rho^{(2)}_1 =  - (-1 + \alpha  - 2 \alpha \epsilon )_+ - (-2 - \alpha \epsilon  + \alpha  - \epsilon)_+  - (-\epsilon - \alpha \epsilon)_+.
\end{align*}
This shows that the second species remains concentrated on the node $x_1$ whenever $0\leq \alpha \leq  1$ and $\epsilon \geq 0$. Moreover, for cross-interaction strengths that are stronger than the self-interactions, $\alpha>1$, it is still possible for the second species to remain stationary; however, now,  the size of $\epsilon$ begins to play a role made precise by the above equation. For instance, if $\alpha=3/2$ and $\epsilon=1/4$, the second species remains located on the first node while it pushes the first species from the fourth node to the third node, which also follows from the form of $\partial_t \rho^{(1)}_3$, initially. The behavior is displayed in Figure 
\ref{fig:fourpoint_epsilon}. If $\epsilon$ is sufficiently small or zero but $\alpha>1$, the above equation yields a change in density of the second species on $x_1$. Similarly, the first species will escape to node $x_3$, which can be seen in Figure \ref{fig:fourpoint_square}.
This relatively minimal example goes to show that the dynamics on the graph are very different from the particle system associated to the energy functional \eqref{eq:particle}.

\begin{figure}[h]
    \captionsetup[subfloat]{labelformat=empty}
    \centering
    \subfloat[$t=0$]{
    \includegraphics[width=0.18\textwidth]{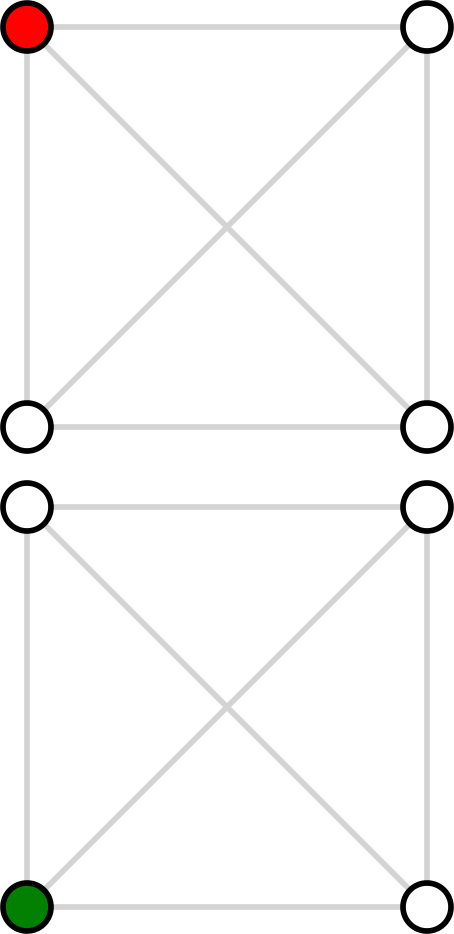}
    }
    \subfloat[$t=6$]{
    \includegraphics[width=0.18\textwidth]{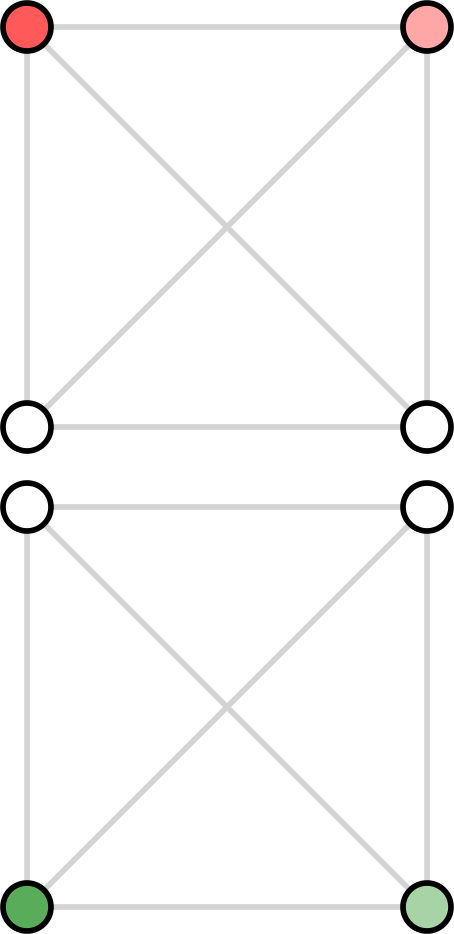}
    }
    \subfloat[$t=11$]{
    \includegraphics[width=0.18\textwidth]{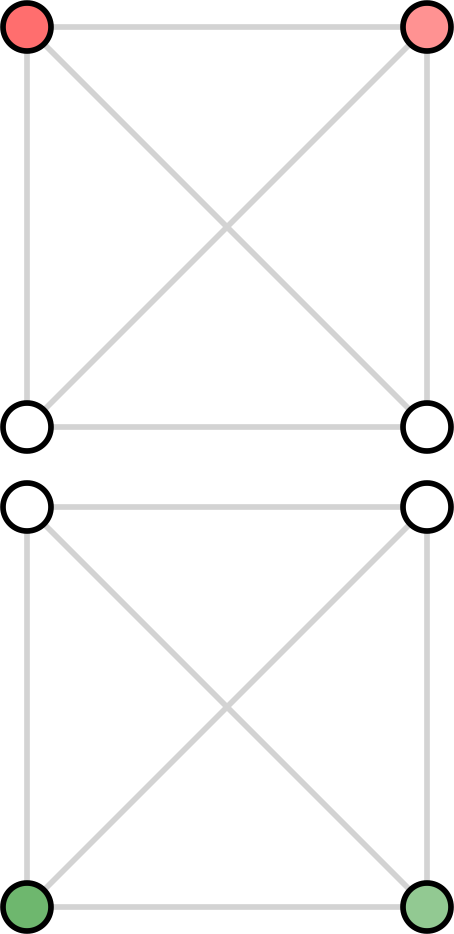}
    }
    \subfloat[ $t=17$]{
    \includegraphics[width=0.18\textwidth]{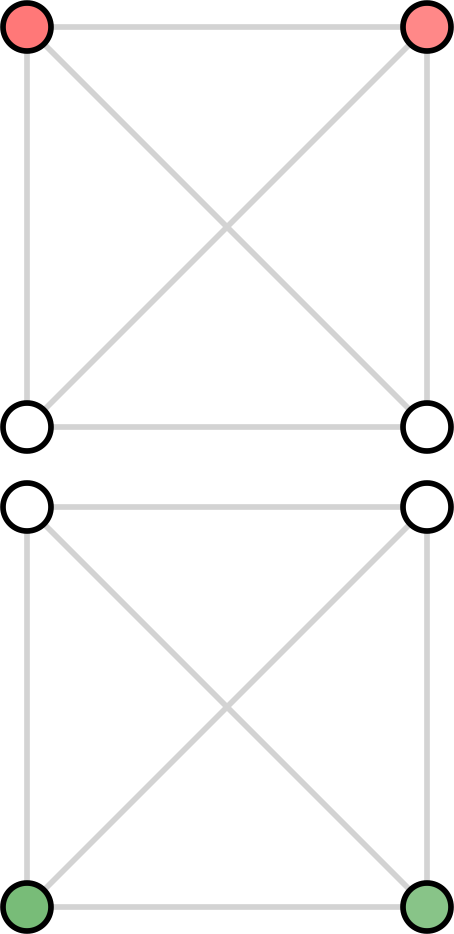}
    }
    \subfloat[$t=20$]{
    \includegraphics[width=0.18\textwidth]{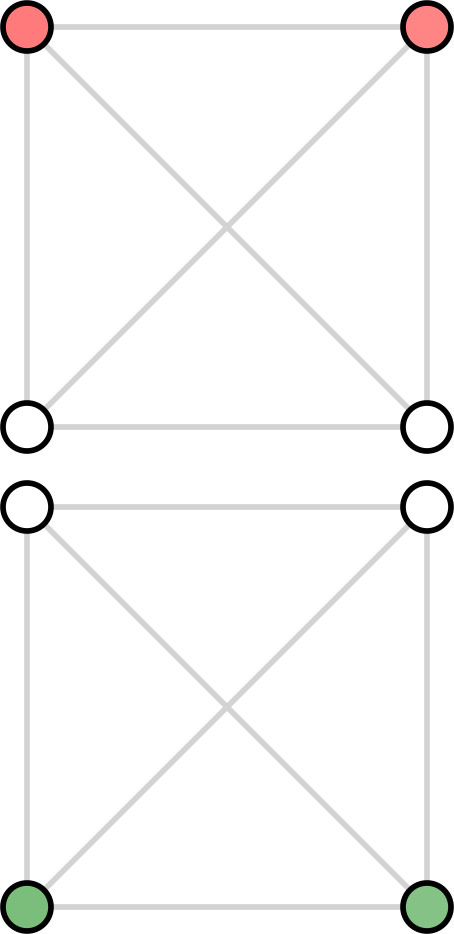}
    }
    \caption{In the symmetric setting, two species that are initially concentrated on neighbouring vertices of the four-point space will drive out one another in the case of sufficiently strong cross-repulsion, $\alpha>1$. Only if the cross-repulsion strength is strong enough, $\alpha>1$, can the second species drive the first one out of the node it is inhabiting and vice versa. This phenomenon is quite counter-intuitive as both species are self-attractive. Nonetheless, the mass of both species is evenly split. Against any intuition based on the associated particle system, this is not the case if the repelling cross-interaction is weak.}
    \label{fig:fourpoint_square}
\end{figure}

\begin{figure}[h]
    \captionsetup[subfloat]{labelformat=empty}
    \centering
    \subfloat[$t=0$]{
    \includegraphics[width=0.18\textwidth]{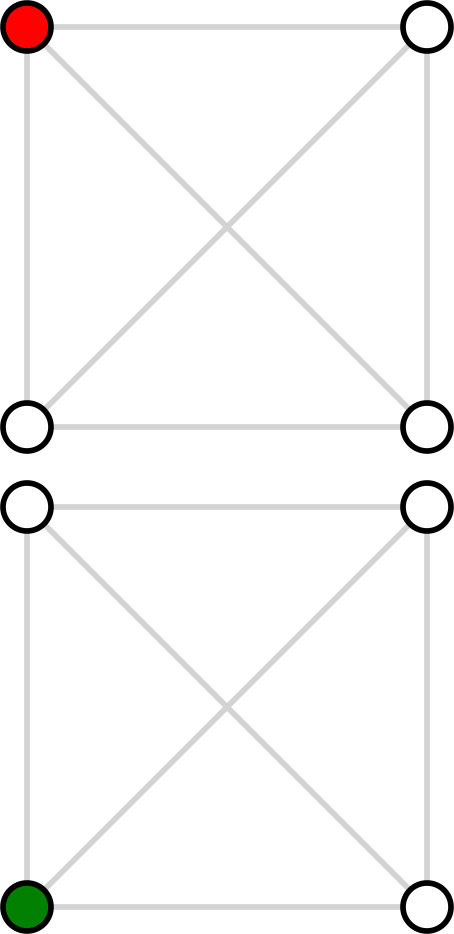}
    }
    \subfloat[$t=1$]{
    \includegraphics[width=0.18\textwidth]{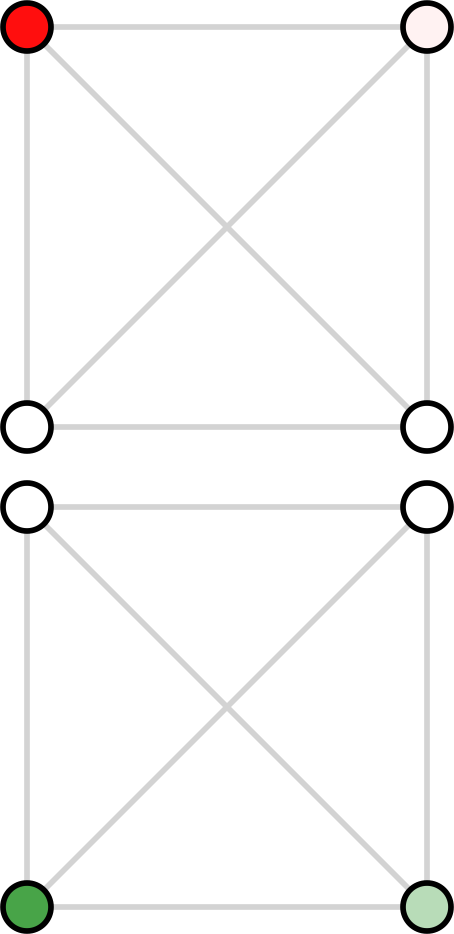}
    }
    \subfloat[$t=2$]{
    \includegraphics[width=0.18\textwidth]{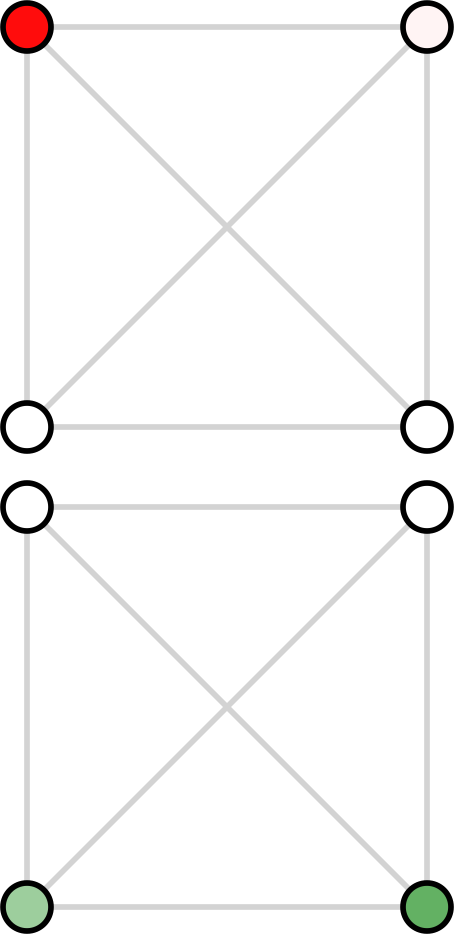}
    }
    \subfloat[$t=3$]{
    \includegraphics[width=0.18\textwidth]{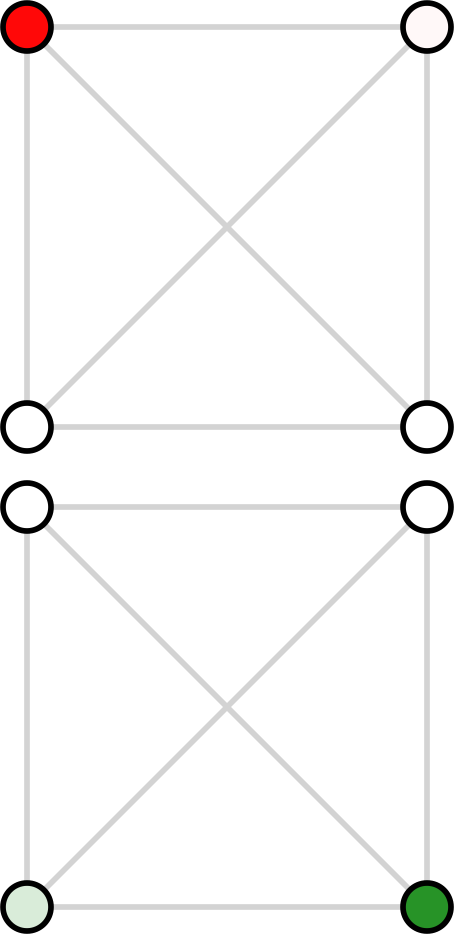}
    }
    \subfloat[$t=4$]{
    \includegraphics[width=0.18\textwidth]{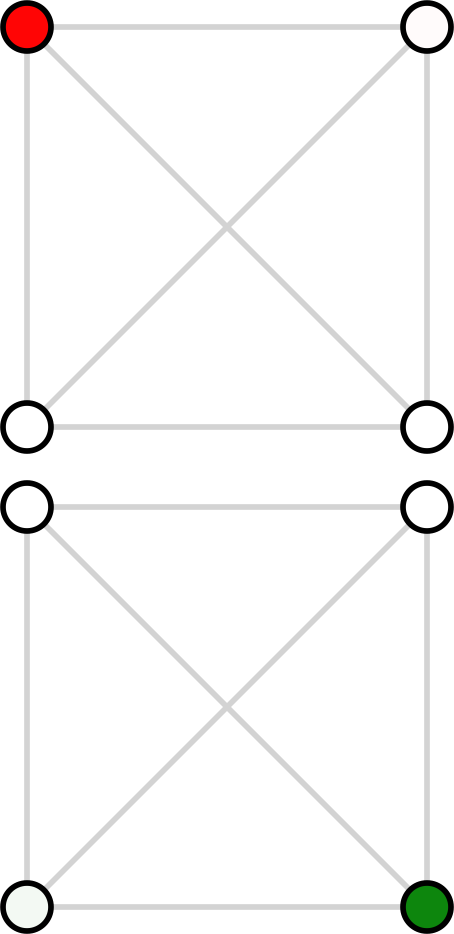}
    }
    \caption{In the situation with symmetric nodes and non-symmetric cross-repulsion (i.e., $\beta \ne 1$), one of the two species which are initially concentrated on neighboring edges of the four-point space will drive out the other in the case of sufficiently strong cross-repulsion, $\alpha>1$. Again, against particle system intuition, this is not the case if the repelling cross-interaction is weak.}
    \label{fig:fourpoint_square_beta<1}
\end{figure}

\newpage
\subsection{Pattern formation on large graphs}
In the preceding section we have studied two species whose intra-specific interaction is of attractive nature while the cross-interaction is purely repulsive. With the intuition gained in the first example, let us consider a larger graph of $10$-by-$10$ nodes arranged in a two-dimensional array with $m(r,s)=r$ and $p=2$. Focusing only on kernels of the form $K^{(ik)}(x,y) = \tilde{K}^{(ik)}(x-y)$, we assume the cross-interactions to be repulsive and to be given by the potentials 
\begin{align*}
    \tilde{K}^{(12)}(x) = (10 - 20|x|)_+.
\end{align*}
Potentials with a compact support like the cross-interaction potential are typically chosen whenever interactions are nonlocal, but the sensing radius is imposed. As for the above choice, this sensing radius is $r=1/2$, a distance beyond which no repulsive interactions are present. 

Concerning the self-interactions, we choose the potentials proposed by Kolokolnikov and Evers, cf. \cite{evers2017equilibria}, which are of the form
\begin{align*}
    \tilde{K}^{(ii)}(x) = - a \log(|x|) + b \frac{|x|^2}{2},
\end{align*}
i.e., consisting of a locally repulsive part and a globally attractive part, respectively. This potential has the remarkable property that it leads to compactly supported steady states of constant density. In fact, for a single species, there holds
\begin{align*}
    \rho^{(i)}\nabla \tilde{K}^{(ii)}\ast \rho^{(i)} = 0,
\end{align*}
or, on the support of $\rho^{(i)}$, we have
\begin{align*}
    \nabla \tilde{K}^{(ii)} \ast \rho^{(i)} = 0.
\end{align*}
Taking the divergence yields
\begin{align*}
    0 
    = \Delta \tilde{K}^{(ii)} \ast \rho^{(i)}
    = -\frac{a}{2\pi} \delta \ast \rho^{(i)} + b
    = b - \frac{a}{2\pi} \rho^{(i)},
\end{align*}
i.e.,
$$ 
    \rho^{(i)} = 2\pi b / a,
$$
where $\rho^{(i)}>0$. For our purpose we choose
\begin{align}
    \label{eq:KEF_self_terms}
    \tilde{K}^{(ii)}(x) = -\log(|x|) - \frac{|x|^2}{400}.
\end{align}
In Figure \ref{fig:10by10graph}, we observe that starting from randomly distributed mass on the nodes ($t=0$), the mass immediately begins to aggregate, ($t=1$ and $t=10$). As time passes the profile approaches the stationary profile ($t=5000$).
\begin{figure}[ht!]
    \captionsetup[subfloat]{labelformat=empty}
    \captionsetup[subfloat]{labelformat=empty}
    \centering
    \subfloat[$t=0$]{
    \includegraphics[width=0.18\textwidth]{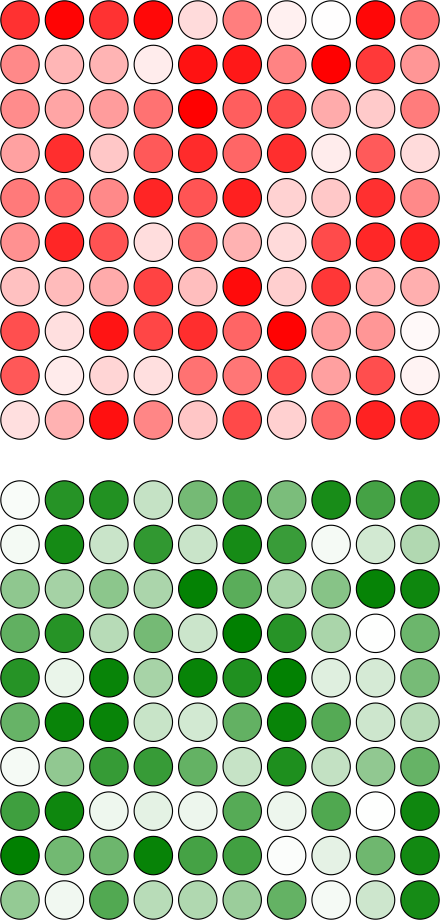}
    }
    \subfloat[$t=1$]{
   \includegraphics[width=0.18\textwidth]{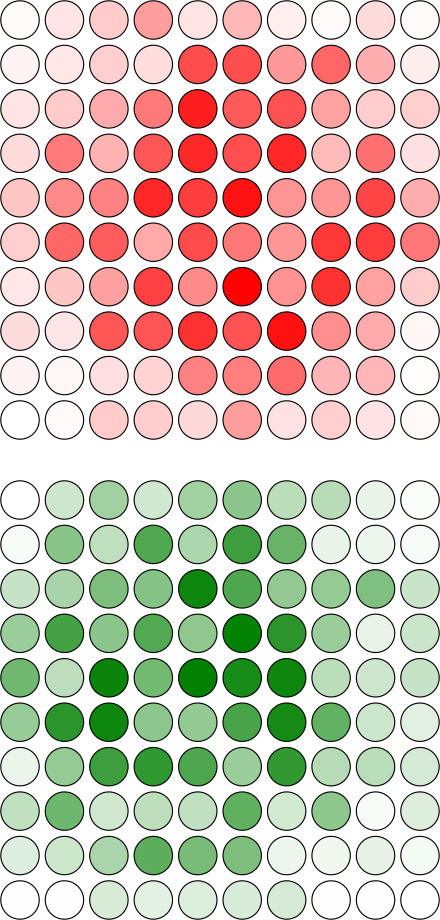}
    }
    \subfloat[$t=10$]{
   \includegraphics[width=0.18\textwidth]{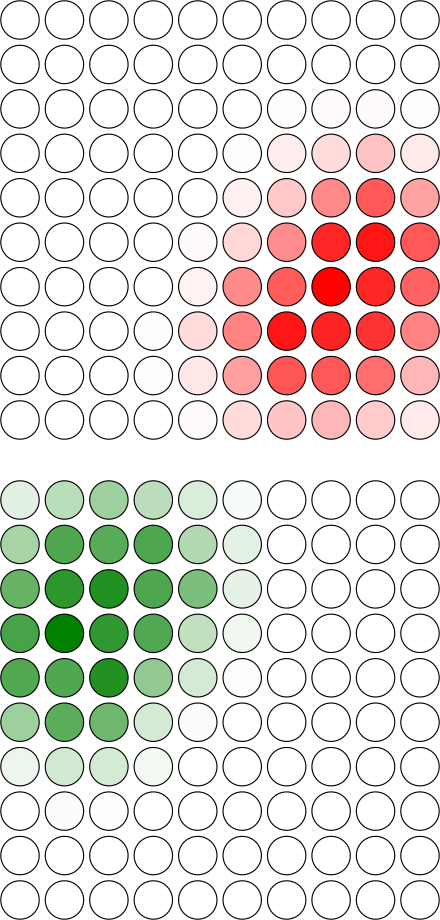}
    }
    \subfloat[$t=200$]{
    \includegraphics[width=0.18\textwidth]{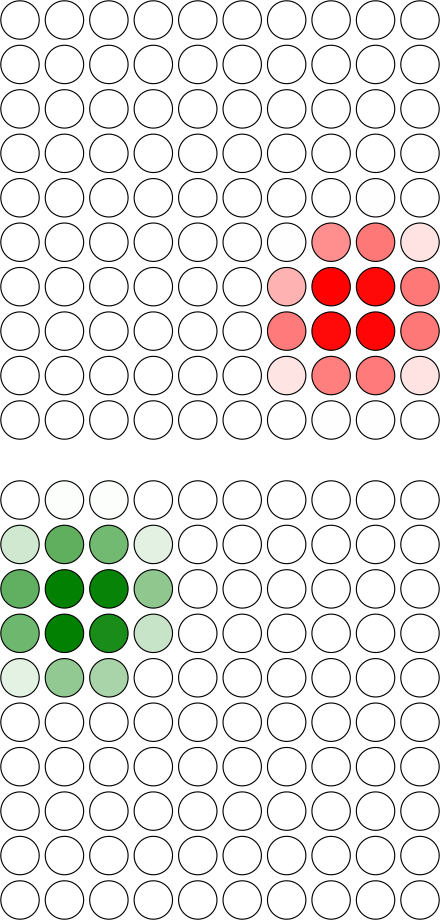}
    }
    \subfloat[$t=5000$]{
    \includegraphics[width=0.18\textwidth]{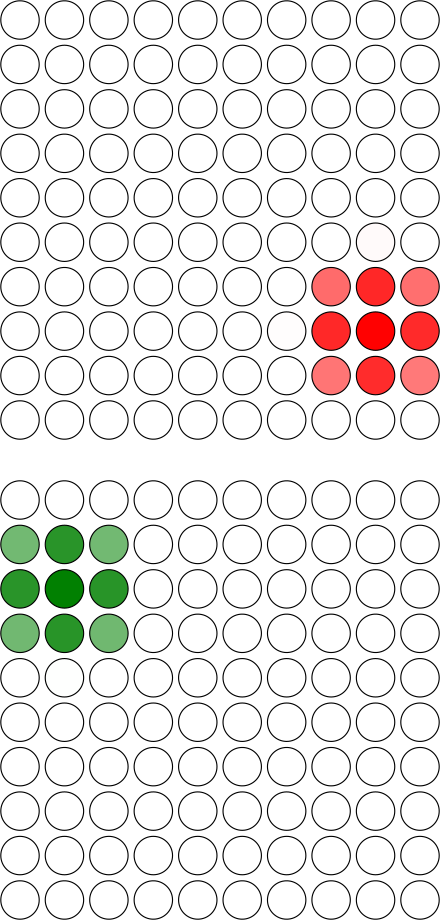}
    }
    \caption{Short-range repulsive, long-range attractive self-interactions and repulsive interactions with a sensing radius lead to segregation (right panel) of the initially randomly distributed species (left panel). The two species are color-coded in red and green, respectively.}
    \label{fig:10by10graph}
\end{figure}

Based on this numerical experiment we alter the self-interaction potential slightly in that we introduce a sensing radius, too. Instead of \eqref{eq:KEF_self_terms}, we truncate the self-interaction potentials and consider
\begin{align*}
    \tilde{K}^{(ii)}(x) = \left(-a \log(|x|) +  b |x|^2\right) \mathbb{1}_{\{|x|< s\}} + \left(-a \log(s) +  b s^2\right)\mathbb{1}_{\{|x| \geq s\}}
\end{align*}
where $a=1$, $b=1/400$, and the sensing radius is chosen as $0<s = 0.2$. Outside of the interaction region, $|x|>s$, the potential is extended by a constant. Concerning the cross-interaction terms, we use compactly supported, repulsive potentials of the form
\begin{align*}
    \tilde{K}^{(12)}(x) = (4 - 20|x|)_+,
\end{align*}
i.e., the cross-interaction radius is $s=0.5$. In Figure \ref{fig:25by25graph} we give a representative example of the dynamics these interaction rules give rise to. As in the previous example, we initialize both densities randomly. As time evolves, we observe a self-sorting phenomenon as regions of co-habitation disappear, and islands of individual occupancy emerge, cf. Figure \ref{fig:25by25graph}, $t=200$.
\begin{figure}[ht!]
    \captionsetup[subfloat]{labelformat=empty}
    \centering
    \subfloat[$t=0$]{
    \includegraphics[width=0.21\textwidth]{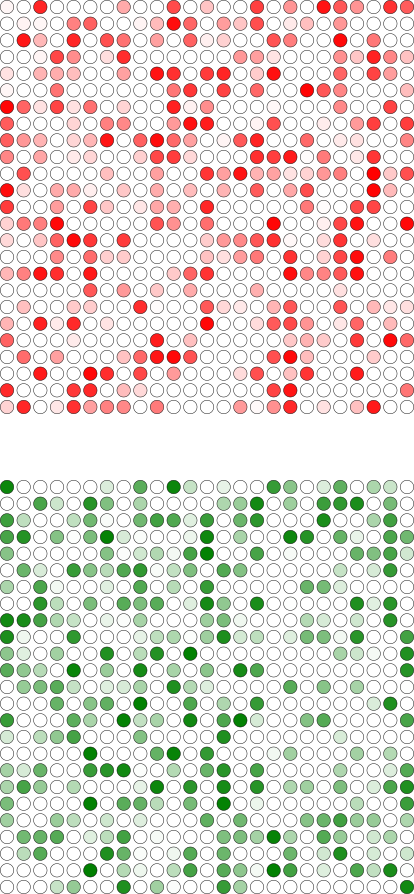}
    }\quad
    \subfloat[$t=50$]{
    \includegraphics[width=0.21\textwidth]{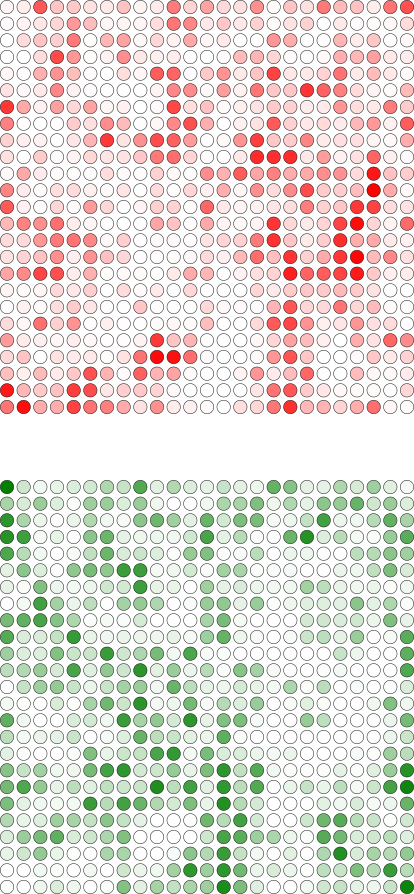}
    }\quad
    \subfloat[$t=100$]{
    \includegraphics[width=0.21\textwidth]{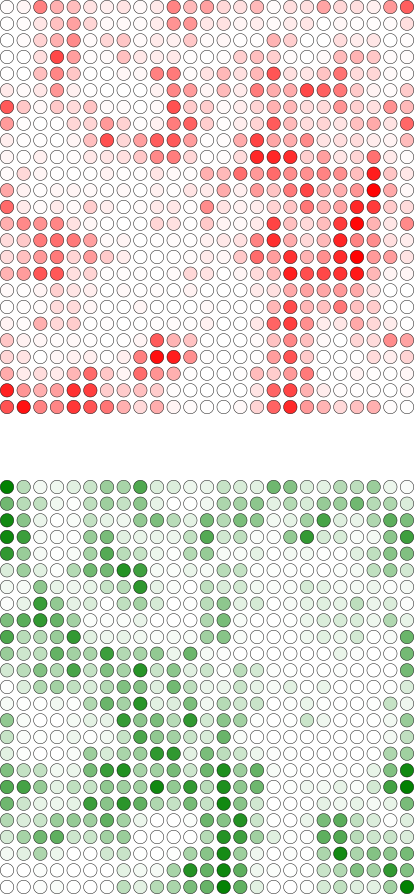}
    }\quad
    \subfloat[$t=200$]{
    \includegraphics[width=0.21\textwidth]{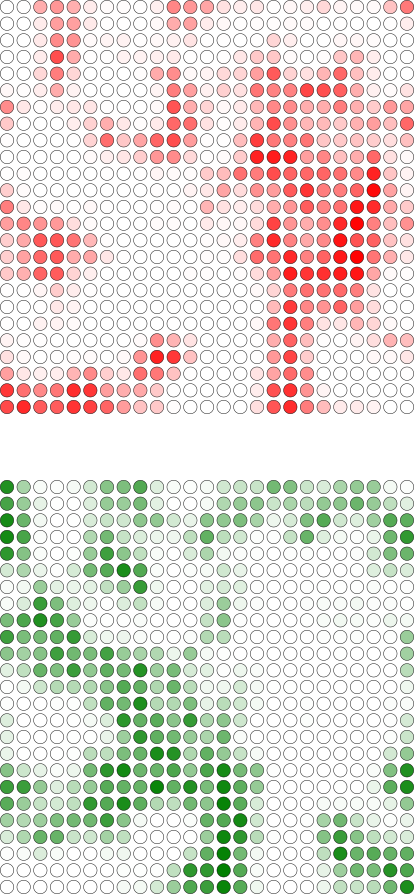}
    }    
    \caption{Starting from random initial data (left panel), the two species (red and green, respectively) have a sensing radius with respect to both, members of their own species and members of the opposing species. We observe emergence behavior in the form of self-sorting, i.e., the final configuration consists of two species with disjoint habitats.}
    \label{fig:25by25graph}
\end{figure}

\subsection{Nonlinear mobilities and nonquadratic exponents}
In this section, we want explore the effects of changing the mobility $m$ as well as the exponent $p=q/(q-1)$ in \eqref{eq:NL2CIE}. To this end, we study the case of $10$-by-$10$ nodes arranged in a two-dimensional lattice, which is fully connected and where the base measure is given by $\mu_\iota =1/20$ on all vertices. To observe the effects of different mobilities most clearly, we choose attractive interactions given by $\tilde{K}^{(ik)}(x) = 20(1-\exp(\abs{x}))$ for $i,k=1,2$. We consider the linear mobility $m(r,s)=r$ with $p\in\{1.65,2,5\}$ as well as the volume filling mobility $m(r,s)=r(1-s)$ for $p=2$. For the comparison of the different models, we choose the same random initial distribution for all cases.

In Figure \ref{fig:linear_10by10}, for all $p$ we observe aggregation of both species to a single vertex as time progresses. Comparing the different $p$, we find that an increase of $p$ slows down the convergence speed towards the stationary state. This is due of the power $q-1=1/(p-1)$ in \eqref{eq:NL2CIE}. If $q<1$ is small,  velocities near $0$ increase, speeding up the convergence near the stationary state. Conversely, if $q>1$ is large,  velocities near $0$ decrease further, slowing down the convergence near the stationary state.

In contrast to the aggregation in Figure \ref{fig:linear_10by10}, in Figure \ref{fig:volume-filling_10by10_p=2} we see that the volume filling mobility truncates the set of admissible states, eliminating configurations, which assign to much mass of any species $i\in\{1,2\}$ to a single vertex. Because $\mu_\iota =1/20$ for all $\iota\in\{1,\ldots,100\}$, the maximum threshold is $\rho^{(i)}_\iota =1/20$ for all $\iota\in\{1,\ldots,100\}$ and $i=1,2$. Due to the geometry of the graph and the kernel structure, the total mass of the stationary state in the volume filling case is spread to $21$ vertices.

\begin{figure}[ht!]
    \captionsetup[subfloat]{labelformat=empty}
    \centering
    $p=1.65$\vspace{-.2cm}\\
    \subfloat{
    \includegraphics[width=0.18\textwidth]{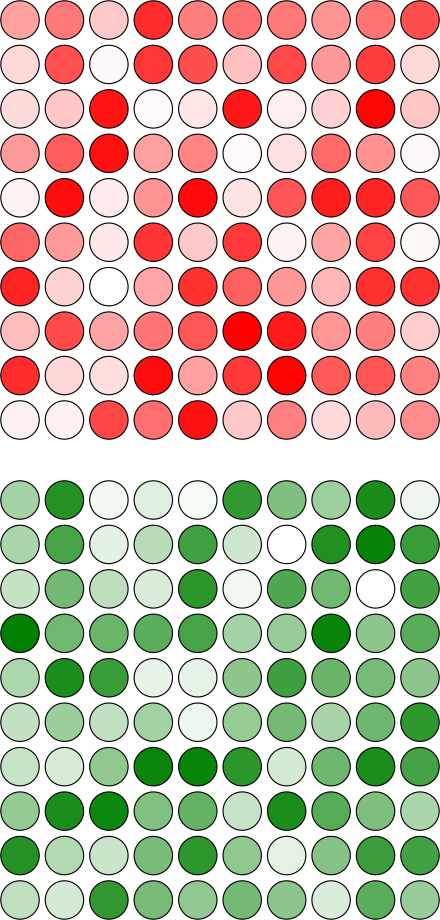}
    }
    \subfloat{
    \includegraphics[width=0.18\textwidth]{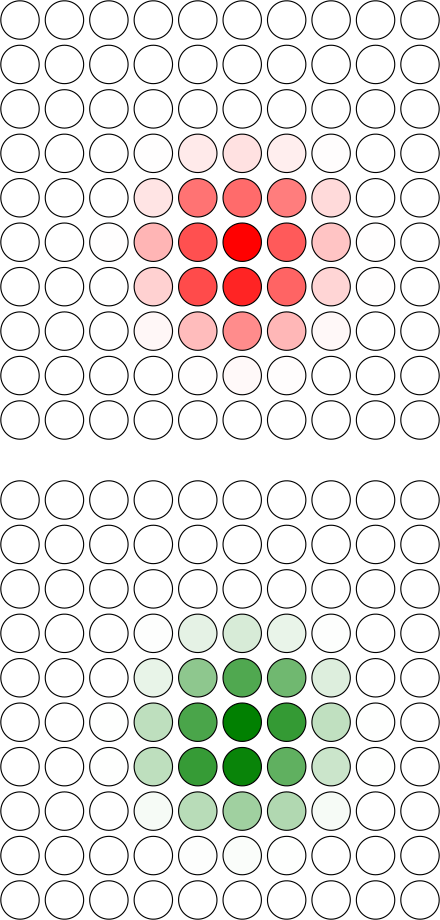}
    }
    \subfloat{
    \includegraphics[width=0.18\textwidth]{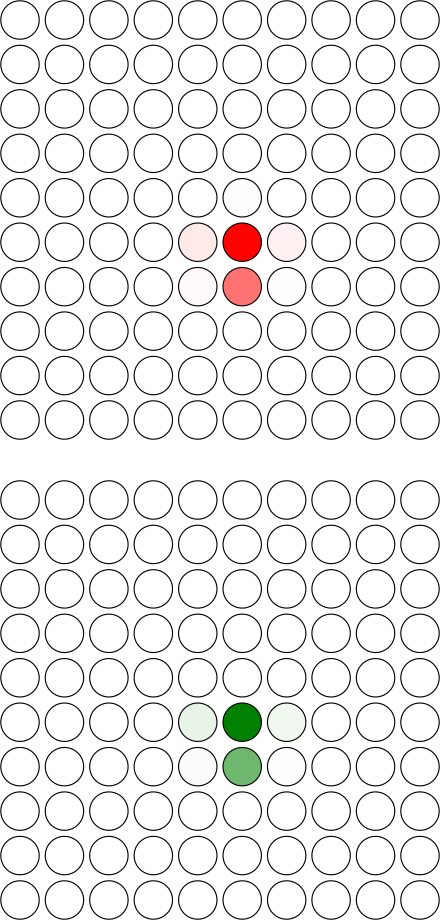}
    }
    \subfloat{
    \includegraphics[width=0.18\textwidth]{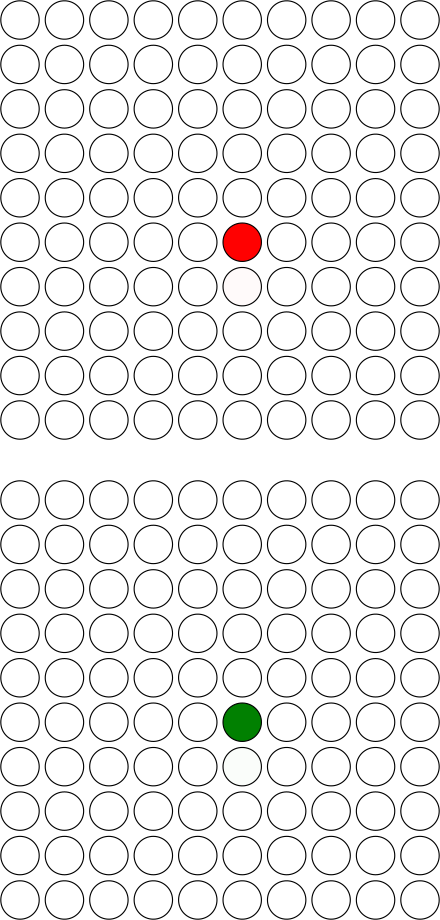}
    }
    \subfloat{
    \includegraphics[width=0.18\textwidth]{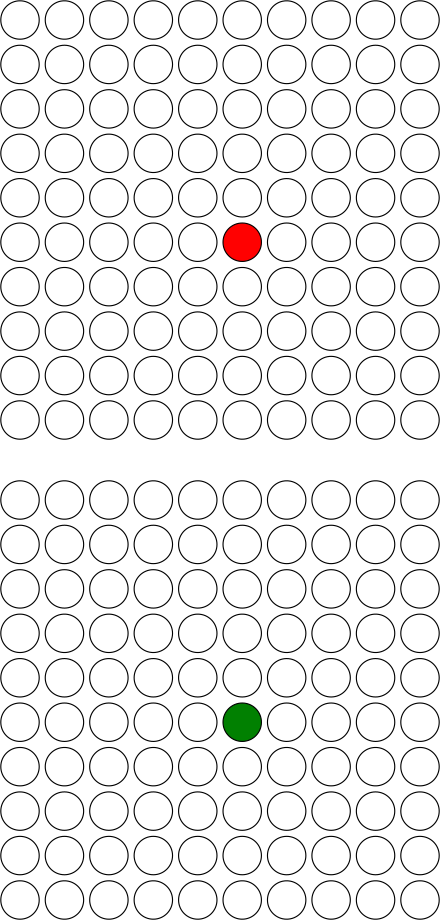}
    }\\
    $p=2$\vspace{-.2cm}\\
    \subfloat{
    \includegraphics[width=0.18\textwidth]{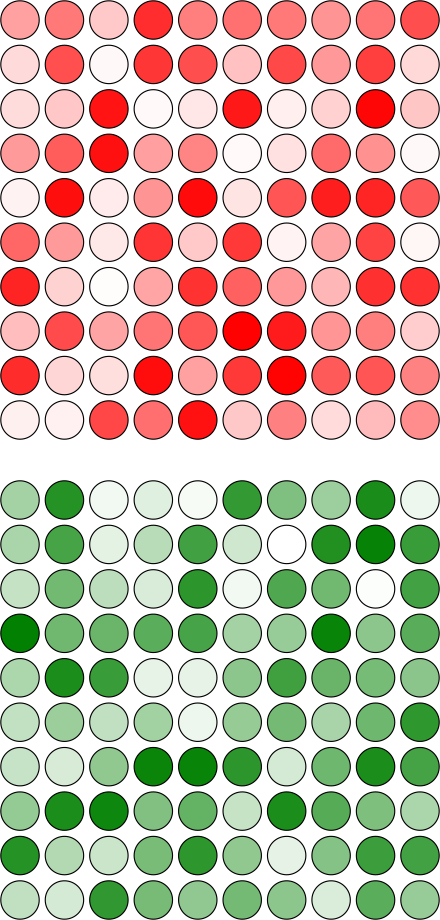}
    }
    \subfloat{
    \includegraphics[width=0.18\textwidth]{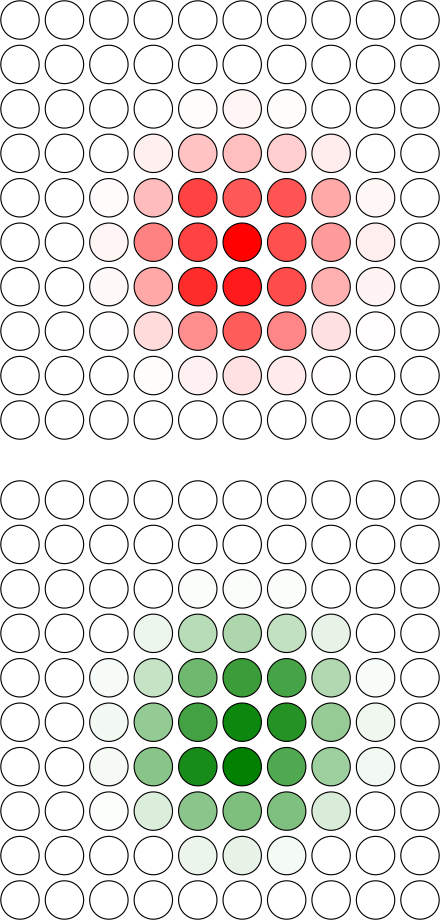}
    }
    \subfloat{
    \includegraphics[width=0.18\textwidth]{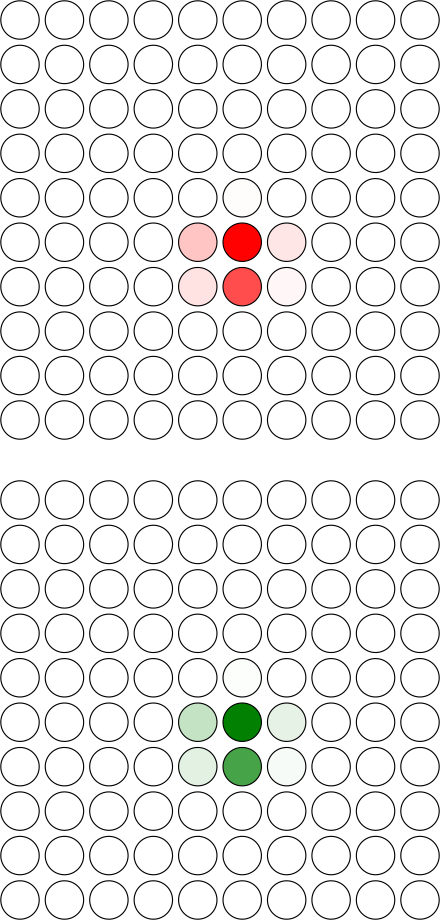}
    }
    \subfloat{
    \includegraphics[width=0.18\textwidth]{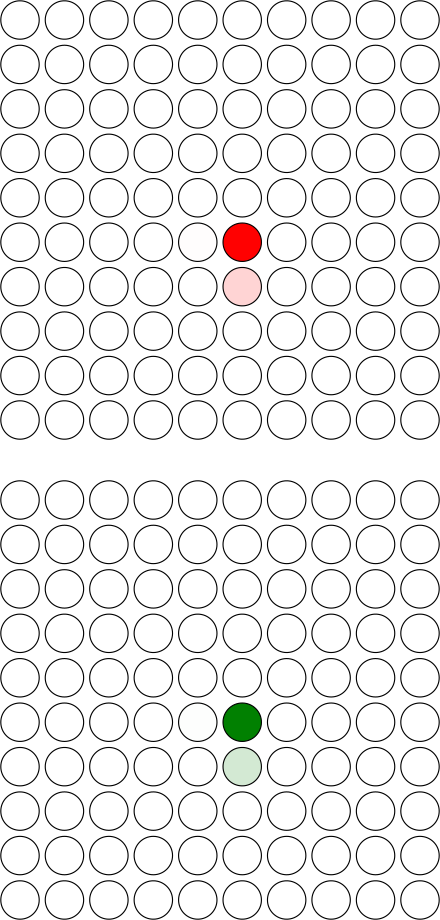}
    }
    \subfloat{
    \includegraphics[width=0.18\textwidth]{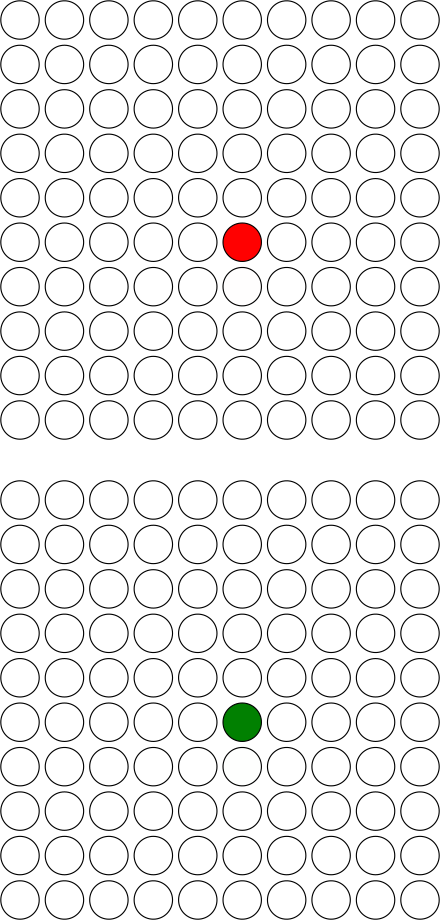}
    }\\
    $p=5$\vspace{-.2cm}\\
    \subfloat[$t=0$]{
    \includegraphics[width=0.18\textwidth]{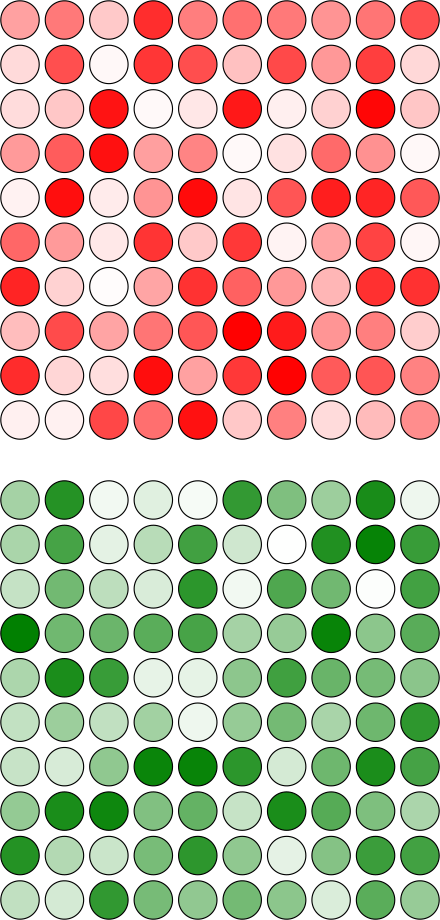}
    }
    \subfloat[$t=10$]{
    \includegraphics[width=0.18\textwidth]{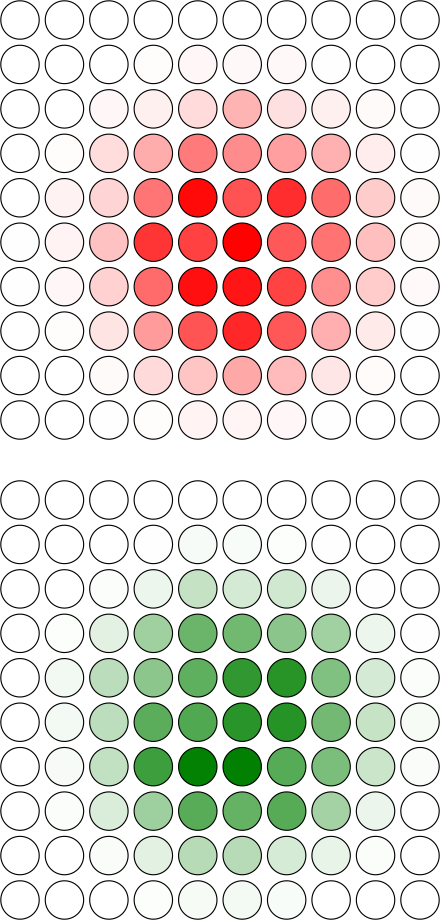}
    }
    \subfloat[$t=300$]{
    \includegraphics[width=0.18\textwidth]{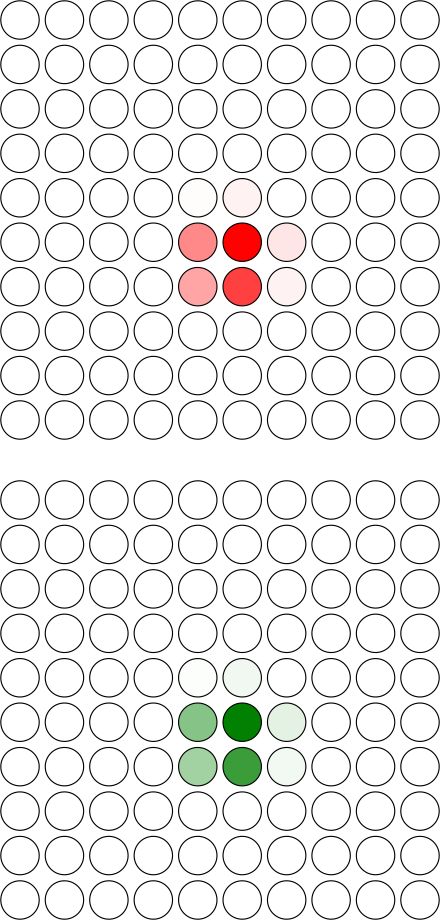}
    }
    \subfloat[$t=1000$]{
    \includegraphics[width=0.18\textwidth]{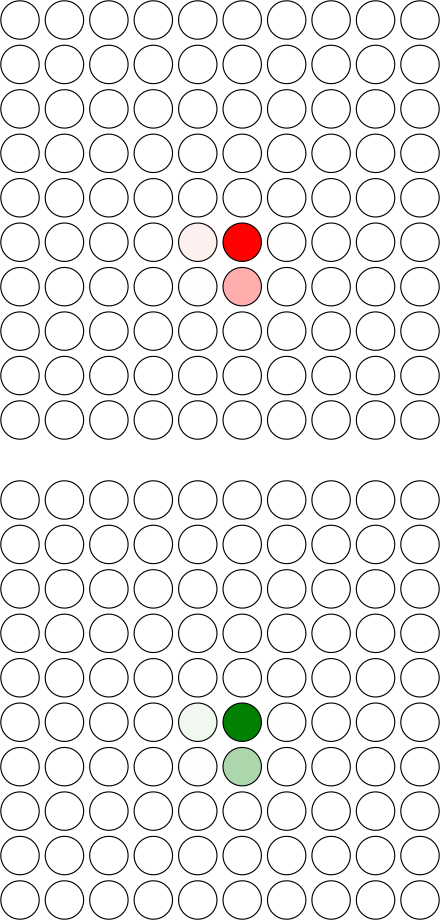}
    }
    \subfloat[$t=2000$]{
    \includegraphics[width=0.18\textwidth]{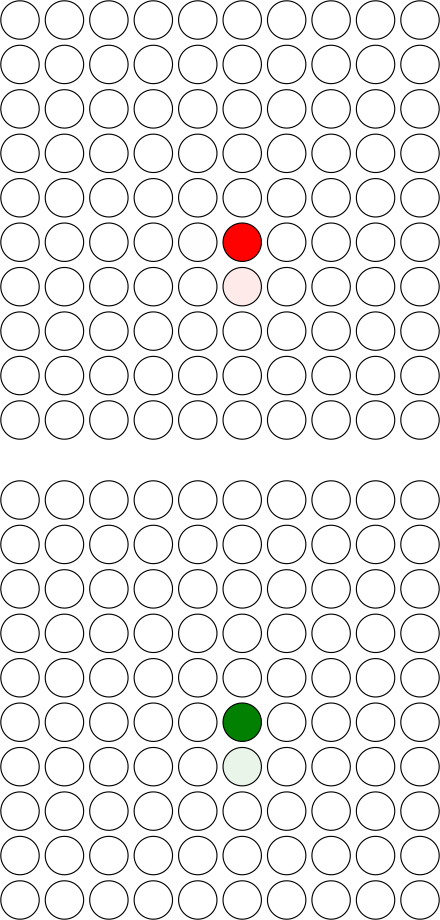}
    }    
    \caption{We display the evolution of both species (red and green, respectively) for attractive self- and cross-interactions, the mobility $m(r,s)=r$ and $p\in\{1.65,2,5\}$, in all cases starting from the same random initial data (left panel). For all $p$ we observe aggregation of both species to a single vertex as time progresses. Comparing the three values of $p$, we see that an increase of $p$ slows down aggregation speed.}
    \label{fig:linear_10by10}
\end{figure}

\begin{figure}[ht!]
    \captionsetup[subfloat]{labelformat=empty}
    \centering
    \subfloat[$t=0$]{
    \includegraphics[width=0.18\textwidth]{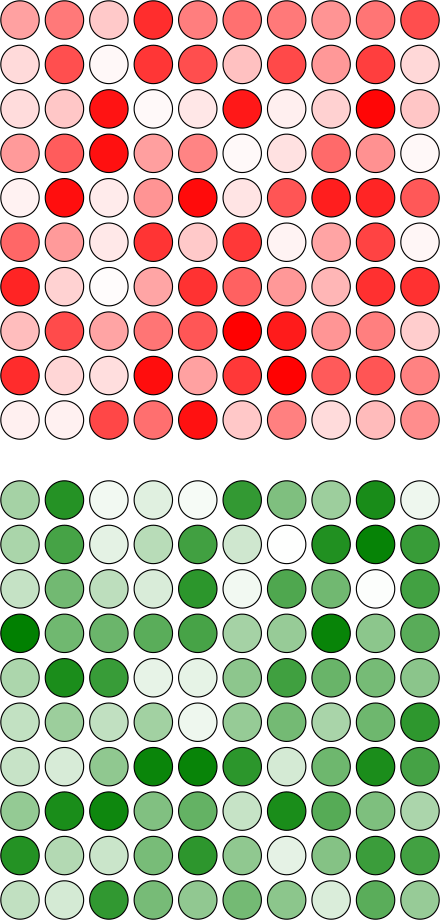}
    }
    \subfloat[$t=10$]{
    \includegraphics[width=0.18\textwidth]{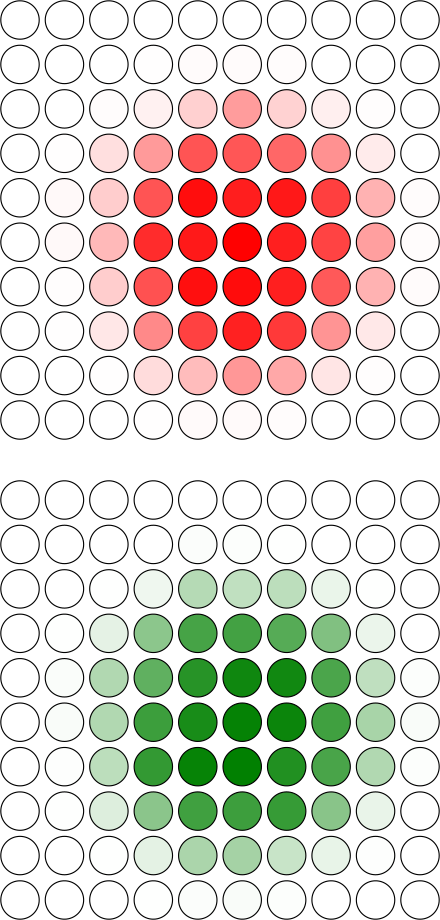}
    }
    \subfloat[$t=300$]{
    \includegraphics[width=0.18\textwidth]{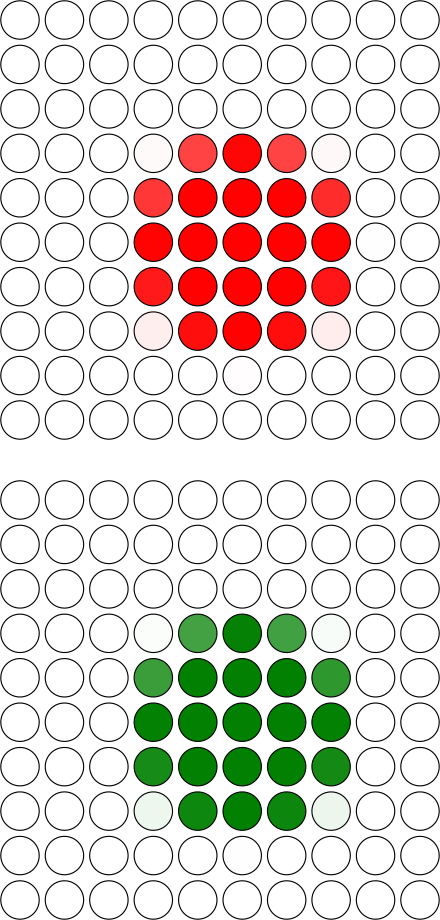}
    }
    \subfloat[$t=1000$]{
    \includegraphics[width=0.18\textwidth]{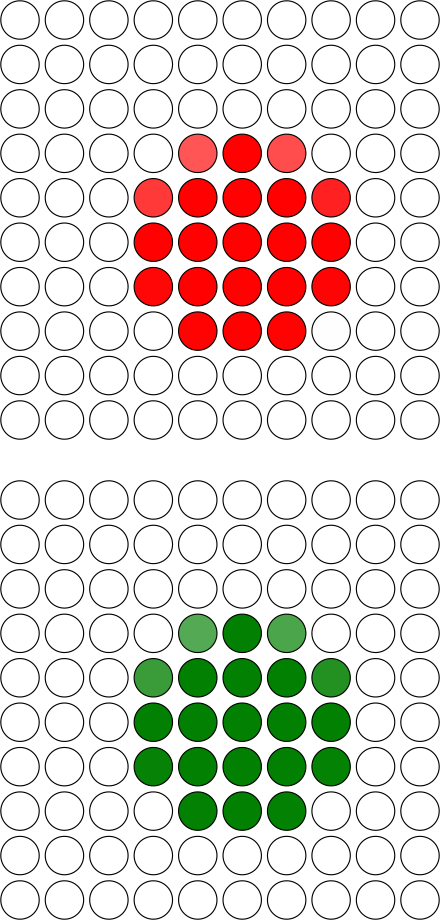}
    }
    \subfloat[$t=2000$]{
    \includegraphics[width=0.18\textwidth]{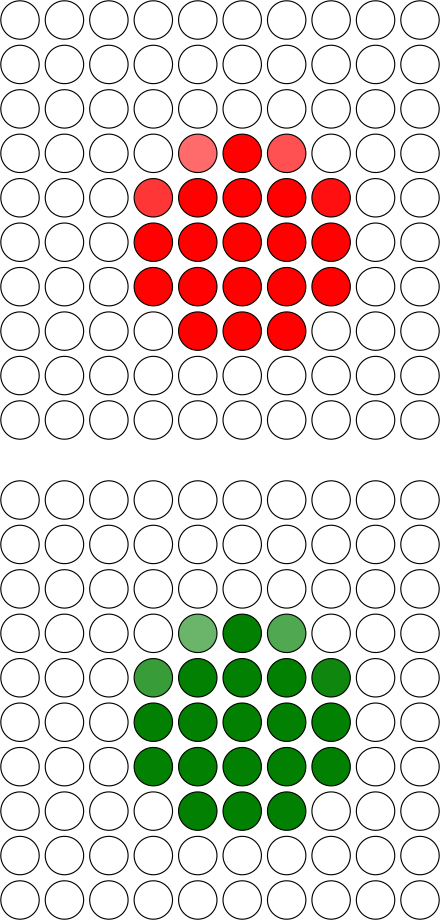}
    }    
    \caption{Starting from random initial data (left panel), we display the evolution of both species (red and green, respectively) for attractive self- and cross-interactions, the mobility $m(r,s)=r(1-s)$, $\mu_\iota=1/20$ for all $\iota\in\{1,\ldots,100\}$ and $p=2$. We observe that the aggregation is confined by the upper bound $\rho^{(i)}_\iota \le \mu_\iota$ for $i=1,2$, which results in the formation of a patch of populated vertices.}
    \label{fig:volume-filling_10by10_p=2}
\end{figure}

\section*{Acknowledgements} 
GH acknowledges support by the Studienstiftung des deutschen Volkes in form of a PhD stipend.
\appendix
\section{Appendix}\label{appendix}
We now provide proofs of the claims made in Subsection \ref{Stationary states on a two-point space for kernels with constant diagonals}. Most arguments rely on a study of the sign of the velocity field, which, in its essence, informs the direction in which mass is being transported. 

\noindent
\textbf{Decoupled case.}\\
The following simplification proves useful for the subsequent analysis. Using $\rho^{(i)}(x_\iota)=1-\rho^{(i)}(x_\kappa)$, we have
\baq\label{eq:v^(i) simplified two-point decoupled}
v^{(i)}_{\iota\kappa}&=\sum_{k=1}^2\sum_{\lambda=1}^{2}\left(K^{(ik)}_{\iota\lambda}-K^{(ik)}_{\kappa\lambda}\right)\rho_\lambda^{(k)}\mu_\lambda=\sum_{\lambda=1}^{2}\left(K^{(ii)}_{\iota\lambda}-K^{(ii)}_{\kappa\lambda}\right)\rho_\lambda^{(i)}\mu_\lambda\\
&=\left(K^{(ii)}_{\iota\iota}-K^{(ii)}_{\kappa\iota}\right)\rho_\iota^{(i)}\mu_\iota+\left(K^{(ii)}_{\iota\kappa}-K^{(ii)}_{\kappa\kappa}\right)\rho_\kappa^{(i)}\mu_\kappa\\
&=-D^{(ii)}\babla\rho^{(i)}_{\iota\kappa}=D^{(ii)}(2\rho^{(i)}(x_\iota)-1)=D^{(ii)}s^{(i)},
\eaq
where we set
\baq\label{eq_def:s^i}
s^{(i)}\coloneqq 2\rho^{(i)}(x_\iota)-1\in[-1,1].
\eaq

\begin{proof}[Proof of Proposition \ref{prop:2-point-energy_decoupled}]
We have seen that for $D^{(12)}=0$ there are two possible ways of obtaining a stationary state:
\begin{enumerate}[label=\alph*)]
\item\label{case decoupled steady not agg_appendix} $v^{(i)}_{\iota\kappa} = 0$,
\item\label{case decoupled steady agg_appendix} $v^{(i)}_{\iota\kappa} > 0$,  $\rho^{(i)}(x_\iota)=0$,
\end{enumerate}
where $\iota,\kappa\in\{1,2\}$ with $\iota\ne \kappa$. 

We see that \ref{case decoupled steady not agg_appendix} always holds if $D^{(ii)}=0$, since by \eqref{eq:v^(i) simplified two-point decoupled} we have $v^{(i)}_{\iota\kappa} = D^{(ii)}s^{(i)}$. Conversely, if $D^{(ii)}\ne0$, then $0 = D^{(ii)}(2\rho^{(i)}_\iota\mu_\iota-1)$ immediately implies $\rho^{(i)}(x_\iota)=\frac{1}{2}$. 

As for statement b), let us assume $D^{(ii)}<0$. 

Inserting $\rho^{(i)}(x_\iota)=1$ in to \eqref{eq:v^(i) simplified two-point decoupled}, we obtain $v^{(i)}_{\iota\kappa} = D^{(ii)}$. Hence, \ref{case decoupled steady agg_appendix} holds if only if $D^{(ii)}<0$. Here, it is important to keep in mind that $\iota=1$ and $\iota=2$ are both possible, leading to aggregation at either $x_1$ or $x_2$.
\end{proof}

\begin{proof}[Proof of Proposition \ref{prop:2-point-dynamics_decoupled}]
We have seen that for $D^{(ii)}=0$ every state $\rho^{(i)}$ is stationary, and that for $D^{(ii)}>0$ the only stationary state is $\rho_*^{(i)}=1/2$, which is therefore attractive, proving the first statement.

Concerning the second statement, let us consider $D^{(ii)}<0$. First, let $\epsilon=0$. Since $\rho_*^{(i)}=1/2$ is stationary, $\rho_t^{i}(x_\iota) = 1/2$, for $\iota =1,2$ and all $t \geq 0$. Next, consider $\varepsilon>0$. Since $v^{(i)}_{\iota\kappa}=D^{(ii)}s^{(i)}$ and $D^{(ii)}<0$, we immediately obtain $v^{(i)}_{\iota\kappa}s^{(i)}=D^{(ii)}\left(s^{(i)}\right)^2<0$ for all $s^{(i)}\ne 0$, i.e. by \eqref{eq:v^(i) simplified two-point decoupled} $v^{(i)}_{\iota\kappa}$ and $s^{(i)}$ must have opposite signs. Thus, $v^{(i)}_{\iota\kappa}<0$, which means mass flows from $x_\kappa$ to $x_\iota$ leading to an aggregation of mass on $x_\iota$. Therefore, $\rho_*^{(i)}=1/2$ is unstable and the stationary states $\rho_*^{(i)}\in\{0,1\}$ are asymptotically stable for any $\varepsilon < 1/2$.
\end{proof}

\noindent
\textbf{Coupled case.}
\begin{proof}[Proof of Proposition \ref{prop:2-point-energy_coupled}]
We recall that by Lemma \ref{lem:stationary state rho_nu^(i) disintegration}, there are four different conditions leading to stationary states:
\begin{enumerate}[label=\alph*)]
\item\label{case steady 0 agg_appendix} $v^{(1)}_{\iota\kappa} = v^{(2)}_{\iota\kappa}=0$,
\item\label{case steady 1 agg_appendix} $v^{(i)}_{\iota\kappa} > 0$ and $\rho^{(i)}(x_\iota)=0$ as well as $v^{(k)}_{\iota\kappa}= 0$,
\item\label{case steady 2 agg 1 node_appendix} $v^{(1)}_{\iota\kappa}>0$ and $v^{(2)}_{\iota\kappa}> 0$  as well as $\rho^{(1)}(x_\iota)=\rho^{(2)}(x_\iota)=0$,
\item\label{case steady 2 agg 2 nodes_appendix} $v^{(1)}_{\iota\kappa}> 0$ and $v^{(2)}_{\iota\kappa}< 0$  as well as $\rho^{(1)}(x_\iota)=0$ and $\rho^{(2)}(x_\iota)=1$,
\end{enumerate}
for $i,k,\iota,\kappa\in\{1,2\}$ with $i\ne k$ and $\iota\ne \kappa$, which we  assume for the remainder of this section. Similar to \eqref{eq:v^(i) simplified two-point decoupled}, we can rewrite $v^{(i)}$ as
\baq\label{eq:T^i kurz}
v^{(i)}_{\iota\kappa} &= D^{(ii)}(2\rho^{(i)}(x_\iota)-1)+D^{(12)}(2\rho^{(k)}(x_\iota)-1) = D^{(ii)}s^{(i)}+D^{(12)}s^{(k)},
\eaq
where $s^{(i)}$ and $s^{(k)}$ are given by \eqref{eq_def:s^i}.
\subparagraph{\ref{case steady 0 agg_appendix}} Using $D^{(12)}\ne0$, we obtain 
\baq\label{eq:s^i,s^k 0 agg gl 1}
s^{(1)}=-\frac{D^{(22)}}{D^{(12)}}s^{(2)},\qquad
s^{(2)}=-\frac{D^{(11)}}{D^{(12)}}s^{(1)}.
\eaq
Hence, if $D^{(11)} = 0$ or $D^{(22)} = 0$, then we immediately find $s^{(1)} = s^{(2)} = 0$, which corresponds to $\rho^{(1)}_\iota\mu_\iota=\rho^{(2)}_\iota\mu_\iota=\frac{1}{2}$. If both $D^{(11)}\ne0$ and $D^{(22)}\ne0$, we additionally have
\baq\label{eq:s^i,s^k 0 agg gl 2}
s^{(1)}=-\frac{D^{(12)}}{D^{(11)}}s^{(2)},\qquad s^{(2)}=-\frac{D^{(12)}}{D^{(22)}}s^{(1)},
\eaq
which we can combine with \eqref{eq:s^i,s^k 0 agg gl 1} to obtain
\baqs
D^{(11)}D^{(22)}s^{(i)} = \left(D^{(12)}\right)^2s^{(i)},
\eaqs
for $i=1,2$. Therefore, if $D^{(11)}D^{(22)}\ne\left(D^{(12)}\right)^2$ we also find $s^{(1)} = s^{(2)} = 0$. Otherwise all equations in \eqref{eq:s^i,s^k 0 agg gl 1} and \eqref{eq:s^i,s^k 0 agg gl 2} are equivalent and define a family of stationary states as claimed.

\subparagraph{\ref{case steady 1 agg_appendix}} Here, if $D^{(kk)} = 0$, then, by \eqref{eq:T^i kurz} $v^{(k)}=0 \iff D^{(12)}=0$. Hence, we require $D^{(kk)} \ne 0$. Then, $v^{(k)}=0$ gives 
\baq\label{eq:1 agg s^i}
s^{(k)}=-\frac{D^{(12)}}{D^{(kk)}}.
\eaq
The fact that $s^{(k)}\in[-1,1]$  implies
\baq\label{eq:1 agg existenz 1 a}
\abs{D^{(12)}}\le \abs{D^{(kk)}}.
\eaq
This is a necessary condition for the stationary state to exist. The inequality $v^{(i)}<0$ gives us the additional condition
\baq\label{eq:1 agg existenz 2 a}
D^{(ii)}< -D^{(12)}s^{(k)} =\frac{\left(D^{(12)}\right)^2}{D^{(kk)}}.
\eaq
It is important to keep in mind that we can exchange $i$ and $k$, leading to an exchange of $i$ and $k$ in \eqref{eq:1 agg s^i}, \eqref{eq:1 agg existenz 1 a} and \eqref{eq:1 agg existenz 2 a}. Also, we can exchange $\iota$ and $\kappa$, which corresponds to exchanging $s^{(i)}$ for $-s^{(i)}$, $s^{(k)}$ for $-s^{(k)}$, $v^{(i)}$ for $-v^{(i)}$ and $v^{(k)}$ for $-v^{(k)}$, but does not affect any of the conditions \eqref{eq:1 agg s^i}, \eqref{eq:1 agg existenz 1 a} and \eqref{eq:1 agg existenz 2 a}. Hence, for any given combination of interaction kernels, there exist $0$, $2$ or $4$ stationary states of the form \ref{case steady 1 agg_appendix}.

\subparagraph{\ref{case steady 2 agg 1 node_appendix} and \ref{case steady 2 agg 2 nodes_appendix}}
Combining \ref{case steady 2 agg 1 node_appendix}, \eqref{eq_def:s^i} and \eqref{eq:T^i kurz} immediately gives us the existence conditions $D^{(ii)} < -D^{(12)}$, $i=1,2$. Similarly, combining \ref{case steady 2 agg 2 nodes_appendix}, \eqref{eq_def:s^i} and \eqref{eq:T^i kurz} yields the existence conditions $D^{(ii)} < D^{(12)}$, $i=1,2$.
\end{proof}
\begin{proof}[Proof of Corollary \ref{cor:only_a<=>pos_semidef}]
Each of the conditions \eqref{eq:a_only} and \eqref{eq:a,r_only} immediately implies that \eqref{eq:1 agg existenz 2} is violated no matter the choice of $i$ and $k$. Additionally, it implies $D^{(ii)} \geq \abs{D^{(12)}}\geq \pm D^{(12)}$ for at least one $i\in\{1,2\}$ and thus also the violation of both \eqref{eq:2 agg 1 node existenz} and \eqref{eq:2 agg 2 nodes existenz}. Conversely, if both \eqref{eq:2 agg 1 node existenz} and \eqref{eq:2 agg 2 nodes existenz} are violated, then $D^{(kk)} \geq \abs{D^{(12)}} > 0$ must hold for at least one $k\in\{1,2\}$, which implies \eqref{eq:1 agg existenz 1}. Therefore, \eqref{eq:1 agg existenz 2} must not be satisfied, i.e., we must have $D^{(11)}D^{(22)}\geq \left(D^{(12)}\right)^2$ and consequently $D^{(ii)}>0$ for $i\ne k$.
\end{proof}
\begin{proof}[Proof of Corollary \ref{cor:c+d<=>b^2}]
For the first claim observe that \eqref{eq:2 agg 1 node existenz} together with \eqref{eq:2 agg 2 nodes existenz} imply
\baqs
D^{(11)}&< -\abs{D^{(12)}},\quad D^{(22)}< -\abs{D^{(12)}}.
\eaqs
This implies \eqref{eq:1 agg existenz 1} as well as
\baqs
D^{(11)}D^{(22)}>\left(D^{(12)}\right)^2,\quad D^{(11)}<0,\quad D^{(22)}<0,
\eaqs
which yields \eqref{eq:1 agg existenz 2} for both $k=1$ and $k=2$.

For the converse case we have
\baq\label{eq:c+d<=>b^2_aux2}
\abs{D^{(11)}}\geq \abs{D^{(12)}},\quad \abs{D^{(22)}}\geq \abs{D^{(12)}},\quad D^{(11)}<\frac{\left(D^{(12)}\right)^2}{D^{(22)}}, \quad D^{(22)}<\frac{\left(D^{(12)}\right)^2}{D^{(11)}}.
\eaq
From this we immediately see that $\sgn\left(D^{(11)}\right)=\sgn\left(D^{(22)}\right)$ as otherwise either the third or the fourth inequality in \eqref{eq:c+d<=>b^2_aux2} is violated. Additionally taking into account the first and second inequality in \eqref{eq:c+d<=>b^2_aux2}, we find that $\sgn\left(D^{(11)}\right)=\sgn\left(D^{(22)}\right) = -1$ as otherwise we would have a contradiction. This gives us
\baqs
D^{(11)}&\le -\abs{D^{(12)}},\quad D^{(22)}\le -\abs{D^{(12)}},
\eaqs
which, when combined with \eqref{eq:c+d<=>b^2_aux2}, shows the claim.
\end{proof}
\begin{lemma}\label{lem:2-point-energy-hierarchy}
Let $N=2$, let \eqref{eq:K^ik const diag} and \eqref{eq:m=0_iff_r=0} hold, and assume $D^{(12)}_{12}\ne 0$. Then, for the stationary states given in Proposition \ref{prop:2-point-energy_coupled}, we have
\begin{enumerate}[label=\arabic*.]
\item If $D^{(11)}D^{(22)}=\left(D^{(12)}\right)^2$, then $\mathcal{E}(\rhoup_{\mathrm{a}_r}) = \mathcal{E}(\rhoup_{\mathrm{a}})$ for $r\in\left[-1\lor -\abs{\frac{D^{(22)}}{D^{(12)}}},1\land \abs{\frac{D^{(22)}}{D^{(12)}}}\right]$.
\item If for $i\in\{1,2\}$ $\dot\rhoup_{\mathrm{b}_i}=0$ for $i\in\{1,2\}$, then $\mathcal{E}(\rhoup_{\mathrm{b}_i}) < \mathcal{E}(\rhoup_\mathrm{a})$.
\item If $\rhoup_\mathrm{c}$ is a stationary state, then $\mathcal{E}(\rhoup_\mathrm{c}) < \mathcal{E}(\rhoup_\mathrm{a})$.
\item If $\rhoup_\mathrm{d}$ is a stationary state, then $\mathcal{E}(\rhoup_\mathrm{d}) < \mathcal{E}(\rhoup_\mathrm{a})$.
\item If $\dot\rhoup_{\mathrm{b}_i}=\rhoup_\alpha=0$ for $\alpha\in\{\mathrm{c},\mathrm{d}\}$ and $i\in\{1,2\}$, then $\mathcal{E}(\rhoup_\alpha) < \mathcal{E}(\rhoup_{\mathrm{b}_i})$.
\item If $\rhoup_\mathrm{c}$ and $\rhoup_\mathrm{d}$ are stationary states, then we have $\sgn(\mathcal{E}(\rhoup_\mathrm{d})-\mathcal{E}(\rhoup_\mathrm{c})) = -\sgn\left(D^{(12)}\right)$.
\end{enumerate}
\end{lemma}
\begin{proof}
In Theorem \ref{thm:aggregation} we have identified conditions implying that aggregation of one or both species is energetically optimal. Now we want to specify these findings in the 2-point setting. To this end we calculate the signed differences between the energies of the explicit stationary states derived in Proposition \ref{prop:2-point-energy_coupled}. Setting $\dd_{\alpha_1 \alpha_2}\coloneqq\mathcal{E}(\rhoup_{\alpha_2}) - \mathcal{E}(\rhoup_{\alpha_1})$, for different combinations of $\alpha_1, \alpha_2 \in \{$a, a$_r$, b$_1$, b$_2$, c, d$\}$. A short computation yields
\begin{subequations}
\label{eq:Energievergleich two-point stationary states}
\baqs
&\dd_\text{aa$_r$}=\frac{1}{4}\left(D^{(11)}r^2-\frac{\left(D^{(12)}\right)^2}{D^{(22)}}r^2\right),\;\; &&\dd_\text{ab$_i$}=\frac{1}{4}\left(D^{(ii)}-\frac{\left(D^{(12)}\right)^2}{D^{(kk)}}\right),\\
&\dd_\text{ac}=\frac{1}{4}\left(D^{(11)}+2D^{(12)}+D^{(22)}\right),\;\;
&&\dd_\text{ad}=\frac{1}{4}\left(D^{(11)}-2D^{(12)}+D^{(22)}\right),\\
&\dd_\text{b$_i$c}=\frac{1}{4D^{(kk)}}\left(D^{(kk)}+D^{(12)}\right)^2,\;\;
&&\dd_\text{b$_i$d}=\frac{1}{4D^{(kk)}}\left(D^{(kk)}-D^{(12)}\right)^2, \;\;
&&&\dd_\text{cd}=-D^{(12)}.
\eaqs
\end{subequations}
With this, $D^{(11)}D^{(22)}=\left(D^{(12)}\right)^2$ immediately yields $\dd_\text{aa$_r$}=0$.

Applying the conditions which imply that $\rhoup_{\mathrm{b}_i}$, $\rhoup_\mathrm{c}$ and $\rhoup_\mathrm{d}$ define stationary states to \eqref{eq:Energievergleich two-point stationary states}, we make the following observations:
\begin{itemize}
    \item We have seen that $\rhoup_{\mathrm{b}_i}$ is a stationary state if both \eqref{eq:1 agg existenz 1} and \eqref{eq:1 agg existenz 2} hold. However, \eqref{eq:1 agg existenz 2} immediately yields $\dd_{\text{ab}_i}< 0$. Hence, if there is a stationary state of the form $\rhoup_{\mathrm{b}_i}$, then it admits a lower energy than $\rhoup_\mathrm{a}$.
    \item We have found that $\rhoup_\mathrm{c}$ defines two stationary states if \eqref{eq:2 agg 1 node existenz} holds, which immediately implies that $\dd_\text{ac}< 0$. Similarly, we find by \eqref{eq:2 agg 2 nodes existenz} that $\dd_\text{ad}< 0$ if $\rhoup_\mathrm{d}$ defines a stationary state. Thus, $\rhoup_\mathrm{c}$ and $\rhoup_\mathrm{d}$ respectively admit lower energies than $\rhoup_\mathrm{a}$, if they are stationary states. 
    \item If \eqref{eq:1 agg existenz 1} and \eqref{eq:1 agg existenz 2} as well as \eqref{eq:2 agg 1 node existenz} hold, then we immediately obtain that $D^{(kk)}< 0$. Analogously, \eqref{eq:1 agg existenz 1} and \eqref{eq:1 agg existenz 2} together with \eqref{eq:2 agg 2 nodes existenz} also imply $D^{(kk)}< 0$. Thus, if $\rhoup_{\mathrm{b}_i}$ as well as either $\rhoup_\mathrm{c}$ or $\rhoup_\mathrm{d}$ are stationary states, then aggregation of both species admits a lower energy than aggregation of only one species.
    \item Finally, if \eqref{eq:2 agg 1 node existenz} and \eqref{eq:2 agg 2 nodes existenz} are simultaneously satisfied,  $\rhoup_\mathrm{c}$ admits a lower energy than $\rhoup_\mathrm{d}$ if we have cross-attraction, while $\rhoup_\mathrm{d}$ admits a lower energy than $\rhoup_\mathrm{c}$ if we have cross-repulsion. Further, \eqref{eq:2 agg 1 node existenz} and \eqref{eq:2 agg 2 nodes existenz} in conjunction imply \eqref{eq:1 agg existenz 1} and \eqref{eq:1 agg existenz 2} for both $i=1,k=2$ and $i=2,k=1$. In this situation all candidates are stationary states.
\end{itemize}
\end{proof}
\begin{remark}[Stability of the energy minimizer]\label{rem:stability_lowest}
If $D^{(11)}D^{(22)}\ne \left(D^{(12)}\right)^2$ and $\rhoup_\mathrm{a}$ is not the unique stationary state, then there exist exactly two ($D^{(12)}\ne 0$) or exactly four ($D^{(12)}= 0$) points minimizing the energy cf. Corollary \ref{cor:c+d<=>b^2} and Lemma \ref{lem:2-point-energy-hierarchy}. Since these are distinct from each other and since the energy can not increase along trajectories of the gradient flow by \eqref{eq:energy time derivative}, the two are asymptotically stable. 
\end{remark}

\begin{remark}[Stability of $\rhoup_\mathrm{a}$]\label{rem:stability_a}
Employing the assumption that all kernels are constant on the diagonal, we can rewrite the energy as
\baq\label{eq:2-point-energy:D-form}
\mathcal{E}(\rhoup)
&=D^{(11)}\left(\rho^{(1)}(x_1)-\frac{1}{2}\right)^2+D^{(22)}\left(\rho^{(2)}(x_1)-\frac{1}{2}\right)^2\\
&+ 2D^{(12)}\left(\rho^{(1)}(x_1)-\frac{1}{2}\right)\left(\rho^{(2)}(x_1)-\frac{1}{2}\right)+\sum_{i,k=1}^2\frac{1}{4}\left(K^{(ik)}_{11}+K^{(ik)}_{12}\right),
\eaq
which shows that the level sets of the energy are quadrics centered around $(1/2,1/2)$ and lets us identify $\rho_\mathrm{a}=(1/2,1/2)$ as a critical point. Moreover, it allows for a quick way to verify that the hessian of $\mathcal{E}$ reads
\baqs
H_{\mathcal{E}}=2\begin{pmatrix} D^{(11)} & D^{(12)}\\ 
D^{(12)} & D^{(22)} \end{pmatrix},
\eaqs
which is positive definite if and only if \eqref{eq:a_only} holds. Hence, $\rhoup_\mathrm{a}$ is a minimizer of the energy if and only if we have \eqref{eq:a_only}. By Corollary \ref{cor:only_a<=>pos_semidef} this is equivalent to $\rhoup_\mathrm{a}$ being the only stationary state of the dynamics on the two-point space. 

This also means that $\rhoup_\mathrm{a}$ is unstable, whenever $D^{(11)}D^{(22)}\ne \left(D^{(12)}\right)^2$ and a second stationary state other than $\rhoup_\mathrm{a}$ exists. What is more, the proof of Corollary \ref{cor:c+d<=>b^2} in conjunction with \eqref{eq:2-point-energy:D-form} show that $\rhoup_\mathrm{a}$ maximizes the energy if we have $\dot\rhoup_\mathrm{c}=\dot\rhoup_\mathrm{d}=0$.
\end{remark}

\begin{remark}[Stability of $\rhoup_\mathrm{c}$ and  $\rhoup_\mathrm{d}$]\label{rem:stability_c,d}
If $\dot\rhoup_\mathrm{d}\ne\dot\rhoup_\mathrm{c}=0$, we have seen that $\rhoup_\mathrm{c}$ minimizes the energy and is therefore asymptotically stable by the Remark \ref{rem:stability_lowest}. Analogously, $\rhoup_\mathrm{d}$ is asymptotically stable whenever $\dot\rhoup_\mathrm{c}\ne\dot\rhoup_\mathrm{d}=0$. 

Let us now consider the case $\dot\rhoup_\mathrm{c}=\dot\rhoup_\mathrm{d}=0$. From Corollary  \ref{cor:c+d<=>b^2} we know that in this case we also have $\dot\rhoup_{\mathrm{b}_1}=\dot\rhoup_{\mathrm{b}_2}=0$. By the proof of Corollary \ref{cor:c+d<=>b^2} we have $D^{(11)}<0$, $D^{(22)}<0$ and $\abs{D^{(12)}}<\sqrt{D^{(11)}D^{(22)}}$. Therefore, equation \eqref{eq:2-point-energy:D-form} tells us that the quadrics describing the level sets of $\mathcal{E}$ are ellipses and that the energy is maximized at $(1/2,1/2)$. Since every stationary state other than $\rhoup_\mathrm{a}$ lies on the border of the feasible set, the energy level sets of these stationary states touch the border of the feasible set. In particular there is at least one stationary state of the form $\rhoup_{\mathrm{b}_i}, i\in\{1,2\}$ on each edge of the feasible set. This means that the set $\{\rho^{(1)}(x_1),\rho^{(2)}(x_1)\in[0,1]\times[0,1]:\mathcal{E}(\rho)<\mathcal{E}(\rhoup_{\mathrm{b}_1})\land\mathcal{E}(\rhoup_{\mathrm{b}_2})\}$ consists of four different connected components, each located in one of the corners of the feasible set (see for example Figure \ref{fig:two-point_self:att_vary:cross}). Since by \eqref{eq:energy time derivative} the energy along the trajectories of the gradient flow cannot increase and the trajectories are continuous, a trajectory starting in one of these components can never leave it and must therefore converge to the stationary state located in the corner of the feasible set, which is contained in the component. Hence, whenever $\rhoup_\mathrm{c}$ or $\rhoup_\mathrm{d}$ are stationary states, they are asymptotically stable.
\end{remark}

\begin{remark}[Stability of $\rhoup_\text{b$_i$}$]\label{rem:stability_b}
From Remark \ref{rem:stability_lowest} and Lemma \ref{lem:2-point-energy-hierarchy} we know that for $i\in\{1,2\}$ the stationary states of the form $\rhoup_\text{b$_i$}$ are asymptotically stable, if the only other existing stationary state is $\rhoup_\mathrm{a}$.

From Corollary \ref{cor:c+d<=>b^2} we know that  $\dot\rhoup_\mathrm{c}=0$ or $\dot\rhoup_\mathrm{d}=0$, if $\dot\rhoup_\text{b$_1$}=\dot\rhoup_\text{b$_2$}=0$. Thus, we now consider the case, where $\dot\rhoup_\text{b$_i$}=0$ for $i\in\{1,2\}$ and either $\dot\rhoup_\mathrm{c}=0$ or $\dot\rhoup_\mathrm{d}=0$. For $k\in\{1,2\}$, $k\ne i$ as before, \eqref{eq:1 agg existenz 1} in conjunction with \eqref{eq:2 agg 1 node existenz} or \eqref{eq:2 agg 2 nodes existenz} implies that we have $D^{(kk)}<\pm D^{(12)}\le\abs{D^{(12)}} \le\abs{D^{(kk)}}$, i.e., $D^{(kk)}<0$. We introduce a small perturbation $\rho^{(i)}=\rho^{(i)}_\text{b$_i$}$, $\rho^{(k)}(x_\iota)=\rho^{(k)}_\text{b$_i$}(x_\iota)+\varepsilon$, $\rho^{(k)}(x_\kappa)=\rho^{(k)}_\text{b$_i$}(x_\kappa)-\varepsilon$, for $\abs{\varepsilon}$ sufficiently small. Upon comparing its energy with that of $\rho_\text{b$_i$}$, we observe
\baqs
\mathcal{E}(\rhoup)-\mathcal{E}(\rhoup_{\mathrm{b}_i})&=D^{(kk)}\varepsilon^2 < 0.
\eaqs
By \eqref{eq:energy time derivative} the energy cannot increase along the gradient flow and the perturbed distribution converges to a different stationary state implying the instability of $\rho_\text{b$_i$}$.
\end{remark}
%\interlinepenalty 10000\relax
%\bibliographystyle{plain}
%\renewcommand{\refname}{Bibliography}
%\bibliography{bibliography}

\end{document}